\documentclass[10pt,a4paper,reqno]{amsart}
\usepackage{amsmath}
\usepackage{amssymb}
\usepackage{amsfonts}
\usepackage{amsthm}
\usepackage{mathrsfs}
\usepackage{dsfont}
\usepackage{mathtools}
\usepackage{bm}
\newcommand*{\mailto}[1]{\href{mailto:#1}{\nolinkurl{#1}}}
\usepackage{todonotes}

\usepackage{color}
\usepackage[pagebackref=true,colorlinks=false]{hyperref}

\definecolor{darkgreen}{rgb}{0.5,0.25,0}
\definecolor{darkblue}{rgb}{0,0,1}
\definecolor{answerblue}{rgb}{0,0,0.75}

\hypersetup{colorlinks,breaklinks,
            linkcolor=darkblue,urlcolor=darkblue,
            anchorcolor=darkblue,citecolor=darkblue}

\newcommand{\ep}{\varepsilon}
\newcommand{\eps}{\varepsilon}
\newcommand{\pd}{\partial}
\newcommand{\LL}{\langle}
\newcommand{\RR}{\rangle}

\newcommand{\Ex}{\mathbb{E}}

\renewcommand{\d}{\mathrm{d}}

\newcommand{\KK}{\bm{\Sigma}}
\newcommand{\JJ}{{\bf j}}

\newcommand{\doublehookrightarrow}{
    \lhook\joinrel\relbar\mspace{-12mu}\hookrightarrow
}
\newcommand{\R}{\mathbb{R}}
\newcommand{\T}{{\mathbb{S}^1}}
\newcommand{\N}{\mathbb{N}}
\newcommand{\Z}{\mathbb{Z}}
\newcommand{\prob}{\mathbb{P}}

\newcommand{\abs}[1]{\left | #1 \right |}
\newcommand{\norm}[1]{\left \| #1 \right \|}
\newcommand{\bk}[1]{ \left(  #1 \right)}

\newcommand{\seq}[1]{\left\{ #1 \right\}}

\newcommand{\dott}{\, \cdot\,}

\newcommand{\ton}{\overset{n\uparrow \infty}{\longrightarrow}}
\newcommand{\toR}{\overset{R\uparrow \infty}{\longrightarrow}}
\newcommand{\todelta}{\overset{\delta\downarrow 0}{\longrightarrow}}

\newtheorem{thm}{Theorem}[section]
\newtheorem{prop}[thm]{Proposition}
\newtheorem{lem}[thm]{Lemma}

\theoremstyle{theorem}
\newtheorem{defin}{Definition}[section]
\theoremstyle{remark}
\newtheorem{rem}[thm]{Remark}

\numberwithin{equation}{section}

\allowdisplaybreaks

\title[Viscous stochastic CH equation]
{Global well-posedness of the viscous 
\\ Camassa--Holm equation with gradient noise}

\author[H. Holden]{Helge Holden}
\address[Helge Holden]{Department 
of Mathematical Sciences\\
NTNU Norwegian University of Science and Technology\\
NO-7491 Trondheim\\ Norway}
\email{\mailto{helge.holden@ntnu.no}}
\urladdr{\url{https://www.ntnu.edu/employees/holden}}

\author[K. H. Karlsen]{Kenneth H. Karlsen}
\address[Kenneth H. Karlsen]{Department of Mathematics\\
University of Oslo\\
NO-0316 Oslo\\ Norway}
\email{\mailto{kennethk@math.uio.no}}

\author[P.H.C. Pang]{Peter H.C. Pang}
\address[Peter H.C. Pang]{Department 
of Mathematical Sciences\\
NTNU Norwegian University of Science and Technology\\
NO-7491 Trondheim\\ Norway}
\email{\mailto{ptr@math.uio.no}}



\thanks{This research was partially supported 
by the Research Council of Norway Toppforsk 
project {\em Waves and Nonlinear Phenomena} (250070). 
The final author is also partially supported by the 
Research Council of Norway project {\em INICE} (301538).}

\date{\today}

\begin{document}

\begin{abstract}
We analyse a nonlinear stochastic partial 
differential equation that corresponds to 
a viscous shallow water equation (of 
the Camassa--Holm type) perturbed by a convective, 
position-dependent noise term. We establish the existence 
of weak solutions in $H^m$ ($m\in\mathbb{N}$) using 
Galerkin approximations and the 
stochastic compactness method. We derive a series 
of a priori estimates that combine a model-specific 
energy law with non-standard regularity estimates. 
We make systematic use of a stochastic 
Gronwall inequality and also stopping time techniques. 
The proof of convergence to a solution argues 
via tightness of the laws of the Galerkin solutions, and 
Skorokhod--Jakubowski a.s.~representations 
of random variables in quasi-Polish spaces.  
The spatially dependent noise function 
constitutes a complication throughout the analysis, 
repeatedly giving rise to nonlinear terms 
that ``balance" the martingale 
part of the equation against the second-order 
Stratonovich-to-It\^{o} correction term. 
Finally, via pathwise uniqueness, we 
conclude that the constructed solutions are probabilistically 
strong. The uniqueness proof is based on a 
finite-dimensional It\^{o} formula 
and a DiPerna--Lions type regularisation procedure, where 
the regularisation errors are controlled by 
first and second order commutators.
\end{abstract}

\maketitle

\setcounter{tocdepth}{1}

\tableofcontents

\section{Introduction and main results}

\subsection{Background}

We are interested in the initial-value problem 
for the stochastic parabolic-elliptic system 
\begin{equation}\label{eq:u_ch_ep}
	\begin{aligned}
		 0&=\d {u}+ \left[{u}\, \pd_x {u} 
		+ \pd_x P-\ep \pd_x^2 {u} \right] \;\d t  \\
		& \quad- \frac12 \sigma(x) 
		\pd_x \bk{\sigma(x) \pd_x {u}}\,\d t
		+\sigma(x) \pd_x {u} \, \d W ,
		\\  -\pd_x^2 P+P &= u^2 
		+\frac{1}{2} \left(\pd_x{u}\right)^2,
		\quad \text{for $(t,x)\in (0,T)\times \T$}, 
	\end{aligned}
\end{equation}
where $\T=\R/\Z$ is the 1D torus (circle), 
$\eps$ and $T$ are positive numbers, $\sigma=\sigma(x)\in 
W^{2,\infty}(\T)$ is a position-dependent noise amplitude, 
and $W$ is a 1D Brownian motion 
defined on a probability space and  adapted to 
some filtration (further details will be given 
later). Formally, by the It\^{o}--Stratonovich 
conversion formula, the two $\sigma$ terms 
in \eqref{eq:u_ch_ep} can be combined into the simple looking 
Stratonovich differential $\sigma\,\pd_x u\circ \d W$, 
which in the literature is referred to as a gradient, 
transport or convection noise term. The 
elliptic equation for $P$ can be ``solved" to give
\begin{equation}\label{eq:P-def}
	P = P[{u}]:=K*\left(u^2 
	+\frac{1}{2} \left(\pd_x{u}\right)^2\right),
\end{equation}
where $K$ denotes the Green's function of the operator
$1-\pd_x^2$ on $\T$, which can be 
given in explicit form, and $*$ denotes convolution 
in $x$. Consequently, \eqref{eq:u_ch_ep} takes the form of a 
nonlinear, nonlocal stochastic partial 
differential equation (SPDE). 

If $\ep=0$ and $\sigma\equiv 0$, then \eqref{eq:u_ch_ep} 
becomes the classical (deterministic) Camassa--Holm (CH) 
equation \cite{Camassa:1993zr,Fuchssteiner:1981fk}, which 
is a nonlinear dispersive PDE that models 
shallow water waves. Besides, it is nonlocal, 
completely integrable and may be written in 
(bi-)Hamiltonian form in terms of the 
so-called momentum variable $m:=\bigl(1-\pd_{xx}^2\bigr)u$. 
The inclusion of gradient type noise is natural in that the perturbation 
can be thought of as one on the transporting velocity field, 
i.e., $u \,\pd_x u$ is replaced by $(u+\sigma \circ \dot{W})\circ \pd_x u$, 
and hence as an additive perturbation of the underlying 
Lagrangian dynamics; see Section \ref{eq:derivation_SPDE} for more details.

The CH equation is a supercritical PDE in 
the sense that the competition between dispersion, 
which tends to spread out a wave, and nonlinearity, which 
causes a wave to concentrate, leads to 
the development of singularities in finite 
time (wave breaking). The well-posedness 
of the CH wave equation,  in different classes 
of weak solutions for general finite-energy initial data 
$u|_{t=0}=u_0\in H^1$, has been widely studied, see 
for example \cite{Bressan:2007fk,Bressan:2006nx,
Coclite:2005tq,HR2007,XZ2000} (and the references therein). 
The relevance of the Sobolev space $H^1$ 
is that its norm is preserved (up to an inequality) 
by the solution operator, and $H^1$ 
regularity is needed to make distributional 
sense to the equation. This space is 
consistent with wave breaking, i.e., 
a solution $u$ remains bounded while 
its $x$-derivative $\pd_x u$ becomes (negatively) 
unbounded \cite{Camassa:1993zr} 
(this is rigorously demonstrated in \cite{CE1998a, CE1998b}).

Random effects are important when 
developing good mathematical models 
of complex phenomena, with carefully 
crafted SPDEs providing tools for modelling, analysis, 
and prediction. Randomness can enter models 
differently, such as through stochastic transport, 
stochastic forcing, or uncertain system parameters like 
random initial and boundary data. The work \cite{Hol2015} 
proposes a general approach to 
deriving SPDEs for fluid dynamics from a 
stochastic variational principle. This approach 
constitutes a stochastic extension of the 
classical variational derivation of Eulerian 
fluid dynamics. The corresponding 
stochastic perturbation of the CH equation leads 
to an SPDE similar to \eqref{eq:u_ch_ep} (with 
$\ep=0$), see \cite{CH2018} and also \cite{Bendall:2021tx}. 
For the related stochastic Hunter--Saxton equation, 
see \cite{HKP2021}. 

\subsection{Stochastic CH equation}\label{eq:derivation_SPDE}

Let us discuss the derivation of \eqref{eq:u_ch_ep} (with 
$\ep=0$) in more detail. Denote by $M_m$ the 
multiplication operator by $m = \bigl(1-\pd^2_{x}\bigr) u$, 
i.e., $M_m[v]=mv$, 
and by $D$ the (spatial) differentiation 
operator on $\T$. As is well known, the deterministic 
CH equation can be written in a bi-Hamiltonian form as
\begin{equation}\label{eq:bi-Hamiltonian}
	0 = \pd_t m +M_m D 
	\frac{\delta \tilde{h} [m]}{\delta m}
	+D M_m \frac{\delta \tilde{h}[m]}{\delta m},
\end{equation}
where the Hamiltonian is
\begin{equation}\label{eq:tilde-h}
	\tilde{h}[m]=\frac{1}{2} 
	\int_\T m(t,x)\, (K*m)(t,x)\;\d x,
\end{equation}
and the kernel
\begin{align}\label{eq:kernel_def}
	K(x) :=\frac{\cosh(x-2 \pi \left[\frac{x}{2\pi}\right]-\pi)}
	{2\sinh(\pi)}
\end{align}
is the Green's function for the operator 
$1-\pd_{x}^2$ on $\T$. One can formally convert the 
bi-Hamiltonian equation \eqref{eq:bi-Hamiltonian} 
into the ``transport'' system
\begin{align*}
	0 & =\pd_t u + u\, \pd_x u + \pd_x {P},
	\qquad \text{with $P=P[u]$ defined in \eqref{eq:P-def}},
\end{align*}
which is a popular formulation of the CH equation. 
It was suggested in \cite{Hol2015,CH2018} that the 
Hamiltonian ought to be perturbed 
by noise directly, so that a physically 
significant stochastic analogue of the 
CH equation should be based on 
the integrated Hamiltonian
$$
H[m]=\int_\T \int_0^t 
\frac{1}{2}m(s,x) (K*m)(t,x) \;\d s
+\int_0^t \bigl(m(s,x) \sigma(x) \bigr)
\circ \d{W}(s)\;\d x.
$$
With $\sigma \equiv 0$, we identify 
$\tilde{h}$, cf.~\eqref{eq:tilde-h}, with $\d H/\d t$. 
The first variation of $H[m]$ is
$$
\frac{\delta H[m]}{\delta m} 
= u+\sigma(x) \;\dot{W},
$$
This expression is of class $C^{-1/2-0}$ in 
time, and it is far from being 
a time-continuous object. However, at the formal 
level, compared with \eqref{eq:bi-Hamiltonian}, 
the analogous stochastic CH equation becomes
$$
0 = \d m + M_{m} D \bigl(u \;\d t+ \sigma(x) 
\; \d{W}\bigr) + D M_{m}\bigl(u\;\d t
+\sigma(x)\; \d{W}\bigr),
$$
where the multiplication operator $M_m$ 
here uses the Stratonovich product $\circ$; 
written out more explicitly, we have
\begin{equation}\label{eq:m_Stoch-CH}
	0 = \d m + \bigl(m\, \pd_x u
	+\pd_x (mu) \bigr)\;\d t 
	+ m\, \pd_x \sigma(x) \circ \d W 
	+ \pd_x \bigl(m \sigma(x) \bigr)\circ \d W.
\end{equation}
We can derive an equation for $u$ that is 
heuristically equivalent to \eqref{eq:m_Stoch-CH}. 
Under the assumption that the functions 
$u$, $m=u-\pd_{x}^2 u$ and $\sigma$ are sufficiently 
regular, we can convolve \eqref{eq:m_Stoch-CH} 
by $K$ to obtain
\begin{align*}
	0 & = \d \Bigl( 
	K*\left(u - \pd_{x}^2 u\right) \Bigr)
	+K*\Bigl(3u\, \pd_x u - 2 \pd_xu\, \pd^2_{x}u 
	-u \pd_{x}^3 u\Bigr)\;\d t 
	\\ & \qquad 
	+K* \Bigl(\pd_x\sigma\, u
	-2\pd_x \sigma\, \pd_{x}^2 u  
	+\pd_x \left(\sigma u\right)
	-\sigma\pd_{x}^3 u\Bigr)\circ \d W.
\end{align*}
Recalling the definition of $K$, 
cf.~\eqref{eq:kernel_def}, we obtain
\begin{align*}
	0 & = \d u + u\, \pd_x u\;\d t 
	+ K* \left( 2 u\, \pd_x u 
	+ \pd_x u\,\pd^2_{x}u\right)\;\d t
	\\ & \quad 
	+ \left[\sigma \pd_x u + 
	K* \left(\pd_x^2 \sigma \, \pd_x u 
	+ 2 \pd_x \sigma\, u\right) \right]\circ \d W.
\end{align*}
Setting $P = P[u]$, cf.~\eqref{eq:P-def}, we 
arrive at the final form
\begin{align}\label{eq:sCH_1}
	0 &=  \d u + \left[u\, \pd_x u
	+\pd_x P \right] \;\d t+\left[\sigma \pd_x u
	+K* \left( 2 \pd_x \sigma\, u
	+\pd_x^2 \sigma \, 
	\pd_x u\right)\right] \circ \d W.
\end{align}
Mathematically, as explained in
Remark \ref{rem:EPnoise1}, the convolution part of 
the noise term offers no new (essential) 
difficulties compared to $\sigma \pd_x u\circ \d W$. 
For the sake of clarity, we 
will therefore focus on the equation
\begin{align*}
	0 &= \d u + \left[u\, \pd_x u 
	+ \pd_x P \right] \;\d t
	+\sigma(x)\pd_x u \circ \d W.
\end{align*}
By the Stratonovich--It\^{o} 
conversion formula, the foregoing equation takes 
the operational form
\begin{align}\label{eq:u_ch}
	0 =  \d u + \left[u\, \pd_x u 
	+ \pd_x P \right] \;\d t
	-\frac12\sigma \pd_x \bk{\sigma(x) 
	\pd_x u}\,\d t
	+\sigma \pd_x u \, \d W.
\end{align}

Regarding the analysis of the stochastic 
CH equation \eqref{eq:u_ch}, there are few 
results available at the moment. To better describe 
the situation, let us note that the 
equations discussed so far are all 
nonlinear SPDEs of the form 
$$
0=\pd_t u + u\pd_x u +S\bigl(u,\pd_x u,\pd_x^2 u\bigr)
+\Gamma\bigl(x,u,\pd_x u\bigr)\, \d W.
$$ 
Depending on the specification of the functions 
$S$ and $\Gamma$, randomness enters the equation in 
different ways, including stochastic forcing 
(noise through a lower order ``source term") 
or gradient-dependent noise (noise through 
a transport/convection operator). Examples of 
stochastic forcing arise if 
$\Gamma\bigl(x,u,\pd_x u\bigr)=\beta(x,u)$, for 
some function $\beta$. A typical gradient 
noise example is $\Gamma\bigl(x,u,\pd_x u\bigr)
=\sigma(x)\pd_x u$, for some function $\sigma$, 
like in \eqref{eq:u_ch} or \eqref{eq:u_ch_ep}.
Now most of the results in the literature concern 
the ``stochastic forcing" case, either via 
additive ($\beta=\beta(x)$) or multiplicative 
($\beta=\beta(u)$) noise, see the 
works \cite{Chen:2016aa,Chen:2012aa,Chen:2020wp,
Huang:2013ab,Zhang:2020vy,Lv:2019wt,
Rohde:2021aa,Tan2018,Tan2020,Zhang:2021wl}. 
For gradient noise, we refer to \cite{ABD2019} for 
a local well-posedness result (up to wave-breaking) 
for \eqref{eq:m_Stoch-CH}. The idea 
in \cite{ABD2019} is to transform the equation 
into a PDE with random coefficients, and 
apply Kato's operator theory.
The work \cite{Alonso-Oran:2021wt} extends this 
result to a stochastic two-component CH system with 
gradient noise $\sigma(x)\pd_x \circ \d W$, 
for smooth ($C^\infty$) noise functions $\sigma(x)$. 
The approach \cite{Alonso-Oran:2021wt} is based on an abstract 
SDE framework \textit{\`a la} \cite{LR2015}, but 
one that is adapted to handle gradient-dependent 
noise operators (the original framework \cite{LR2015} 
applies to stochastic forcing operators).  
The global-in-time existence of properly 
defined weak solutions for the stochastic 
CH equation \eqref{eq:u_ch} is an open problem, but 
see \cite{CDG2021} for some partial results 
if $\sigma$ is a constant. 

\subsection{Main results}

In this paper, we study a regularised 
version of \eqref{eq:u_ch}, namely the 
SPDE \eqref{eq:u_ch_ep}, which contains a viscous 
dissipation term $\ep \pd_x^2 u$, $\ep>0$. 
The second-order operator in \eqref{eq:u_ch} involving 
$\sigma$ is not a regularising 
(parabolic) operator. 
The technical reason for 
this is that the quadratic variation of the martingale 
part of the equation coincides 
with the dissipation generated by the 
second-order operator. The difference 
between these two terms arises when  
computing the nonlinear composition $S(u)$ 
using It\^{o}'s formula.
There are several reasons why we focus on \eqref{eq:u_ch_ep} 
instead of \eqref{eq:u_ch}. First of all, 
with a few exceptions (discussed above), 
a general (global) well-posedness 
theory for \eqref{eq:u_ch} is missing, 
and \eqref{eq:u_ch_ep} appears to be 
a natural place to start. Apart from that, 
in ongoing work, we are investigating the 
existence of dissipative weak solutions 
for \eqref{eq:u_ch}. This class of 
global weak solutions is strongly linked to the 
well-posedness of \eqref{eq:u_ch_ep}. 
Indeed, in the deterministic literature there are two 
natural classes of weak solutions, ``dissipative" and 
``conservative", which differ in how they 
continue the solution past the blow-up time.
Conservative solutions ask that the PDE holds 
weakly and that the total energy is preserved. 
In contrast, dissipative solutions 
are characterized by a drop in the total 
energy (at the time of blow up).
To demonstrate the existence of an appropriately 
defined dissipative solution to \eqref{eq:u_ch}, 
one starts from the well-posedness of 
the viscous SPDE \eqref{eq:u_ch_ep} to 
construct an approximate solution sequence 
$\seq{u_\ep}_{\ep>0}$, exhibiting good 
regularity properties and a priori estimates, 
and then attempt to pass to the 
limit $\ep\to 0$ to produce a solution of 
the inviscid equation \eqref{eq:u_ch}, making 
use of subtle weak convergence and propagation 
of compactness techniques (the details will 
be presented in an upcoming work). 

In the present paper, as a first step towards 
global existence for \eqref{eq:u_ch},  
we will develop a rather complete (global) 
well-posedness theory for \eqref{eq:u_ch_ep}, which 
allows for general ``non-smooth" noise 
functions $\sigma(x)$. Roughly speaking, 
by a solution to \eqref{eq:u_ch_ep} we mean a stochastic 
process $(\omega,t)\mapsto u(\omega,t,\dott)$ 
that takes values in $H^1(\T)$ and satisfies 
the SPDE in the weak sense in $x$. These 
solutions are strong (or pathwise) in 
the probabilistic sense, i.e., they are 
adapted to an underlying fixed filtration.  
For a detailed description of the concept of 
solution, see Definitions \ref{def:wk_sol} 
and \ref{def:st_sol}. The first main theorem of the paper is
the following result.

\begin{thm}[Well-posedness in $H^1$]\label{thm:main1}
Suppose $\sigma \in W^{2,\infty}(\T)$, $p_0 > 4$, 
and $u_0 \in L^{p_0}(\Omega;H^1(\T))$. 
There exists a unique strong $H^1$ solution 
to \eqref{eq:u_ch_ep} with initial 
condition $u|_{t=0}=u_0$. 
\end{thm}

The second main result supplies well-posedness 
in higher-regularity classes:

\begin{thm}[Well-posedness in $H^m$]\label{thm:main2} 
Fix $m \ge 2$ and $p_0 > 4$. Suppose $\sigma \in W^{m + 1,\infty}(\T)$, 
and $u_0 \in L^{p_0}(\Omega;H^{m}(\T))$. 
There exists a unique strong $H^m$ 
solution to \eqref{eq:u_ch_ep} with 
initial condition $u|_{t=0}=u_0$. 
\end{thm}

There is a sense in which the ``natural" energy space 
given by the structure of the CH equation
is $L^\infty_tH^1_x$; see beginning of 
Section \ref{sec:higherreg_existence} where $H^m_x$ 
estimates are discussed and compared against 
estimates in $H^1_x$.

The $H^m_x$ estimates ($m\ge 2$) 
do not follow from standard parabolic regularity theory 
because of nonlinear factors of cubic type,
and the fundamental lack of 
$L^1_tL^\infty_{\omega,x}$ estimates on 
$u$ and $\pd_x u$ (this is in contrast to 
the deterministic equation \cite{XZ2000}). 
We cope with these problems 
using stopping time arguments and a stochastic 
Gronwall inequality. The moment requirement 
$p_0 > 4$ on the initial condition in $H^{m}_x$ comes from 
technical lemmas (Lemmas \ref{thm:commutator1}, \ref{thm:commutator2} 
and Proposition \ref{thm:commutator3}, and  
 see also Remark \ref{rem:L8_omega}).

\subsection{Organisation of paper}\label{sec:organisation}
We bring this introduction to an end by outlining 
the organization of the paper, along with a quick 
exposition of ideas behind the proofs. 

First, in Section \ref{sec:sol_def}, 
we state precisely the different 
solution concepts used throughout the paper. 
The existence parts of Theorems 
\ref{thm:main1} and \ref{thm:main2} 
are based on weak solutions and the 
introduction of suitable Faedo--Galerkin 
approximations $\seq{u_n}$, where the epithet
``weak" refers to probabilistic weak 
and so-called martingale solutions. 
A refined stochastic compactness method \cite{Ond2010} 
is used to conclude convergence $\seq{u_n}$ 
towards a weak solution. In the context of SPDEs, 
the stochastic compactness method goes back to 
\cite{Bensoussan:1995aa} and it was subsequently 
used in numerous works, see for example \cite{DGT2011,FG1995,GV2014} 
and the references therein. In Section \ref{sec:galerkin}, we define and establish 
the well-posedness of the Faedo--Galerkin approximations.
A priori estimates and tightness properties 
of the approximations $\seq{u_n}$ are proved in 
Sections \ref{sec:apriori} and \ref{sec:laws}. 
More precisely, in Proposition~\ref{thm:galerkin_H1p} 
and Lemma~\ref{thm:u_Ctheta_L2} we supply 
several $n$-uniform (and $\ep$-uniform) 
bounds that imply
\begin{align*}
	\seq{u_n} \subseteq_b 
	L^p\bigl(\Omega;L^\infty([0,T]; H^1(\T))\bigr)
	\cap L^p\bigl(\Omega;C^\theta([0,T];L^2(\T))\bigr),
\end{align*}
for appropriate ranges of $p$ and $\theta$ (where $\subseteq_b$ means ``bounded inclusion", i.e., $A\subseteq_b X$ if
$A\subseteq X$ and $\sup_{a\in A}\norm{a}_X<\infty$). 
We use this and the compact inclusion
$$
L^\infty([0,T]; H^1(\T))\cap C^\theta([0,T];L^2(\T))
\doublehookrightarrow C([0,T];H^1_w(\T))
$$
to deduce the tightness of the probability 
laws of the Faedo--Galerkin solutions in the 
quasi-Polish space $C([0,T];H^1_w(\T))$.   
Here $H^1_w(\T)$ denotes the 
space $H^1(\T)$  with the weak topology. 
Because of a uniform-in-$n$ bound 
on $\Ex\norm{u_n}_{L^2([0,T];H^2(\T))}^2$, 
arising from the $\ep$-dissipation operator 
in \eqref{eq:u_ch_ep}, we also obtain the 
uniform stochastic boundedness of 
$\seq{u_n}$ in the space $L^2([0,T];H^2(\T)) 
\cap W^{\theta',2} ([0,T];L^2(\T))$, 
with $\theta'<\theta$. Hence, it follows 
that the probability laws of $\seq{u_n}$ are 
tight on $L^2([0,T];H^1(\T))$, 
cf.~Lemma \ref{thm:stochastic_boundedness}. Using the 
Skorokhod--Jakubowski theorem \cite{Jak1998} 
of almost sure representations of 
random variables in quasi-Polish spaces 
(see Appendix \ref{sec:toolbox}), we deduce 
in Section \ref{sec:existence1} the existence of weak 
(martingale) solutions to the viscous stochastic CH 
equation \eqref{eq:u_ch_ep}. In Sections \ref{sec:commutators}--\ref{sec:pathwise_unique}, we prove pathwise uniqueness 
for \eqref{eq:u_ch_ep} by a renormalisation procedure, 
bypassing the need for an infinite dimensional 
It\^{o} formula.
The uniquenss proof requires some non-standard 
first- and second-order commutator estimates (that 
extend beyond the standard DiPerna--Lions estimates), 
which are established in Lemmas \ref{thm:commutator1}, 
\ref{thm:commutator2} and Proposition \ref{thm:commutator3}. 
Pathwise uniqueness, along with the weak existence 
result and also the Gy{\"o}ngy--Krylov characterization 
of convergence in probability, allows 
us to conclude in Section  \ref{sec:H1_existence} 
the existence of a unique strong 
(pathwise) $H^1$ solution to \eqref{eq:u_ch_ep}, 
thus concluding the proof of 
Theorem \ref{thm:main1}.  

One-sided strong temporal continuity characterises 
{\em dissipative} weak solutions in the inviscid 
$\ep \downarrow 0$ limit. For fixed positive viscosity, 
solutions satisfy (two-sided) strong temporal continuity. 
This is demonstrated afterwards in Section \ref{sec:temporalcont}.

In Section \ref{sec:higherreg}, we turn to Theorem \ref{thm:main2} and 
solutions with higher regularity. 
In Section \ref{sec:higherreg_existence}, we 
fix $m \ge 2$ and prove $n$-uniform bounds in
$L^p\left(\Omega;L^\infty([0,\tau];H^m(\T)\right)$, 
for $p \in [1,\infty)$, up to a suitable 
stopping time $\tau$ (Proposition \ref{thm:galerkinHm}). 
Using this we conclude the stochastic 
boundedness (see \eqref{eq:stochbounded_defin}) in the higher regularity space 
$L^2([0,T];H^{m+1}(\T)) \cap W^{\theta, 2}([0,T];L^2(\T))$, 
for some $\theta<1$, as long as the initial 
condition $u_0$ belongs to $L^p(\Omega;H^1(\T)) \cap 
L^2(\Omega;H^{m}(\T))$. By 
some additional stopping time 
arguments, this implies that the laws 
of $\seq{u_n}$ are tight on $L^2([0,T];H^m(\T))$, 
(see Lemma \ref{thm:stoch_bound_Hm_v2}), and by 
a Skorokhod--Jakubowski procedure 
(as in  Section \ref{sec:existence1}) we extract 
a weak solution in $L^2([0,T];H^m(\T))$. 
The key difference between the $H^1$ and 
$H^m$ cases lies in the lack of a bound on $\Ex 
\norm{u_n}_{L^2([0,T];H^{m}(\T))}^2$, that is, if 
$m\ge 2$, then Lemma \ref{thm:stoch_bound_Hm_v2} 
is available for $H^{m}$ only up to some 
stopping time $\tau<T$ but not on $[0,T]$. 
This obstacle, which is peculiar to the 
stochastic problem, makes it necessary 
to argue along several layers of stopping times, 
see Lemma \ref{thm:stoch_bound_Hm_v2} and its proof. 
Finally, in Section \ref{sec:st_sol_Hm}, we establish 
the pathwise uniqueness in $L^1(\Omega;L^\infty([0,T];H^m(\T)))$ 
and conclude the well-posedness of strong $H^m$ solutions. 

In Appendix \ref{sec:toolbox}, we record some 
results of stochastic analysis frequently 
deployed in this paper.

\section{Solution concepts}
\label{sec:sol_def}

In this section, we present the solution concept used 
in Theorems \ref{thm:main1} and \ref{thm:main2}. 
We denote by $\bigl(\Omega,\mathcal{F},
\mathbb{P}\bigr)$ a complete \textit{probability space} 
with (countably generated) $\sigma$-algebra 
$\mathcal{F}$ and probability measure $\mathbb{P}$. 
We consider filtrations $\{\mathcal{F}_t\}_{t\in[0,T]}$ 
that satisfy the ``usual conditions" of being complete 
and right-continuous. We refer to
\begin{equation}\label{eq:stoch-basis}
	\mathcal{S}:=\bigl(\Omega,\mathcal{F},
	\{\mathcal{F}_t\}_{t\ge 0},\mathbb{P}\bigr)
\end{equation}
as a \textit{stochastic basis} (sometimes called a filtered 
probability space). 

Theorems \ref{thm:main1} 
and \ref{thm:main2} speak of \textit{strong} $H^m$ 
solutions. These are weak (distributional) solutions 
in the PDE sense in the Sobolev space $H^m$. 
From the probabilistic point of view, however, we 
will have to consider first so-called 
\textit{martingale} solutions, which are also 
referred to as \textit{weak} solutions.  
The notions of weak/strong probabilistic solutions 
have a different meaning from weak/strong solutions 
in the PDE literature. If the stochastic basis $\mathcal{S}$ 
\eqref{eq:stoch-basis} and the Wiener process 
$W$ are fixed in advance, we speak of a strong 
(or pathwise) solution. If $\bigl(\mathcal{S},W\bigr)$ 
is a part of the unknown solution, the relevant notion 
is a martingale solution. In what follows, ``weak $H^m$ 
solutions" refer to solutions that are probabilistic weak 
and weak in the PDE sense, whereas ``strong $H^m$ 
solutions" refer to solutions that are pathwise 
and weak in the PDE sense.

In view of Theorems \ref{thm:main1} and \ref{thm:main2}, 
the $H^1$ well-posedness theory deviates 
slightly from the $H^m$ theory for $m \ge 2$. 
The corresponding solution concepts differ 
in their requirement on the initial condition and 
the condition $u\in L^2\left(\Omega;L^2([0,T];H^2(\T))\right)$ 
if $m=1$ versus the weaker stochastic 
boundedness condition in $L^2([0,T];H^{m+1}(\T))$ if $m\ge 2$, as is 
seen in the next definition.

\begin{defin}[Weak $H^m$ solution]\label{def:wk_sol} 
Fix $m\in\N$ and $p_0 > 4$. Let $\Lambda$ be a probability 
measure on $H^{m}(\T)$ satisfying
\begin{equation}\label{eq:u0-cond-Hmp1}
	\int_{H^{m}(\T)} 
	\norm{v}_{H^{m}(\T)}^{p_0} 
	 \Lambda(\d v)<\infty.
\end{equation}
The triple $\bigl(\mathcal{S},u,W \bigr)$ 
is a weak (or martingale) $H^m$ solution 
to \eqref{eq:u_ch_ep} with initial distribution 
$\Lambda$ if the following conditions hold:
\begin{itemize}
	\item[(a)] $\mathcal{S}=\bigl(\Omega, 
	\mathcal{F},\{\mathcal{F}_t\}_{t \ge 0},
	\mathbb{P}\bigr)$ is stochastic basis, 
	cf.~\eqref{eq:stoch-basis};
	
	\item[(b)] $W$ is a standard Wiener process 
	on $\mathcal{S}$;

	\item[(c)] $u\colon\Omega\times[0,T]\to H^1(\T)$ is 
	adapted, with $u\in L^{p_0}\left(\Omega;C( [0,T];H^1(\T)\right)$
	and $u \in L^2([0,T];H^{m}(\T))$, $\mathbb{P}$-almost surely.
	Moreover,
	$$
	\begin{cases}
		u\in L^2\left(\Omega;L^2([0,T];
		H^2(\T)\right), & \textit{if $m=1$}, 
		\\ 
		u\in_{\rm sb} L^2([0,T];H^{m + 1}(\T)) \cap L^\infty([0,T];H^m(\T)), 
		& \textit{if $m\ge 2$},
	\end{cases}
	$$
	where $\in_{\rm sb}$ means stochastically bounded 
	(see \eqref{eq:stochbounded_defin});
	 
	\item[(d)] the law of the initial data $u_0:=u(0)$  on $H^{m}(\T)$ 
	is $\Lambda$, i.e., $\bk{u(0)}_*\mathbb{P}=\Lambda $, 
	or $\Lambda(A)=\mathbb{P}(u(0)^{-1}(A))$ for measurable sets $A$;

	\item[(e)] for all $t \in [0,T]$ and all 
	$\varphi \in C^1(\T)$
	the following equation holds $\mathbb{P}$-almost surely (in 
	the sense of It\^o):
		\begin{align*}
  			\int_\T u(t) \varphi& \,\d x
  			-\int_\T u_0 \varphi\,\d x\\
  			&= \int_0^t\int_\T \left(- u\,\pd_x u \,\varphi +  \left(P
  			-\ep \pd_x u \right)\pd_x \varphi\right) \,\d x\, \d s
  			\\ & \quad 
  			-\frac12\int_0^t\int_\T \sigma \pd_x u
  			\, \pd_x\bk{\sigma \varphi} \,\d x \,\d s
  			-\int_0^t\int_\T \varphi \sigma 
  			\, \pd_x u\,\d x \, \d W(s),
  			\end{align*}
\end{itemize}
where $P=P[u] :=K*\left(u^2+\frac{1}{2}
\left(\pd_x {u}\right)^2\right)$.
\end{defin}

Finally, we introduce the notion of 
strong (pathwse) $H^m$ solution.

\begin{defin}[Strong $H^m$ solution]\label{def:st_sol} 
Fix a stochastic basis $\mathcal{S}$, 
cf.~\eqref{eq:stoch-basis}, and a Wiener process $W$ 
defined on $\mathcal{S}$. Fix $m\in\N$ and $p_0 > 4$, and consider 
a random variable $u_0\in L^{p_0}(\Omega;H^1(\T))$. 
A process $u$, defined relative to $\mathcal{S}$, is a strong $H^m$ solution 
to \eqref{eq:u_ch_ep} if $\bigl(\mathcal{S},u,W\bigr)$ 
is a weak $H^m$ solution to \eqref{eq:u_ch_ep} with initial 
law $\Lambda:=\bk{u_0}_*\mathbb{P}$, 
i.e., $\Lambda$ obeys \eqref{eq:u0-cond-Hmp1} 
and $\bigl(\mathcal{S},u,W\bigr)$ 
satisfies (a)--(e) in Definition \ref{def:wk_sol}.
\end{defin}

\section{The Galerkin approximation}\label{sec:galerkin}

We now specify our Galerkin scheme for constructing 
approximate solutions. Let $\seq{e_1,e_2,\ldots} 
\subseteq H^1(\T)$ be an orthonormal basis of $L^2(\T)$ 
that is dense in $H^1(\T)$ and set 
$H_n=\mathrm{span}\seq{e_1,\ldots, e_n}$. 
In particular, we take $\seq{e_i}_{i\in \N}$ 
to be the eigenfunctions of $\pd_x^2$ on the 
circle $\T$, i.e., $e_{2j}(x) = \cos(2 \pi jx)$ and 
$e_{2j+1}(x) = \sin(2\pi jx)$, $x\in[0,1]$, for concreteness. 
Let $\bm{\Pi}_n:\bigl(H^1(\T)\bigr)^* \to H_n$ 
be defined by
$$
\bm{\Pi}_n u:= \sum_{i = 1}^n 
\bigl\LL u,e_i\bigr\RR_{L^2(\T)} \,e_i,
$$
so that, restricted to $L^2(\T)$, $\bm{\Pi}_n$ is 
the orthogonal projection onto $H_n$.

For each $n \in \mathbb{N}$, we 
consider the Galerkin approximation of \eqref{eq:u_ch_ep} 
on $H_n$, that is, we seek a function 
$$
u_n(\omega,t,x)=\sum_{i=1}^n w_i(\omega,t)e_i(x),
$$
where the unknown coefficients 
$\seq{w_i=w_i(\omega,t)}_{i=1}^n$ are 
determined by requiring that
\begin{equation}\label{eq:galerkin_n}
	\begin{aligned}
		 0 &= \d u_n-\ep \pd_x^2 u_n \,\d t
		+\bm{\Pi}_n \bk{u_n \pd_x u_n+\pd_xP[u_n]} \,\d t
		\\ &\quad
		-\frac{1}{2} \bm{\Pi}_n \bk{\sigma \pd_x
		\bk{\sigma \pd_x u_n}}\, \d t
		+\bm{\Pi}_n\bk{\sigma \pd_x u_n} \, \d W,
		\\
		u_{n}(0)&=\bm{\Pi}_n u_0.
	\end{aligned}
\end{equation}
Here, $u_0$ is a random variable $\Omega \to H^{1}(\T)$ 
with law $\Lambda$ and a bounded second 
moment, i.e., $\Ex \norm{u_0}_{H^{1}(\T)}^2<\infty$. 

\begin{thm}
For any fixed $n$, there exists a unique $C([0,T];H_n)$-valued 
adapted process $u_n$ that is a strong solution 
to \eqref{eq:galerkin_n}.
\end{thm}

\begin{proof}
The proof consists of noting that \eqref{eq:galerkin_n} 
is a SDE system with coefficients that are locally Lipschitz 
continuous in $w=\seq{w_i}_{i\in \N}$. By a standard 
well-posedness theorem for 
SDEs \cite[Thm.~IX.2.1]{Revuz:1999wi}, this 
immediately implies the existence and uniqueness of a 
continuous (strong) solution of 
\eqref{eq:galerkin_n} on $[0,T]$.

It remains to argue that the Galerkin 
equation \eqref{eq:galerkin_n} can be viewed as a 
SDE system in $w$. First, by properties of the basis functions, 
$$
\pd_x u_n=\sum_{i=1}^n C_i\, w_i\, e_i,\quad
\pd_x^2 u_n=\sum_{i=1}^n -C_i^2\, w_i\, e_i,
$$
where $C_i$ are constants depending only on $i$. 
Next, the nonlinear term
\begin{align*}
	\bm{\Pi}_n \bigl(u_n \pd_x u_n\bigr)
	& =\sum_{i,j=1}^n C_j \, w_i w_j\,\bm{\Pi}_n
	\bigl(e_i e_j\bigr)
	\\ & =\sum_{i,j,k=1}^n C_j\, w_i w_j 
	\int_\T e_i(y) e_j(y) e_k(y) \,\d y \, e_k,
\end{align*}
is locally Lipschitz in $w$. 
Regarding the nonlocal operator, we can calculate thus:
\begin{align*}
	&\bm{\Pi}_n \pd_x P[u_n]\\
	&=\sum_{k=1}^n\int_\T \pd_y K*\bk{u_n^2
	+\frac12\bk{\pd_y u_n}^2} e_k(y)\,\d y\, e_k
	\\ & 
	= \sum_{i,j,k=1}^n\int_\T \int_\T 
	\pd_z K(z-y)\Bigl(w_iw_j e_i(y) e_j(y) 
	\\ & \qquad\qquad 
	+\frac12 C_i C_j w_i w_j e_i(y) e_j(y)\Bigr)
	\, e_k (z)\,\d y\,\d z\,e_k
	\\ & 
	= \sum_{i,j,k=1}^n w_i w_j  \bk{\int_\T \int_\T 
	\pd_z K(z - y)\bk{1+\frac12 C_i C_j }
	e_i(y) e_j(y) \, e_k (z)\,\d y\,\d z}\; e_k,
\end{align*}
which is then seen also to be locally Lipschitz in $w$. Similarly, 
we can show that the linear terms $\bm{\Pi}_n 
\bk{\sigma \pd_x\bk{\sigma \pd_x u_n}} $ and 
$\bm{\Pi}_n\bk{\sigma \pd_x u_n}$ 
are locally Lipschitz in $w$.
\end{proof}

\section{A priori estimates}\label{sec:apriori}

Our first result is a fundamental model-specific 
energy estimate that we will refer to repeatedly 
throughout this work.

We use frequently the fact that for any 
function $f \in L^2(\T)$,
\begin{equation}\label{eq:Pin-SA}
	\int_\T u_{n} \bm{\Pi}_n f \, \d x 
	=\int_\T u_{n} f \, \d x,
\end{equation}
because $\bm{\Pi}_n$ is self-adjoint and idempotent. 
For any $f\in H^1(\T)$ and $\frac{n-1}{2} 
\in \mathbb{N}$, we compute the spatial derivative of 
$\bm{\Pi}_n f$ as follows:
\begin{equation}\label{eq:Pin_cancellations}
	\begin{aligned}
		\pd_x  \left(\bm{\Pi}_n f\right) 
		&= \sum_{2j\le n} \LL f, e_{2j}
		\RR_{L^2(\T)}\pd_x e_{2j}
		+\sum_{2j+1\le n} \LL f,e_{2j+1}
		\RR_{L^2(\T)}\pd_x  e_{2j + 1}
		\\ & 
		=2\pi \Big(\sum_{2j\le n} j \LL f,e_{2j}\RR_{L^2(\T)} 
		e_{2j+1}-\sum_{2j+1\le n} j
		\LL f,e_{2j+1}\RR_{L^2(\T)}e_{2j}  \Big)
		\\ & =
		-\sum_{2j \le n}  \LL f,\pd_x e_{2j+1}
		\RR_{L^2(\T)} e_{2j+1}-\sum_{2j+1\le n} 
		\LL f,\pd_x e_{2j}\RR_{L^2(\T)}e_{2j}  
		\\ & 
		= \sum_{2j+1\le n} \LL \pd_x f, e_{2j+1}
		\RR_{L^2(\T)} e_{2j+1}
		+\sum_{2j\le n} \LL \pd_x f, 
		e_{2j}\RR_{L^2(\T)}e_{2j}
		\\ & 
		=\bm{\Pi}_n\left(\pd_x f\right).
	\end{aligned}
\end{equation}

\begin{prop}[Energy estimate]\label{thm:galerkin_H1} 
For each $n\in \N$, let $u_n$ be a solution 
to \eqref{eq:galerkin_n} with $\Ex 
\norm{u_0}_{H^1(\T)}^2<\infty$. 
There exists a constant 
$$
C=C\bk{T,\Ex \norm{u_0}_{H^1(\T)}^2, 
\norm{\sigma}_{W^{2,\infty}(\T)}},
$$
independent of $n$ and $\ep$, such that
\begin{align}\label{eq:galerkin_unifbd1}
	\Ex \norm{u_{n}}_{L^\infty([0,T];H^1(\T))}^2
	+\ep \Ex\int_0^T \norm{\pd_x u_{n}(t)}_{H^1(\T)}^2 
	\, \d t \le C.
\end{align}
\end{prop}

\begin{proof}
We multiply (via It\^{o}'s formula) the 
SDE \eqref{eq:galerkin_n} against $u_n$ and then 
integrate in $x\in \T$. 
Using \eqref{eq:Pin-SA} and \eqref{eq:Pin_cancellations}, 
we obtain the SDE
\begin{align*}
	\frac12 \d \int_\T \abs{u_n}^2 \,\d x 
	&+ \ep \int_\T \abs{\pd_x u_n}^2 \,\d x \,\d t  \\
	&= -\int_\T \Big(u_n^2 \pd_x u_n
	+ u_n \pd_x K*\big(u^2_n+\frac12\bk{\pd_x u_n}^2\big)\Big)
	\,\d x\,\d t
	\\ & \quad
	+\frac12 \int_\T \Big(\sigma u_n \pd_x 
	\bk{\sigma \pd_x u_n}
	+\abs{\bm{\Pi}_n \bk{\sigma \pd_x u_n}}^2\Big)
	\,\d x\,\d t
	\\ &\quad
	- \int_\T \sigma u_n\pd_x u_n \,\d x \, \d W.
\end{align*}
Differentiating \eqref{eq:galerkin_n} and 
multiplying through 
by $\pd_x u_n=\bm{\Pi}_n \pd_x u_n$ (via It\^{o}'s formula), 
using again \eqref{eq:Pin-SA} and 
\eqref{eq:Pin_cancellations}, yields
\begin{align*}
	\frac12 \d \int_\T \abs{\pd_x u_{n}}^2 \,\d x 
	&+ \ep \int_\T \abs{\pd_x^2 u_n}^2\,\d x \, \d t 
	\\ & 
	=  \int_\T  \Big(u_n \pd_x u_n \pd_x^2 u_n  
	+ \pd_x^2 u_n \pd_x K*\big(u^2_n+\frac12 
	\bk{\pd_x u_n}^2\big)\Big)\,\d x \,\d t
	\\ &\quad
	-\frac12 \int_\T \Big(\pd_x^2 u_n
	\sigma \pd_x\bk{\sigma \pd_x u_n}
	-\abs{\pd_x \bk{\bm{\Pi}_n
	\bk{\sigma \pd_x u_n}}}^2\Big)\,\d x \,\d t
	\\ & \quad
	+\int_\T \sigma \pd_x u_n\,\pd_x^2 u_n\,\d x\,\d W.
\end{align*}
Adding the previous two equations, we arrive at
\begin{equation}\label{eq:H1_intermediate1}
	\frac12 \d \norm{u_n}_{H^1(\T)}^2
	+\ep \int_\T \Big( \abs{\pd_x  u_n}^2 
	+ \abs{\pd_x^2 u_{n}}^2\Big)\,\d x \, \d t
	= I_1^n \,\d t + I_2^n \,\d t + I_3^n \,\d W,
\end{equation}
where
\begin{equation}\label{eq:H1_intermediates}
	\begin{aligned}
		I_1^n & :=  
		-\int_\T \Big(u^2_n \pd_x u_n + u_n \pd_x K*\big(u^2_n
		+\frac12 \bk{\pd_x u_{n}}^2\big)\Big)\,\d x
		\\ & \quad 
		+\int_\T \Big(u_n \pd_x u_n \pd_x^2 u_n
		+\pd_x^2 u_n \, \pd_x K *\big(u^2_n
		+\frac12 \bk{\pd_x u_{n}}^2\big)\Big)\,\d x, 
		\\ 
		I_2^n & := 
		\frac12 \int_\T \Big(\sigma u_n \pd_x \bk{\sigma \pd_x u_n}
		+\abs{\bm{\Pi}_n \bk{\sigma \pd_x u_n}}^2\Big)\,\d x
		\\ & \quad 
		-\frac12 \int_\T \Big(\pd_x^2 u_n
		\sigma \pd_x\bk{\sigma \pd_x u_n}
		-\abs{\pd_x \bk{\bm{\Pi}_n\bk{\sigma 
		\pd_x u_n}}}^2\Big)\, \d x 
		\\ & =: \frac12I_{2,1}^n-\frac12I_{2,2}^n,
		\\ I_3^n &: =
		-\int_\T \Big(\sigma u_n \pd_x u_n 
		-\sigma \pd_x u_n \pd_x^2 u_n\Big)\,\d x.
	\end{aligned}
\end{equation}

\medskip
{\it 1. Estimate of $I_1^n$.}

Using integration by parts and the 
kernel property of $K$ that 
$K - \pd_x^2 K = \bm{\delta}$, the Dirac mass,
\begin{equation}\label{eq:H1_estm_A}
	\begin{aligned}
		I_1^n & =
		-\int_\T \Big(u^2_{n} \pd_x u_n
		-\pd_x  u_{n} K*\big(u^2_n
		+\frac12 \bk{\pd_x u_{n}}^2\big)\Big)\,\d x
		\\ & \quad
		+\int_\T \Big(u_n\pd_x u_n \pd_x^2u_n 
		-\pd_x u_n \pd_x^2 K*\big(u^2_n 
		+\frac12\bk{\pd_x u_{n}}^2\big)\Big)\,\d x= 0.
	\end{aligned}
\end{equation}

\medskip
{\it 2. Estimate of $I_2^n$.}

By Bessel's inequality,
$$
\int_\T \abs{\bk{\bm{\Pi}_n
\bk{\sigma \pd_x u_n}}}^2\,\d x 
\le \int_\T \abs{\sigma \pd_x u_n}^2\,\d x.
$$
Combining this with an integration by parts 
in the $I_{2,1}^n$-term $\sigma u_n \pd_x 
\bk{\sigma \pd_x u_n}$, and then expanding out 
$\pd_x \bk{\sigma u_n}$ followed by 
another integration by parts, yields
$$
I_{2,1}^n\le -\int_\T \sigma\pd_x 
\sigma u_n \pd_x u_n\, \d x
=-\frac14\int_\T \pd_x \sigma^2 \pd_x u_n^2\, \d x.
$$
Similarly, by \eqref{eq:Pin_cancellations} and 
Bessel's inequality,
\begin{align*}
	\int_\T \abs{\pd_x \bigl(\bm{\Pi}_n
	\bk{\sigma \pd_x u_n}\bigr)}^2\,\d x 
	&=\int_\T \abs{\bm{\Pi}_n\bigl(\pd_x 
	\bk{\sigma \pd_x u_n}\bigr)}^2\,\d x
	\\ &  
	\le \int_\T 
	\abs{\pd_x \bk{\sigma \pd_x u_n}}^2\,\d x
	\\ & 
	=\int_{\T} \Big(\abs{\pd_x\sigma \pd_x u}^2
	+\frac12\pd_x\sigma^2 \pd_x \abs{\pd_x u_n}^2
	+\abs{\sigma \pd_x^2 u_n}^2\Big)\, \d x
	\\ & 
	=\int_{\T}\Big(\abs{\pd_x\sigma \pd_x u}^2
	-\frac12\pd_x^2\sigma^2 \abs{\pd_x u_n}^2
	+\abs{\sigma \pd_x^2 u_n}^2\Big)\, \d x.
\end{align*}
We combine this with an expansion of 
the $I_{2,2}^n$-term $\pd_x^2 u_n\sigma 
\pd_x\bk{\sigma \pd_x u_n}$ into the sum
$\sigma^2 \abs{\pd_x^2 u_n}^2+\frac14\pd_x \sigma^2  
\pd_x\bk{\pd_x u_n}^2$, along with an integration 
by parts in the latter term:
\begin{align*}
	I_{2,2}^n &\ge 
	\int_\T \Big(\sigma^2 \abs{\pd_x^2 u_n}^2
		+\frac14\pd_x \sigma^2 \pd_x\bk{\pd_x u_n}^2
		\\ & \qquad \qquad
		-\abs{\pd_x\sigma \pd_x u}^2
		+\frac12\pd_x^2\sigma^2 \abs{\pd_x u_n}^2
		-\abs{\sigma \pd_x^2 u_n}^2\Big)\, \d x
		\\ &= \int_\T \Big(\frac14\pd_x^2 \sigma^2\abs{\pd_x u_n}^2
		-\abs{\pd_x\sigma \pd_x u}^2\Big)\, \d x.
\end{align*}
Hence
\begin{align}
	2 I_2^n & = I_{2,1}^n-I_{2,2}^n 
	\notag \\ & 
	\le \int_\T \Big(\frac14\pd_x \sigma^2 u_n^2
	-\frac14\pd_x^2 \sigma^2 \abs{\pd_x u_n}^2
	+\abs{\pd_x\sigma \pd_x u}^2\Big)\, \d x
	\notag \\ & 
	\le C_\sigma \norm{u_n}_{H^1(\T)}^2.
	\label{eq:H1_estm_BC}
\end{align}

\medskip
{\it 3. Estimate of martingale term $I_3^n$.}

First, since $I_3^n= \frac{1}{2} \int_\T \pd_x\sigma
\left(\abs{u_n}^2-\abs{\pd_x u_n}^2\right)\,\d x$, 
we have the estimate
\begin{align*}
	\abs{I_3^n} 
	& \le \frac12\norm{\pd_x \sigma}_{L^\infty(\T)}
	\norm{u_n}_{H^1(\T)}^2.
\end{align*}
By the Burkholder--Davis--Gundy inequality,
\begin{align*}
	\Ex \sup_{s\in [0,t]}
	\abs{\int_0^s I_3^n \,\d W}
	& \le \Ex \bk{\int_0^t \abs{I_3^n}^2\,\d s}^{1/2}
	\\ & 
	\le \tilde{C}_\sigma\, \Ex \bk{\int_0^t 
	\norm{u_n(s)}_{H^1(\T)}^4 \,\d s}^{1/2} 
	=:\tilde{I}_3^n.
\end{align*}
By the H\"older and Young inequalities, we can 
further estimate the above by
\begin{align}
	\tilde{I}_3^n & \le \tilde{C}_\sigma \,\Ex
	\bk{\int_0^t\norm{u_n(s)}_{H^1(\T)}^2
	\,\d s \sup_{s\in [0,t]}
	\norm{u_n(s)}_{H^1(\T)}^2}^{1/2}
	\notag \\ & 
	\le C_\sigma\Ex \int_0^t 
	\norm{u_n(s)}_{H^1(\T)}^2\,\d s
	+\frac12 \Ex \sup_{s\in [0,t]}\norm{u_n(s)}_{H^1(\T)}^2.
	\label{eq:H1_estm_M}
\end{align}

\medskip
{\it 4. Conclusion.}

Gathering the estimates \eqref{eq:H1_estm_A},
\eqref{eq:H1_estm_BC}, and \eqref{eq:H1_estm_M}, 
we conclude that there exists a constant $C$, 
independent of $n$ and $\ep$, such that
\begin{align*}
	\frac12\Ex \sup_{s\in[0,t]}
	\norm{u_n(s)}_{H^1(\T)}^2 
	&+\ep \int_0^t \int_\T \Big(\abs{\pd_x u_n}^2 
	+\abs{\pd_x^2 u_n}^2\Big)\,\d x \, \d s 
	\\ & 
	\le \Ex \norm{u_{n}(0)}_{H^1(\T)}^2
	+C\Ex \int_0^t 
	\norm{u_{n}(s)}_{H^1(\T)}^2\;\d s,
	\quad t\in [0,T],
\end{align*}
which implies \eqref{eq:galerkin_unifbd1} 
by Gronwall's inequality.
\end{proof}

The previous lemma supplies control of the second 
moment of the $H^1(\T)$-norm. This effectively 
guarantees that higher moments are bounded as well.

\begin{lem}[Higher moment bounds for the $H^1(\T)$-norm]
\label{thm:galerkin_H1p}
Fix $p \in (4,\infty)$, and let $u_{n}$ be a solution 
to \eqref{eq:galerkin_n} with $\Ex \norm{u_0}_{H^1(\T)}^p 
<\infty$. There exists a constant 
$$
C=C\bk{p,T,\norm{\sigma}_{W^{2,\infty}(\T)}},
$$
independent of $n$ (and $\ep$), such that
\begin{equation}\label{eq:galerkin_unifbd2}
	\begin{aligned}
		\Ex &\sup_{t \in [0,T]}\norm{u_n}_{H^1(\T)}^p 
		\\ &\qquad
		+\ep^{p/2} \Ex \bk{\int_0^T  \int_\T  \big(\abs{\pd_x u_n}^2
		+\abs{\pd_x^2 u_n}^2\big)\,\d x \, \d t }^{p/2}
		\le C\Ex \norm{u_0}_{H^1(\T)}^{p}.
	\end{aligned}
\end{equation}

\end{lem}

\begin{rem}
Insofar as the bound $\Ex 
\norm{u_{n}}_{L^\infty([0,T];H^1(\T))}^{p} 
\le C$ is concerned, since $(\Omega, \mathbb{P})$ 
is a finite measure space, the bound 
with the same constant holds for any 
$r\in [1,p)$ in place of $p$, 
though the theorem is stated for $p > 4$. 
By \eqref{eq:galerkin_unifbd2} and 
the one-dimensional embedding 
$H^1(\T) \hookrightarrow L^\infty (\T)$, 
we have also that
\begin{equation}\label{eq:Linfty-est}
	\Ex\norm{u_n}_{L^\infty([0,T]\times \T)}^p
	\lesssim_p 1, \quad p\in [1,\infty),
\end{equation}
but $u_n$ is not uniformly 
bounded in $L^\infty_{\omega,t,x}$. 
\end{rem}

\begin{proof}
By \eqref{eq:H1_intermediate1} and parts 1, 2 
of the proof of Proposition \ref{thm:galerkin_H1}, we have
\begin{align*}
	& \frac12\d \norm{u_n}_{H^1(\T)}^2 
	+\ep \int_\T  \big(\abs{\pd_x u_n}^2
	+\abs{\pd_x^2 u_n}^2\big)\,\d x \, \d t 
	\\ & \qquad
	=\frac12\int_\T \big(\sigma u_n\pd_x \bk{\sigma \pd_x u_n}
	+\abs{\bm{\Pi}_n \bk{\sigma \pd_x u_n}}^2\big) \,\d x\,\d t
	\\ & \quad \qquad 
	-\frac12\int_\T \big(\pd_x^2 u_n\sigma
	\pd_x\bk{\sigma \pd_x u_n}-\abs{\pd_x 
	\bk{\bm{\Pi}_n\bk{\sigma\pd_x u_n}}}^2\big)\, \d x \,\d t
	\\ & \quad\qquad
	-\int_\T \big(\sigma\, u_n \pd_x u_n
	-\sigma \,\pd_x u_n \pd_x^2 u_n\big)\,\d x \, \d W.
\end{align*}

We again use Bessel's inequality to eliminate 
the two projection operators (remembering that 
projection commutes with differentiation).
Then, integrating in time, rasing both sides to the power $p/2$,
and taking expectation, we find
\begin{align*}
	& \frac1{2^{p/2} }\Ex \sup_{t \in [0,T]}\norm{u_n}_{H^1(\T)}^p 
	+\ep^{p/2} \Ex \bk{\int_0^T  \int_\T  \abs{\pd_x u_n}^2
	+\abs{\pd_x^2 u_n}^2\,\d x \, \d t }^{p/2}
	\\ & \qquad \lesssim_p 
	\Ex \sup_{t \in [0,T]}\bigg|\int_0^t 
	\int_\T \sigma u_n\pd_x \bk{\sigma \pd_x u_n}
	+\abs{ \bk{\sigma \pd_x u_n}}^2 \,\d x\,\d t
	\\ &\qquad \quad \qquad 
	-\int_0^t \int_\T \pd_x^2 u_n\sigma
	\pd_x\bk{\sigma \pd_x u_n}-\abs{\pd_x 
	\bk{\sigma\pd_x u_n}}^2\, \d x \,\d t\bigg|^{p/2}
	\\ & \quad\qquad
	+\Ex \sup_{t \in [0,T]}\abs{\int_0^t \int_\T \sigma\, u_n \pd_x u_n
	-\sigma \,\pd_x u_n \pd_x^2 u_n\,\d x \, \d W}^{p/2} =: I_1 + I_2.
\end{align*}

For $I_1$, we find:
\begin{align*}
	&\int_\T \sigma u_n\pd_x \bk{\sigma \pd_x u_n}
	+\abs{  \bk{\sigma \pd_x u_n}}^2 \,\d x\\
	& =  \int_\T - \abs{\sigma\,\pd_x u_n}^2 
	+\sigma  u_n\, \pd_x \sigma \,\pd_x  u_n
	+ \abs{  \bk{\sigma \pd_x u_n}}^2\,\d x\\
	& = \int_\T -\frac12 \pd_x\bk{\sigma  \, \pd_x \sigma} \,  u_n^2\,\d x,
\end{align*}
and, similarly,
\begin{align*}
	&\int_\T  \pd_x^2 u_n\sigma
	\pd_x\bk{\sigma \pd_x u_n}-\abs{\pd_x 
	\bk{ \sigma\pd_x u_n}}^2\,\d x\\
	& = \int_\T  \abs{\pd_x\bk{\sigma \pd_x u_n}}^2
	- \pd_x \sigma \,\pd_x u_n \,\pd_x\bk{ \sigma \,\pd_x u_n}
	-\abs{ \pd_x \bk{\sigma\pd_x u_n}}^2\,\d x\\
	& =  \int_\T  \bk{ \frac12  \pd_x \bk{\sigma \,\pd_x \sigma} 
	- \abs{\pd_x \sigma }^2 }\abs{\pd_x u_n}^2\,\d x.
\end{align*}

This give us 
$$
I_1 \lesssim_\sigma  \int_0^T \Ex  \norm{u_n}_{H^1(\T)}^p\,\d t,
$$
where the implied constant depends on $\norm{\sigma}_{W^{2,\infty}(\T)}$.

For $I_2$, by the Burkholder--Davis--Gundy inequality,
we have
\begin{align*}
	I_2 
	&\le \Ex \bk{\int_0^T \abs{\int_\T \sigma\, u_n \pd_x u_n
	-\sigma \,\pd_x u_n \pd_x^2 u_n\,\d x }^2\, \d t}^{p/4}\\
	&\le \Ex  \bk{\int_0^T \abs{\int_\T \frac12 \pd_x \sigma\, 
	\bk{\abs{u_n}^2 - \abs{\pd_x u_n}^2}\,\d x }^2 \d t}^{p/4}\\
	& \lesssim_\sigma   \int_0^T \Ex\norm{u_n }_{H^1(\T)}^{p}\,\d t,
\end{align*}
where we used the convexity of $x \mapsto x^{p/4}$ 
in the final inequality, provided by the assumption $p > 4$.

The estimates on $I_1$ and $I_2$ then allow us to derive 
the stated bound in the Lemma statement by a standard
application of Gronwall's inequality.
\end{proof}

\begin{rem}[Full Euler--Poincar\'e structure 
in the noise]\label{rem:EPnoise1}
It can be verified that there is no additional 
difficulty with the incorporation of full 
Euler--Poincar\'e noise of the form
\begin{align*}
	 \mathcal{B}(u) \circ\d W &=\bk{\sigma\,\pd_x u
	+\mathcal{J}_1(u)} \circ \,\d W,  \\
	\mathcal{J}_1(u)&:= K*\tilde{B}(u):= K*\bk{2 \pd_x \sigma \, u 
	+\pd_x^2 \sigma \,\pd_x u},
\end{align*}
in place of $\sigma \pd_x {u} \circ \d W$ 
in \eqref{eq:u_ch_ep}, see \eqref{eq:sCH_1}.

Written out in It\^o form, the noise 
$\mathcal{B}(u)\circ \d W = \mathcal{B}(u)\,\d W
+\frac12 \mathcal{C}(u)\,\d t$ gives a Stratonovich--It\^o 
correction $\mathcal{C}$ of the form
\begin{align*}
	 2\mathcal{C}(u)&= \bigl\LL\mathcal{B}(u),W\bigr\RR 
	=-\mathcal{B} \left(\mathcal{B}(u)\right)
	=-\sigma \,\pd_x\left( \sigma\, \pd_x u \right)
	+2 \mathcal{J}_2(u), 
	\\ 
	 2\mathcal{J}_2(u)&= - \sigma \pd_x K*\tilde{B}(u)
	-2 K*\left(\pd_x \sigma\, \left( \sigma \pd_x u 
	+K*\tilde{B}(u)\right)  \right)
	\\ & \quad
	-K*\left(\pd_x^2 \sigma \pd_x \left( \sigma \pd_x u
	+K*\tilde{B}(u)\right)  \right).
\end{align*}
Since the transformation $\sigma \pd_x u \mapsto 
\mathcal{B}(u)=\sigma \pd_x u+K*\tilde{B}(u)$ does not 
introduce higher-order derivatives on $u$, but 
it does on $\sigma$, the only extra requirement 
for bounds on $\mathcal{B}$ or $\mathcal{C}$ 
is that $\sigma \in W^{3,\infty}(\T)$ 
instead of $W^{2,\infty}(\T)$. 

The corresponding energy balance is
\begin{align*}
	\frac12\d \norm{{u}}_{H^1(\T)}^2&
	+\ep \norm{\pd_x{u}}_{H^1(\T)}^2\,\d t 
	\\ & \qquad 
	= I_2 \,\d t+\int_\T \bigl(u\mathcal{J}_2({u})
	+\pd_x u \pd_x \mathcal{J}_2(u)\bigr)\,\d x  \, \d t 
	\\ &\qquad \quad
	+ \frac12 \int_\T\big( \abs{\mathcal{J}_1(u)}^2
	+\abs{\pd_x \mathcal{J}_1(u)}^2\big)\,\d x\,\d t
	\\ &\qquad\quad
	+ \int_\T \big(\mathcal{J}_1({u}) \sigma \pd_x u
	+\pd_x \mathcal{J}_1(u)\pd_x
	\bk{\sigma \,\pd_x u}\big) \,\d x\,\d t  
	\\ &\qquad\quad
	+\bk{I_3- \int_\T u \bk{2 \pd_x \sigma u
	+\pd_x^2 \sigma \pd_x u}\,\d x}\,\d W,
\end{align*}
where 
\begin{align*}
	& I_2 := \frac12 \int_\T \big(\sigma u \pd_x\bk{\sigma \pd_x u}
	+\abs{\sigma \pd_x u}^2\big)\, \d x
	\\ &\qquad
	-\frac12 \int_\T \big(\pd_x^2 u\,
	\sigma \pd_x\bk{\sigma \pd_x u}
	-\abs{\pd_x \bk{\sigma \pd_x u}}^2\big)\, \d x,
	\\ & 
	I_3 :=-\int_\T \big(\sigma u \pd_x u 
	-\sigma \pd_x u \pd_x^2 u\big) \,\d x
\end{align*}
are as in \eqref{eq:H1_intermediates} for ready comparison.
\end{rem}

\section{Tightness of probability laws}\label{sec:laws}

We will prove that the probability laws 
$\seq{(u_n)_*\mathbb{P}}$ of the Galerkin 
approximations $\seq{u_n}$ are ($n$-uniformly) tight, 
on suitable Polish and quasi-Polish spaces.  We will later 
construct weak solutions by applying a 
stochastic compactness argument. In one step of the 
argument, one makes use of tightness, which is 
linked to weak compactness of the laws. 
In contrast to the results in Section \ref{sec:apriori}, 
tightness results are not uniform-in-$\ep$.

\subsection{Tightness on $L^2([0,T];H^1(\T))$ and on $ 
C([0,T];H^1_w(\T))$}\label{sec:tightness_final}

We will first show improved temporal regularity 
in $L^2(\T)$. This will be used to establish the 
tightness of laws of $\seq{u_n}$ on the Polish space 
$L^2([0,T];H^1(\T))$ and on the quasi-Polish 
space $C([0,T];H^1_w(\T))$. 

\begin{lem}[Temporal $L^2$ continuity]\label{thm:u_Ctheta_L2}
For each $n \in \N$, let $u_n$ be a solution 
to \eqref{eq:galerkin_n} with $\Ex \norm{u_0}_{H^1(\T)}^p
<\infty$ for $p>2$. 
For any $\theta \in [0,(p-2)/4p)$, there 
exists a constant 
$$
C=C\bk{T, p,\theta,\Ex \norm{u_0}_{H^1(\T)}^{4p}, 
\norm{\sigma}_{W^{2,\infty}(\T)}},
$$
independent of $n$ and $\ep$, such that
\begin{equation}\label{eq:galerkin_Ctheta}
	\Ex \norm{u_{n}}_{C^\theta([0,T];L^2(\T))}^{2p}
	\le C.
\end{equation}
\end{lem}

\begin{proof}
We shall estimate $\Ex\norm{u_n(t)-u_n(s)}_{L^2(\T)}^{2p}$ 
in terms of $|t-s|^{1+\gamma}$ for some $\gamma  >1$, 
and then appeal to Kolmogorov's continuity criterion. 

First, we separate the spatial integral as 
$$
\int_\T \bigl(u_n(t) - u_n(s)\bigr)^2\,\d x 
=\int_\T \bigl(u_n(t)-u_n(s)\bigr)
\int_s^t \, \d u_n(r)\,\d x
=\sum_{i=1}^5 I_i^n,
$$ 
where
\begin{align*}
	I_1^n & :=-\int_\T \bigl(u_n(t)-u_n(s)\bigr)
	\int_s^t u_n(r) \pd_x u_n(r) \,\d r \,\d x,
	\\ 
	I_2^n &:=  -\int_\T \bigl(u_n(t)-u_n(s)\bigr)
	\int_s^t \pd_x P[u_n](r) \,\d r \,\d x,
	\\
	I_3^n & := 
	\ep\int_\T \bigl(u_n(t)-u_n(s)\bigr)
	\int_s^t \pd_x^2 u_n(r)\,\d r \, \d x,
	\\
	I_4^n & := 
	\frac12 \int_\T \bigl(u_n(t)-u_n(s)\bigr)
	\int_s^t \sigma(r) \pd_x\bk{\sigma \pd_x u}(r) \,\d r \, \d x,
	\\
	I_5^n & := \int_\T\big(u_n(t)-u_n(s)\bigr) 
	\int_s^t \sigma(r)\pd_x u_n(r)\,\d W(r)\,\d x.
\end{align*}
After an integration by parts involving 
$u_n \pd_x u_n=\pd_x \left (u_n^2/2\right)$, and using the bound 
 $\norm{\pd_x u_n}_{L^\infty([0,T];L^1(\T))} 
\le C\norm{u_n}_{L^\infty([0,T];H^1(\T))}$,
\begin{align*}
	\Ex \abs{I_1^n}^p & 
	\le \Ex \left(\norm{u_n}_{L^\infty([0,T]\times\T)}^{2p}
	\norm{\pd_x u_n}_{L^\infty([0,T];L^1(\T))}^p
	\right)\abs{t-s}^p
	\\ & 
	\le \left(\Ex\norm{u_n}_{L^\infty([0,T]\times\T)}^{4p}
	\right)^{1/2}
	\left(\Ex \norm{u_n}_{L^\infty([0,T];H^1(\T))}^{2p}
	\right)^{1/2}\abs{t-s}^p.
\end{align*}
We conclude, given the higher moment 
bounds \eqref{eq:galerkin_unifbd2} 
and \eqref{eq:Linfty-est}, that $\Ex\, \abs{I_1^n}^p 
\le C \abs{t-s}^p$.

Recalling that $\pd_x P[u_n]=
\pd_x K*\bk{u_n^2 +\frac12\bk{\pd_x u_n}^2}$ 
and using Young's convolution inequality, we obtain
\begin{align*}
	\norm{\pd_x P[u_n]}_{L^2(\T)} 
	& \le \norm{\pd_x K}_{L^2(\T)}
	\norm{u_n^2+\frac12 \bk{\pd_x u_n}^2}_{L^1(\T)}
	\\ & 
	\le \norm{\pd_x K}_{L^2(\T)}\norm{u_n}_{H^1(\T)}^2,
\end{align*}
and therefore
\begin{align*}
	\Ex \abs{I_2^n}^p
	& \le \left(2\norm{\pd_x K}_{L^2(\T)}\right)^p 
	\Ex \left(\norm{u_n}_{L^\infty([0,T]\times\T))}^p
	\norm{u_n}_{L^\infty([0,T];H^1(\T))}^{2p}\right) 
	\abs{t-s}^p
	\\ & 
	\le {C} \left(\Ex 
	\norm{u_n}_{L^\infty([0,T]\times\T)}^{2p}
	\right)^{1/2}
	\left(\Ex \norm{u_n}_{L^\infty([0,T];H^1(\T))}^{4p}
	\right)^{1/2}\abs{t-s}^p.
\end{align*}
Again making use of \eqref{eq:galerkin_unifbd2} 
and \eqref{eq:Linfty-est}, we arrive 
at $\Ex\,\abs{I_2^n}^p\le C \abs{t-s}^p$.

Similarly, after integration by parts, we obtain
$$
\Ex \abs{I_3^n}^p \le \left(2\ep\right)^p 
\Ex\norm{u_n}_{L^\infty([0,T];H^1(\T))}^{2p}
\abs{t-s}^p \le C \abs{t-s}^p
$$
and, noting that
\begin{align*}
	&\pd_x\bigl(\sigma \bk{u_n(t)-u_n(s)}\bigr)
	\sigma \pd_x u(r)
	\\ & \quad
	=\sigma \pd_x\sigma \bk{u_n(t)-u_n(s)}\pd_x u(r) 
	+\sigma^2 \pd_x\bk{u_n(t)-u_n(s)}\pd_x u(r)
\end{align*}
generates terms of the type handled before, 
$\Ex\,\abs{I_4^n}^p \le C \abs{t-s}^p$.

Finally, we estimate the stochastic term $I_5^n$. 
We cannot exchange the temporal and 
spatial integrals because the stochastic process 
$$
(\omega,t)\mapsto \int_\T \sigma 
\bk{u_n(t)-u_n(s)}\pd_x u_n(r) \,\d x
$$
is not $\mathcal{F}_r$-measurable, so we will instead 
estimate using the Cauchy--Schwarz inequality 
repeatedly and then the 
Burkholder--Davis--Gundy inequality. 
We have  that $\norm{\pd_x\bigl(\sigma 
\bk{u_n(t)-u_n(s)}\bigr)}_{L^2(\T)}
\le C_\sigma \norm{u_n}_{L^\infty([0,T];H^1(\T)))}$.
After an integration by parts,
\begin{align*}
	\Ex \abs{I_5^n}^p 
	 & 
	\le \Ex\, \abs{\,\norm{\pd_x
	\bigl(\sigma \bk{u_n(t)-u_n(s)}\bigr)}_{L^2(\T)}
	\bk{\int_\T\abs{\int_s^t u_n
	\,\d W(r)}^2\,\d x}^{1/2}\,}^p
	\\ & 
	\le \tilde{C}_{\sigma,p} 
	\Ex \left(\, \norm{u_n}_{L^\infty([0,T];H^1(\T))}^p
	\bk{\int_\T \abs{\int_s^t u_n\,\d W(r)}^2
	\,\d x\,}^{p/2}\,\right)
	\\ & 
	\le \tilde{C}_{\sigma,p}\left(\Ex 
	\norm{u_n}_{L^\infty([0,T];H^1(\T))}^{2p}\right)^{1/2}
	\left(\Ex \bk{\int_\T 
	\abs{\int_s^t u_n\,\d W(r)}^2\,\d x}^{p}\right)^{1/2}
	\\ &  
	\le C_{\sigma,p}\left(\int_\T\Ex 
	\abs{\int_s^t u_n\,\d W(r)}^{2p}\,\d x\right)^{1/2},
\end{align*}
where the final inequality is the result 
of \eqref{eq:galerkin_unifbd2} and Jensen's inequality. 
Finally, by \eqref{eq:galerkin_unifbd2} 
and the Burkholder--Davis--Gundy inequality, 
pointwise in $x$, 
\begin{align*}
	 \left(\int_\T \Ex \abs{\int_s^t u_n
	\,\d W(r)}^{2p} \,\d x\right)^{1/2} 
	&\le C_p \left(\Ex\int_\T 
	\bk{\int_s^tu_n^2\,\d r}^p\, \d x\right)^{1/2}
	\\ & 
	\le C \bk{\Ex
	\norm{u_n}_{L^\infty([0,T]\times\T)}^{2p}}^{1/2}
	\abs{t-s}^{p/2} \\ \, & \overset{\eqref{eq:Linfty-est}}{\le} 
	C\abs{t-s}^{p/2}.
\end{align*}

Summarising, we have obtained
$$
\Ex \norm{u_n(t)-u_n(s)}_{L^2(\T)}^{2p}
\le C\abs{t-s}^{p/2}=C\abs{t-s}^{1+(p-2)/2},
$$
where the constant $C$ is independent of $n$ and $\ep$. 
By Kolmogorov's continuity criterion, there is a 
version of $u_n$ in $C^{\theta}([0,T];L^2(\T))$, 
for any $\theta\in [0,(2 - p)/4p )$, and a bound of 
the form \eqref{eq:galerkin_Ctheta}.
\end{proof}

\begin{rem}
The temporal continuity bound \eqref{eq:galerkin_Ctheta} 
can also be carried out with respect to a fractional 
Sobolev norm via a computation 
following,  e.g.,  \cite[Lemma 2.1]{FG1995}.
\end{rem}

\begin{lem}[Tightness on {$C([0,T];H^1_w(\T))$}]
\label{rem:Ctheta_tightness}
For each $n \in \N$, let $u_{n}$ be a solution to 
\eqref{eq:galerkin_n} with $\Ex \norm{u_0}_{H^1(\T)}^p 
< \infty$ for $p > 2$. The laws of 
$\seq{u_n}$ are tight on $C([0,T];H^1_w(\T))$.
\end{lem}

\begin{proof} Choose $\theta \in (0, (2 - p)/4p)$. 
Given the compact embedding \cite[Cor.~B.2]{Ond2010} 
$$
L^\infty([0,T];H^1(\T)) \cap C^\theta([0,T];L^2(\T)) \doublehookrightarrow C([0,T];H^1_w(\T)),
$$
the laws of $\seq{u_n}$ are tight on $C([0,T];H^1_w(\T))$ 
via the following standard computation: 
By the compact embedding, the sets
$$
\mathcal{K}_R:=\seq{u \in C([0,T];H^1_w(\T)):
\norm{u}_{L^\infty([0,T];H^1(\T))}
+\norm{u}_{C^\theta([0,T];L^2(\T))}\le R}
$$
are compact in $\mathcal{X}:=C([0,T];H^1_w(\T))$. 
Therefore, by Markov's inequality,
\begin{align*}
	\bk{u_n}_*\mathbb{P}
	\bigl(\mathcal{X} \backslash \mathcal{K}_R\bigr)
	\le \frac1R \Ex\norm{u_n}_{L^\infty([0,T];H^1(\T))}
	+\frac1R \Ex\norm{u_n}_{C^\theta([0,T];L^2(\T))}.
\end{align*}
By \eqref{eq:galerkin_unifbd1} and \eqref{eq:galerkin_Ctheta}, 
the right-hand side tends to zero as $R \to \infty$.
\end{proof}

We need the following variant of the 
Aubin--Lions lemma \cite[Thm.~2.1]{FG1995} 
(see also \cite[Sec.~13.3]{Temam:LN}) to establish 
tightness on $L^2([0,T];H^1(\T))$.

\begin{lem}\label{thm:temam_flandoli_gatarek}
Let $B_0 \subseteq B \subseteq B_1$ be 
Banach spaces, $B_0$ and $B_1$ reflexive, 
with compact embedding of $B_0$ in $B$. 
Fix $p \in (1,\infty)$ and $\alpha\in (0,1)$. Let
$$
\mathcal{Y} = L^p([0,T];B_0) \cap W^{\alpha,p}([0, T];B_1),
$$
be endowed with the natural norm. 
The embedding of $\mathcal{Y}$ in $L^p([0,T];B)$ is compact.
\end{lem}

Tightness of probability measures is related to the 
stochastic boundedness of random variables. Below 
we prove that $u_n\in_{\rm sb} L^2_tH^2_x\cap 
W^{\theta,2}_tL^2_x$, uniformly in $n$, for 
some $\theta< (2 - p)/4p$, making essential use of 
the dissipation part of \eqref{eq:galerkin_unifbd1}.

\begin{lem}[Tightness on {$L^2([0,T];H^1(\T))$}]
\label{thm:stochastic_boundedness}
For each $n \in \N$, let $u_{n}$ be a solution 
to \eqref{eq:galerkin_n} with $\Ex \norm{u_0}_{H^1(\T)}^p 
< \infty$ for $p > 2$. Let $\theta' \in (0, (2 - p)/4p)$. 
The following stochastic boundedness estimate holds 
uniformly in $n$: 
\begin{equation}\label{eq:stoch-bound-H2}
	\lim_{M \to \infty}\mathbb{P}
	\bk{\norm{u_n}_{ L^2([0,T]; H^2(\T))
	\cap W^{\theta',2}([0,T];L^2(\T))}>M} = 0.
\end{equation}
Moreover, the laws of $\seq{u_n}$ 
are tight on $L^2([0,T];H^1(\T))$.
\end{lem}

\begin{proof}
A natural norm on $L^2([0,T]; H^2(\T))
\cap W^{\theta',2}([0, T]; L^2(\T))$ is
$$
\norm{u_n}_{ L^2([0,T]; H^2(\T)) 
\cap W^{\theta',2}([0,T]; L^2)}
=\norm{u_n}_{ L^2([0,T]; H^2(\T))}
+\norm{u_n}_{W^{\theta', 2}([0, T]; L^2(\T))}.
$$

For $\theta \in (\theta', (2 - p)/4p)$, the embeddings 
$C^{\theta}([0,T];B_1) \hookrightarrow 
C^{\theta'}([0,T]; B_1) \hookrightarrow 
W^{\theta',2}([0,T];B_1)$ are continuous. 
Using Markov's inequality and 
\eqref{eq:galerkin_Ctheta} with $\theta=1/5$,
$$
\prob\bk{\seq{\norm{u_n}_{W^{\theta',2}([0,T];L^2(\T))} 
>M}} \le \frac1M \,\Ex \norm{u_n}_{C^\theta([0,T];L^2(\T))} 
\lesssim \frac1M.
$$
Next, using the energy estimate of 
Proposition \ref{thm:galerkin_H1}, we obtain
$$
\prob\bk{\seq{\norm{u_n}_{L^2([0,T];H^2(\T))}>M}}
\le \frac1{M^2} \, \Ex\norm{u_n}_{L^2([0,T];H^2(\T))}^2 
\lesssim \frac{1}{\ep M^2}.
$$
In view of the natural norm for the intersection space, 
this implies \eqref{eq:stoch-bound-H2}. 

Tightness of the laws on $L^2([0,T];H^1(\T))$ 
now follows from Lemma \ref{thm:temam_flandoli_gatarek} 
and the stochastic boundedness estimate 
\eqref{eq:stoch-bound-H2}. In particular, for each 
$\delta>0$, there exists a number $M>0$ and a compact set
\begin{align*}
	\mathcal{A}_M 
	& = \Bigl\{v \in L^2([0,T];H^2(\T))
	\cap W^{\theta',2}([0,T];L^2(\T)):
	\\ & \qquad \qquad
	\norm{v}_{L^2([0,T];H^2(\T))}+
	\norm{v}_{W^{\theta', 2}([0, T]; L^2(\T))}
	\le M\Bigr\},
\end{align*}
such that the complement $\mathcal{A}^c_M$ 
satisfies $\bk{u_n}_*\prob\bigl(\mathcal{A}^c_M\bigr)<\delta$.
\end{proof}

\section{Weak (martingale) solutions}\label{sec:existence1}

To be able to pass to the limit in the nonlinear terms 
in the SPDE \eqref{eq:u_ch_ep}, we must show that 
the Galerkin approximations $\seq{u_n}$ converge 
strongly in $(\omega,t,x)$. Setting aside the probability 
variable $\omega$, strong $(t,x)$ convergence is linked 
to the spatial and temporal a priori estimates established 
in Section  \ref{sec:apriori} and Section  \ref{sec:laws}. 
On the other hand, the available estimates only 
ensure weak convergence in $\omega$. To rectify this 
unfortunate (but typical) situation, we will replace 
the random variables $\seq{u_n}$ by 
Skorokhod--Jakubowski a.s.~representations 
$\seq{\tilde{u}_n}$, which are defined on a new 
stochastic basis and will converge almost surely. 
The existence of $\seq{\tilde{u}_n}$ will follow 
from the tightness estimates established 
in Section \ref{sec:laws}. Finally, we will show that 
the strong limit of $\seq{\tilde{u}_n}$ constitutes 
a weak (martingale) solution according to 
Definition~\ref{def:wk_sol}. 

\subsection{Skorokhod--Jakubowski a.s.~representations}

Introduce the path spaces
\begin{equation}\label{eq:pathspaces}
	\begin{aligned}
		 \mathcal{X}_{u,s}&:=L^2([0,T];H^1(\T))\\
		 \mathcal{X}_{u,w}&:= C([0,T];H^1_w(\T)),
		\\  
		\mathcal{X}_W&:= C([0,T]), \\
		\mathcal{X}_0&:= H^1(\T),
	\end{aligned}
\end{equation}
and set $\mathcal{X}:=\mathcal{X}_{u,s} 
\times \mathcal{X}_{u,w} \times 
\mathcal{X}_W \times  \mathcal{X}_0$.  
Denote by $\mu^n$ the (joint) law of the 
$\bigl(\mathcal{X},\mathcal{B}(\mathcal{X})\bigr)$-valued 
random variable $\bigl(u_{n,s}, u_{n,w} ,W,\bm{\Pi}_n u_0\bigr)$.
We denote by $\mu_{u,s}^n$, $\mu_{u,w}^n$, $\mu_W^n$ and $\mu_0^n$ 
the laws of $u_n$, $u_n$, $W$ and $\bm{\Pi}_n u_0$ 
on $\mathcal{X}_{u,s}$, $\mathcal{X}_{u,w}$, $\mathcal{X}_W$ 
and $\mathcal{X}_0$, respectively.  
(The subscripts ``$s$'' and ``$w$'' 
refer to the ``strong'' and ``weak'' topologies used in the 
subscripted path spaces and laws defined on them.)

Note carefully that we have used two 
copies of $u_n$ in separate spaces 
$\mathcal{X}_{u,s}$ and $\mathcal{X}_{u,w}$
that do not inject continuously into one another. 
The aim of this manoeuvre is to ensure convergence in two separate 
topologies of the Skorokhod--Jakubowski 
representations of $u_n$. The rationale for this is 
explained in the appendix following Theorem \ref{thm:skorokhod}.
The two variables are identified {\em post hoc} 
in Lemma \ref{thm:identification_lemma}.

\begin{lem}[Tightness of Galerkin approximations]
\label{thm:jointlaw_tightness}
The laws $\seq{\mu^n}$ are tight.
\end{lem}

\begin{proof}
The tightness on $\mathcal{X}$ of the product measures 
$\seq{\mu_{u,s}^n\otimes \mu_{u,w}^n\otimes  \mu_W^n \otimes\mu_0^n}$ 
implies the tightness of the joint 
laws $\seq{\mu^n}$ on $\mathcal{X}$. The tightness of 
$\seq{\mu_{u,i}^n}$ on $\mathcal{X}_{u,i}$ for $i = 1, 2$ are stated 
in Lemmas \ref{rem:Ctheta_tightness} 
and \ref{thm:stochastic_boundedness}. 
Since $\bm{\Pi}_n u_0 \to u_0$ in $H^1(\T)$, the 
laws $\seq{\mu_0^n}$ are tight on $H^1(\T)$. 
The elements of $\seq{\mu_W^n}$ do not change 
with $n$, each $\mu_W^n$ is equal to the 
law of the Wiener process $W$ (which is tight 
on $\mathcal{X}_W$). Hence, the tightness of the 
product measures $\seq{\mu_{u,s}^n\otimes \mu_{u,w}^n\otimes \mu_W^n 
\otimes\mu_0^n}$ follows.
\end{proof}

\begin{thm}[Skorokhod--Jakubowski representations]\label{thm:skorohod}
There exist a probability space 
$\bigl(\tilde{\Omega},\tilde{\mathscr{F}},
\tilde{\mathbb{P}}\bigr)$ and $\mathcal{X}$-valued 
variables $\bigl\{\bigl(\tilde{u}_{n,s}, \tilde{u}_{n,w}, \tilde{W}_n, 
\tilde{u}_{0,n}\bigr)\bigr\}_{n\in\N}$,
$\bigl(\tilde{u}_s, \tilde{u}_w,\tilde{W}, \tilde{u}_0\bigr)$, defined 
on $\bigl(\tilde{\Omega},\tilde{\mathscr{F}},
\tilde{\mathbb{P}}\bigr)$, such that along a 
subsequence (not relabelled),
\begin{equation}\label{eq:laws-sim}
	\tilde{u}_{n,s} \sim u_n,  
	\quad 
	\tilde{u}_{n,w} \sim u_n,  
	\quad 
	\tilde{W}_n \sim W, 
	\quad
	\tilde{u}_{0,n} \sim \bm{\Pi}_n u_0
\end{equation}
and, $\tilde{\mathbb{P}}$-almost surely, 
$$
\bigl(\tilde{u}_{n,s}, \tilde{u}_{n,w},\tilde{W}_n, 
\tilde{u}_{0,n}\bigr) \ton 
\bigl(\tilde{u}_s,\tilde{u}_w,\tilde{W}, \tilde{u}_0\bigr) 
\quad \mbox{in $\mathcal{X}$}.
$$
\end{thm}

\begin{proof}
Apply Theorem \ref{thm:skorokhod}.
\end{proof}

\begin{rem}
We need Jakubowski's version \cite{Jak1998} 
of the Skorokhod representation theorem 
because our path space $\mathcal{X}$ contains
the non-metrisable (but quasi-Polish) 
space $C([0,T];H^1_w(\T))$.
\end{rem}

\begin{lem}[identification of doubled variables]
\label{thm:identification_lemma}
For the sequence of variables defined in \eqref{eq:laws-sim}, 
$\tilde{u}_{n,s} = \tilde{u}_{n,w}$, 
$\tilde{\mathbb{P}}\otimes \d t \otimes \d x$--a.e. 
Moreover, $\tilde{u}_s=\tilde{u}_w$, 
$\tilde{\mathbb{P}}\otimes \d t \otimes \d x$--a.e.
\end{lem}

\begin{rem}\label{rem:identification_remark}
It is then henceforward sufficient to speak of 
$\tilde{u}_n:=\tilde{u}_{n,s} = \tilde{u}_{n,w}$ 
and $\tilde{u}:=\tilde{u}_s = \tilde{u}_w$. 
\end{rem}

\begin{proof}
For a fixed $n$, this follows directly from \cite[Lemma 1]{Jak1998},
where an identification was made for variables in two Polish spaces. 
However, the completeness and separability of the path spaces 
were not used in the proof, and the lemma can be proven unchanged 
for quasi-Polish spaces.

For any $\varphi \in C^\infty([0,T]\times \T)$, 
$\eta \in  L^\infty(\tilde{\Omega})$, as $n\to \infty$,
\begin{align*}
	\tilde \Ex \int_0^T \int_\T  \eta \varphi \tilde{u}_{n,s}
	\,\d x\,\d t 
	& \to \tilde \Ex \int_0^T \int_\T \eta \varphi \tilde{u}_s
	\,\d x \,d t,
	\\ 
	\tilde \Ex \int_0^T \int_\T  \eta \varphi \tilde{u}_{n,w}
	\,\d x\,\d t
	& \to \tilde \Ex \int_0^T \int_\T \eta \varphi \tilde{u}_w
	\,\d x \,d t,
\end{align*}
and, since $\tilde{u}_{n,s} = \tilde{u}_{n,w}$, 
$$
\tilde \Ex \bk{\eta  \int_0^T \int_\T \varphi 
\tilde{u}_s\,\d x \,d t} 
= \tilde \Ex \bk{\eta \int_0^T \int_\T \varphi 
\tilde{u}_w\,\d x \,d t}.
$$
From this it easily follows that $\tilde{u}_s = \tilde{u}_w$, 
$\tilde{\mathbb{P}}\otimes \d t \otimes \d x$-a.e.
\end{proof}

With $t \in [0,T]$ and $X$ denoting $L^2([0,T];H^1(\T))$, 
$C([0,T];H^1_w(\T))$ or $C([0,T])$, let 
$f \mapsto f|_{[0,t]}:X \to X|_{[0,t]}$ denote 
the restriction operator to $[0,t]$. 
We define $\{\tilde{\mathcal{F}}_t\}_{t\ge 0}$ to 
be the $\tilde{\prob}$-augmented 
canonical filtration of $\bigl(\tilde{u},\tilde{W},
\tilde{u}_0\bigr)$, i.e.,
$$
\tilde{\mathcal{F}}_t := {\it \Sigma}\bk{
{\it \Sigma}\bigl(\tilde{u}\big|_{[0,t]},
\tilde{W}\big|_{[0,t]},\tilde{u}_0\bigr)
\cup \bigl\{N \in \tilde{\mathcal{F}}:
\tilde{\prob}(N)=0\bigr\}}.
$$
where, for a collection $E$ 
of subsets of $\tilde{\Omega}$, ${\it \Sigma}(E)$ 
denotes the smallest sigma algebra containing $E$.  
Denote by $\tilde{\mathcal{S}}$ the corresponding 
stochastic basis, that is,
\begin{equation}\label{eq:stoch-basis-tilde}
	\tilde{\mathcal{S}}:=
	\bigl(\tilde{\Omega},\tilde{\mathcal{F}}, 
	\{\tilde{\mathcal{F}}_t\}_{t\ge 0},
	\tilde{\mathbb{P}}\bigr).
\end{equation}
Similarly, based on $\bigl(\tilde{u}_n,\tilde{W}_n,
\tilde{u}_{0,n}\bigr)$, we define $\tilde{\mathcal{S}}_n
=\bigl(\tilde{\Omega},\tilde{\mathcal{F}},
\{\tilde{\mathcal{F}}_t^n\}_{t\ge 0},
\tilde{\prob}\bigr)$. Then $\tilde{u}$, $\tilde{W}$ 
and $\tilde{u}_n$, $\tilde{W}_n$ are adapted relative 
to the stochastic bases $\tilde{\mathcal{S}}$, 
$\tilde{\mathcal{S}}_n$, respectively. 
Besides, by the equality of laws, $\tilde{W}_n$ 
is a Brownian motion on $\tilde{\mathcal{S}}_n$, 
and we have the following result.

\begin{lem}[Brownian motion]\label{thm:W_BM}
The process $\tilde{W}$ is a Brownian motion 
on $\tilde{\mathcal{S}}$.
\end{lem}

\begin{proof}
The proof is standard (see, e.g., \cite[Lemma 4.8]{DHV2016}), relying on L\'evy's 
characterisation theorem (e.g., \cite[Thm.~ IV.3.6]{Revuz:1999wi}) 
and the equality of laws. The claim follows if we 
establishes that $\tilde{W}$ is a martingale 
relative to $\tilde{\mathcal{S}}$.

By the equivalence of laws, for $0 \le s \le t \le T$, 
\begin{align*}
	\tilde{\Ex}& 
	\bk{\bigl(\tilde{W}_n(t)-\tilde{W}_n(s)\bigr)
	\, \gamma\bigl(\tilde{u}_n\big|_{[0,s]},\tilde{u}_n\big|_{[0,s]},
	\tilde{W}_n\big|_{[0,s]}\bigr)}
	\\ & 
	=\Ex\bk{\bigl(W(t)-W(s)\bigr)
	\, \gamma\bigl(u_n\bigl|_{[0,s]},\tilde{u}_n\big|_{[0,s]},
	W\big|_{[0,s]}\bigr)}=0,
\end{align*}
because $W$ is a martingale relative to 
$\mathcal{S}=\bigl(\Omega,\mathcal{F}, 
\{\mathcal{F}_t\}_{t\ge 0},\mathbb{P}\bigr)$, 
for any continuous function 
$\gamma\colon L^2([0,s];H^1(\T))\times 
 C([0,s];H^1_w(\T))\times C([0,s])\to \R$. Moreover,
$$
\sup_{n \in \mathbb{N}} \tilde{\Ex}
\abs{\tilde{W}_n(t)}^2
=\Ex \abs{W(t)}^2=t<\infty.
$$
Therefore, by the Vitali 
convergence theorem,  
$$
\tilde{\Ex} 
\bk{\bigl(\tilde{W}(t)-\tilde{W}(s)\bigr) 
\gamma\bigl(\tilde{u}\big|_{[0,s]},\tilde{u}\big|_{[0,s]},
\tilde{W}\big|_{[0,s]}\bigr)} = 0,
$$
so that $\tilde{W}$ is a martingale (and 
hence a Brownian motion) on $\tilde{\mathcal{S}}$.
\end{proof}

Next, we collect the convergence and continuity properties 
that are needed later to prove that the 
limit $\bigl(\tilde{\mathcal{S}},\tilde{u},
\tilde{W}\bigr)$ is a weak $H^m$ solution.

\begin{lem}[Convergence]\label{thm:products_convergences}
Let $u_n$, $\tilde{u}_n$ and $\tilde{u}$ be defined as 
in Theorem \ref{thm:skorohod} and 
Remark \ref{rem:identification_remark}, 
and set $q_n := \pd_x u_n$, $\tilde{q}_n := \pd_x\tilde{u}_n$ and 
$\tilde{q}:= \pd_x \tilde{u}$. Then 
$q_n \sim \tilde{q}_n$, and the following 
convergences hold $\tilde{\mathbb{P}}$-almost surely:
\begin{subequations}\label{eq:basic-conv}
	\begin{align}
		 \tilde{u}_n &\ton \tilde{u}, \quad
		\text{in $C([0,T];H^1_w(\T))$}, \label{eq:basic-convA}\\
			\tilde{q}_n &\ton \tilde{q} \quad
		\text{in $L^2([0,T]\times \T)$},\label{eq:basic-convB}
		\\ 
		\tilde{u}_n^2 &\ton \tilde{u}^2   
		\quad 
		\text{in $L^1([0,T];W^{1,1}(\T))$}, \label{eq:basic-convC}\\
		\tilde{q}_n^2 &\ton  \tilde{q}^2
		\quad \text{in $L^1([0,T]\times \T)$},\label{eq:basic-convD}
		\\ 
		\tilde{u}_n \tilde{q}_n
		&\to \tilde{u}\tilde{q} 
		\quad 
		\text{in $L^2([0,T]\times \T)$}.\label{eq:basic-convE}
	\end{align}
\end{subequations}
\end{lem}

\begin{proof}
Since $u_n \sim \tilde{u}_n$ on $L^2([0,T];H^1(\T))$, 
we have $\pd_x u_n \sim \pd_x \tilde{u}_n$ on 
$L^2([0,T]\times\T)$. Next, regarding the convergence 
claims \eqref{eq:basic-conv}, the limits \eqref{eq:basic-convA} and \eqref{eq:basic-convB}
 follow directly from Theorem \ref{thm:skorohod}. 

By the standard calculus inequality
$$
\norm{fg}_{W^{1,1}} \le \norm{f}_{H^1}\norm{g}_{L^2}
+\norm{g}_{H^1}\norm{f}_{L^2},
$$
we also have (with $f = \tilde{u}_n - \tilde{u}$ 
and $g = \tilde{u}_n + \tilde{u}$),
\begin{align*}
	&\norm{\tilde{u}_n^2-\tilde{u}^2}_{L^1([0,T]; W^{1,1}(\T))} 
	\\ &\qquad
	\le \norm{\tilde{u}_n+\tilde{u}}_{L^2([0,T];H^1(\T))} 
	\norm{\tilde{u}_n - \tilde{u}}_{L^2([0,T];H^1(\T))}
	\ton 0,
\end{align*}
$\tilde{\mathbb{P}}$-almost surely. 
Finally, by \eqref{eq:galerkin_unifbd1}, the embedding 
$H^1(\T) \hookrightarrow L^\infty (\T)$ and the 
equivalence of laws, 
\begin{align*}
	 \norm{\tilde{u}_n\tilde{q}_n
	-\tilde{u} \tilde{q}}_{L^2([0,T]\times\T)}
	&\le \norm{\tilde{u}_n-\tilde{u}}_{L^2([0,T] ;L^\infty(\T))}
	\norm{\tilde{q}_n}_{L^\infty([0,T];L^2(\T))}
	\\ &\quad 
	+\norm{\tilde{q}_n-\tilde{q}}_{L^2([0,T]\times \T)}
	\norm{\tilde{u}}_{L^\infty([0,T]\times\T)} \ton 0,
\end{align*}
$\tilde{\mathbb{P}}$-almost surely. This establishes 
\eqref{eq:basic-convC}--\eqref{eq:basic-convE}.
\end{proof}

\begin{thm}[Weak $H^1$ solution]\label{thm:existence_H1}
Suppose $\sigma \in W^{2,\infty}(\T)$, $p > 2$, and that 
$\Ex \norm{u_0}_{H^1(\T)}^p<\infty$. 
Let $\tilde{u}$, $\tilde{W}$, $\tilde{u}_0$ 
be the Skorokhod--Jakubowski a.s.~representations 
from Theorem \ref{thm:skorohod} (and 
Remark \ref{rem:identification_remark}), and let $\tilde{\mathcal{S}}$ 
be the corresponding stochastic basis 
\eqref{eq:stoch-basis-tilde}. Then 
$\bigl(\tilde{\mathcal{S}},\tilde{u},\tilde{W}\bigr)$ 
is a weak $H^1$ solution of \eqref{eq:u_ch_ep}  with 
initial law $\tilde{\Lambda}
:=\bk{\tilde{u}_0}_*\tilde{\mathbb{P}}$, substituting
 for (c) in Definition~\ref{def:wk_sol}  the following (here, $m = 1$):
\begin{itemize}
	\item[(c')]$\tilde{u}:\Omega\times[0,T]\to H^1(\T)$ is

	$$
	\tilde{u}(\omega,\dott)\in C([0,T];H^1_w(\T))
	$$ 
	for $\tilde{\mathbb{P}}$-almost every $\omega\in \Omega$. 
	Moreover, $\tilde{u} \in L^p(\tilde{\Omega};L^\infty([0,T];H^1(\T)))
	\cap L^2(\tilde{\Omega} \times [0,T]; H^2(\T))$.
\end{itemize}
\end{thm}

\begin{rem}\label{rem:wk_time_cont}
The difference between (c') and (c) of 
Definition \ref{def:wk_sol} is the weakened 
$\tilde{\mathbb{P}}$-almost sure inclusion 
$\tilde{u} \in C([0,T];H^1_w(\T))$ along with 
the lack of temporal continuity requirement in
$\tilde{u} \in L^p(\Omega ;L^\infty([0,T];H^1(\T)))$ in (c') in place of
$\tilde{u} \in L^p(\Omega;C([0,T];H^1(\T)))$ of (c). 
This continuity in $H^1(\T)$ (``strong temporal 
continuity'') is not necessary for establishing 
pathwise uniqueness in Section \ref{sec:pathwise_unique} below, 
and subsequent (stochastic) strong 
existence (Section \ref{sec:H1_existence}). 
We therefore relegate the proof of strong 
temporal continuity to Proposition~\ref{thm:st_time_cont}. 
\end{rem}

\begin{proof}

We continue to use the notations from 
Lemma \ref{thm:products_convergences}.
By the equality of joint laws on $\mathcal{X}$, see 
\eqref{eq:laws-sim}, we also have 
\begin{equation}\label{eq:laws-sim2}
	\bigl(u_n, u_n, W,\bm{\Pi}_n u_0,q_n\bigr) 
	\sim \bigl(\tilde{u}_n,\tilde{u}_n,\tilde{W}_n,\tilde{u}_{0,n},
	\tilde{q}_n\bigr)
	\quad \text{on $\mathcal{X}\times 
	L^2([0,T]\times\T)$},
\end{equation}
because $\pd_x$ is a bounded operator from 
$H^1(\T)$ to $L^2(\T)$. For each fixed 
$n\in \N$ and $\varphi \in C^1(\T)$, consider 
the function $F_{\varphi,n}:\mathcal{X} \times 
L^2([0,T]\times \T) \to \R$ defined by
\begin{equation*}
	 F_{\varphi,n}\bigl[ \bigl(\tilde{u}_n(s),
	\tilde{W}_n(s),\tilde{u}_{0,n},\tilde{q}_n(s), 
	s\in [0,t]\bigr)\bigr] 
	= I_1^n(t) + I_2^n(t) + I_3^n(t) + I_4^n(t),
\end{equation*}
where
\begin{align*}
	I_1^n(t)&:=\int_\T \tilde{u}_n(t) \varphi \,\d x
	-\int_\T \tilde{u}_{0,n} \varphi \,\d x
	-\int_0^t \int_\T 
	\ep \, \tilde{q}_n \pd_x\varphi\,\d x\,\d s,
	\\ 
	I_2^n(t) &:=\int_0^t \int_\T 
	\Bigl[ \tilde{u}_n\,\tilde{q}_n
	+\bm{\Pi}_n \pd_x P\bigl[\tilde{u}_n\bigr]\Bigr] 
	\varphi\,\d x\,\d s, 
	\\  
	I_3^n(t) &:=-\frac12 \int_0^t \int_\T 
	\sigma
	\tilde{q}_n\, \pd_x \bk{\sigma \bm{\Pi}_n\varphi}
	\,\d x \,\d s, 
	\\
	I_4^n(t)&:=\int_0^t \int_\T \sigma \pd_x \tilde{u}_n\, 
	\bm{\Pi}_n\varphi  \,\d x \,\d \tilde{W}_n.
\end{align*}
Recall that a Baire function of class $\kappa$, where 
$\kappa$ is an ordinal number, is a function that is 
the pointwise limit of Baire functions of class $\kappa-1$, 
and class $0$ Baire functions are the continuous functions. 
We have that $F_{\varphi,n}$ is a Baire function 
of class $1$. In particular, the inclusion of the 
stochastic integral in this class can be seen 
by it being the pointwise limit of temporally 
mollified approximations along the lines 
of Benssousan \cite[Sec.~4.3.5]{Bensoussan:1995aa} 
or \cite[Lemma 2.1]{DGT2011} (see Lemma \ref{lem:stoch-conv}). 
Hence, by the equivalence of joint 
laws \eqref{eq:laws-sim2}, 
we have \cite[p.~105]{Sit2005}
\begin{align*}
	& \tilde{\mathbb{P}}
	\left(\seq{F_{\varphi,n}\bigl[ \bigl(\tilde{u}_n(s),
	\tilde{W}_n(s),\tilde{u}_{0,n},\tilde{q}_n(s), 
	s\in [0,t]\bigr)\bigr]=0}\right)
	\\ & \quad 
	= \mathbb{P}\Bigl(\seq{F_{\varphi,n}
	\bigl[ \bigl(u_n(s),
	W(s),\bm{\Pi}_n u_0,q_n(s), s\in [0,t]\bigr)
	\bigr]=0}\Bigr)\overset{\eqref{eq:galerkin_n}}{=} 1.
\end{align*}

We now establish the convergence 
of $I_1^n$,\dots,$I_4^n$ separately.

\medskip
{\it 1. Convergence of $I_2^n$.}

We estimate as follows:
\begin{align*}
	& \norm{\tilde{u} \,\tilde{q}+\pd_x P[\tilde{u}]
	-\bm{\Pi}_n \Bigl(\tilde{u}_n 
	\,\tilde{q}_n+\pd_x P\bigl[\tilde{u}_n\bigr] 
	\Bigr)}_{L^2([0,T]\times \T)}
	\\ & \quad
	\le \norm{\tilde{u} \tilde{q}
	-\bm{\Pi}_n\bk{\tilde{u}_n
	\tilde{q}_n}}_{L^2([0,T]\times \T)}  
	+\norm{\pd_x P\bigl[\tilde{u}\bigr]
	-\bm{\Pi}_n \pd_x P\bigl[\tilde{u}_n
	\bigr]}_{L^2([0,T]\times \T)}
	\\ & \quad 
	\le \norm{\tilde{u} \tilde{q}
	-\bm{\Pi}_n\bk{\tilde{u} \tilde{q}}}_{L^2([0,T]\times \T)} 
	+\norm{\bm{\Pi}_n \bk{\tilde{u}\tilde{q}
	-\tilde{u}_n \tilde{q}_n}}_{L^2([0,T]\times \T)}
	\\ &\qquad  
	+\norm{\pd_x P\bigl[\tilde{u}\bigr]
	-\bm{\Pi}_n \pd_x P\bigl[\tilde{u}\bigr]}_{L^2([0,T]\times\T)}
	+\norm{\bm{\Pi}_n\bk{\pd_x P\bigl[\tilde{u}\bigr]
	-\pd_x P\bigl[\tilde{u}_n\bigr]}}_{L^2([0,T]\times\T)}
	\\
	& \quad\le \norm{\tilde{u} \tilde{q}
	-\bm{\Pi}_n\bk{\tilde{u} \tilde{q}}}_{L^2([0,T]\times \T)} 
	+\norm{\pd_x P\bigl[\tilde{u}\bigr]-\bm{\Pi}_n 
	\pd_x P\bigl[\tilde{u}\bigr]}_{L^2([0,T]\times \T)} 
	\\ &\quad\quad 
	+\norm{\tilde{u} \tilde{q}-\tilde{u}_n 
	\tilde{q}_n}_{L^2([0,T]\times \T)}
	+\norm{\pd_x P[\tilde{u}]-\pd_x 
	P[\tilde{u}_n]}_{L^2([0,T]\times \T)} \,\, \text{(Bessel's ineq.)}
	\\
	& \quad
	\le \norm{\bk{1-\bm{\Pi}_n}
	\bk{\tilde{u}\tilde{q}}}_{L^2([0,T]\times \T)}
	+\norm{\bk{1-\bm{\Pi}_n} 
	\pd_x P[\tilde{u}]}_{L^2([0,T]\times \T)}
	\\ &\quad\quad
	+\norm{\tilde{u} \tilde{q}-\tilde{u}_n
	\tilde{q}_n}_{L^2([0,T]\times \T)}
	+\norm{\pd_x K}_{L^2(\T)}
	\norm{\tilde{u}^2-\tilde{u}_n^2}_{L^1([0,T]\times \T)}
	\\ &\quad\quad
	+\frac12\norm{\pd_x K}_{L^2(\T)}
	\norm{\tilde{q}^2-\tilde{q}_n^2}_{L^1([0,T]\times \T)}
	\ton 0, \quad\text{$\tilde{\mathbb{P}}$-a.s.}, \quad \text{(by Lemma \ref{thm:products_convergences})}
\end{align*}
using the convergence $\bm{\Pi}_n \to 1$ in the operator 
norm on $L^2(\T) \to L^2(\T)$. This implies that
\begin{align}\label{eq:galerkin_skorokhod1}
	I_2^n \ton \int_0^t \int_\T  
	\Bigl[\tilde{u}\, \pd_x \tilde{u}
	+\pd_x P\bigl[\tilde{u}\bigr] \Bigr]
	\varphi\,\d x\,\d s, \quad 
	\text{$\tilde{\mathbb{P}}$-a.s.}
\end{align}

\medskip
{\it 2.  Convergence of $I_1^n$ and $I_3^n$.}

For any $\varphi \in C^1(\T)$, 
$\tilde{\mathbb{P}}$-almost surely,
\begin{align*}
	&\abs{-\frac12\int_0^t \int_\T \sigma\tilde{q}
	\,\pd_x\bk{\sigma\varphi} \,\d x\,\d s-I_3^n}
	\\& 
	\le \frac12\int_0^t\abs{\int_\T \sigma
	\bigl(\tilde{q}-\tilde{q}_n\bigr) 
	\pd_x\bk{\sigma\varphi}\,\d x} \,\d s 
	\\ &\quad
	+ \frac12\int_0^t\abs{\int_\T \sigma\tilde{q}_n 
	\,\pd_x \bigl(\sigma \bk{1-\bm{\Pi}_n}
	\varphi\bigr)\,\d x}\,\d s
	\\ &  
	\le C_{\sigma, \varphi}\norm{\tilde{q}
	-\tilde{q}_n}_{L^2([0,T]\times \T)} 
	\\ & \quad \,\,
		+C_\sigma \norm{\tilde{q}_n}_{L^2([0,T]\times \T)} 
	\norm{\bk{1-\bm{\Pi}_n}\varphi}_{H^1(\T)}\ton 0.
\end{align*}
Similarly, by the $\tilde{\mathbb{P}}$-a.s.~$L^2_{t,x}$ 
convergence of $\tilde{q}_n$, cf.~\eqref{eq:basic-conv},
$$
\abs{\int_0^t\int_\T \bigl(\ep\tilde{q}-\ep\tilde{q}_n\bigr)
\, \pd_x \varphi \,\d x \,\d s} \ton 0,
\quad \text{$\tilde{\mathbb{P}}$-a.s.}
$$
The convergence  
$$
\int_\T \tilde{u}_n(t) \varphi \,\d x 
\to \int_\T \tilde{u}(t)\varphi 
\,\d x, \quad
\quad \text{$\tilde{\mathbb{P}}$-a.s.}
$$
follows from the $\tilde{\mathbb{P}}$-a.s.~convergence 
$\tilde{u}_n \to \tilde{u}$ in $C([0,T];H^1_w(\T))$, see
\eqref{eq:basic-conv}, noting that $\varphi 
\in C^1(\T) \subseteq H^1(\T)$. Finally, the convergence
$$
\int_\T \tilde{u}_{0,n} \varphi \,\d x  
\ton \int_\T \tilde{u}_0 \varphi \,\d x, 
\quad \text{$\tilde{\mathbb{P}}$-a.s.}
$$
is a direct consequence of Theorem \ref{thm:skorohod} 
and \eqref{eq:pathspaces}.

Combining these results we find that
\begin{equation}\label{eq:galerkin_skorokhod2}
	\begin{split}
		I_1^n + I_3^n & 
		\ton \int_\T \tilde{u}(t) \varphi \,\d x  
		-\int_\T \tilde{u}_0 \varphi \,\d x
		-\int_0^t \int_\T \ep\pd_x\tilde{u}\,
		\pd_x\varphi\,\d x\,\d s
		\\ & \qquad 
		-\frac12\int_0^t \int_\T \sigma\,\pd_x \tilde{u}\,
		\pd_x \bk{\sigma \varphi}\,\d x \,\d s, 
		\quad \text{$\tilde{\mathbb{P}}$-a.s.}
	\end{split}
\end{equation}

\medskip
{\it 3.  Convergence of $I_4^n$.} 

First, $\tilde{\mathbb{P}}$-almost surely., we have
\begin{align*}
	\norm{\sigma \tilde{q}
	-\bm{\Pi}_n\bk{\sigma \tilde{q}_n}}_{L^2([0,T]\times\T)}
	&\le \norm{\sigma
	\bk{\tilde{q}-\tilde{q}_n}}_{L^2([0,T]\times\T)}
	\\ & \quad 
	+\norm{\sigma \tilde{q}_n
	-\bm{\Pi}_n\bk{\sigma \tilde{q}_n}}_{L^2([0,T]\times\T)} 
	\ton 0,
\end{align*}
and so $\bm{\Pi}_n\bk{\sigma\,\tilde{q}_n}\ton \sigma
\tilde{q}$ in $L^2([0,T]\times \T)$ in probability. 
Besides, $\tilde{W}_n \ton \tilde{W}$ in 
$C([0,T])$ $\tilde{\mathbb{P}}$-almost surely, and thus 
in probability. Therefore, by 
Lemma \ref{lem:stoch-conv},
\begin{align}\label{eq:DGHT_convergence}
	I_4^n \ton \int_0^t 
	\sigma \tilde{q} \,\d \tilde{W}, 
	\quad\mbox{in $L^2([0,T]\times \T)$},
\end{align}
in probability and hence $\tilde{\mathbb{P}}$-almost 
surely along a subsequence.

\medskip
{\it 4. Weak formulation.}

Gathering \eqref{eq:galerkin_skorokhod1}, 
\eqref{eq:galerkin_skorokhod2}, and \eqref{eq:DGHT_convergence}, 
we have shown that $\tilde{u}$, $\tilde{W}$, 
$\tilde{u}_0$ satisfy, for any $\varphi \in C^1_c(\T)$,
\begin{equation} \label{eq:weak_solution_eq}
	\begin{aligned}
		 0 &= \int_\T \varphi \tilde{u}\,\d x \bigg|_0^T
		 +\int_0^T \int_\T \varphi \,\tilde{u} \,\pd_x \tilde{u} \,\d x \,\d t
		 -\ep  \int_0^T \int_\T \varphi \pd_x^2 \tilde{u}\,\d x \,\d t 
		 \\ &\quad 
		 +\int_0^T\int_\T \varphi \pd_x P\bigl[\tilde{u}\bigr]
		 \,\d x \,\d t
		 -\frac{1}{2}\int_0^T \int_\T \sigma\, \pd_x\bk{\sigma\pd_x\tilde{u}}
		 \,\d x\, \d t
		 \\ &\quad 
		 +\int_0^T \int_\T \sigma \pd_x \tilde{u} 
		 \,\d x \, \d \tilde{W},
		 \\ \tilde{u}(0) &= \tilde{u}_0
	\end{aligned}
\end{equation}
as in Definition \ref{def:wk_sol}(e).

\medskip
{\it 5. Appropriate inclusions.}

The $\tilde{\mathbb{P}}$-almost sure inclusion 
$\tilde{u} \in C([0,T];H^1_w(\T)) $
follows directly from the Skorokhod argument of 
Theorem \ref{thm:skorohod}. 
Therefore, we are left to show that 
$\tilde{u} \in L^p(\Omega ;L^\infty([0,T];H^1(\T)))
\cap L^2(\Omega \times [0,T]; H^2(\T))$. 

By the Lusin--Souslin theorem \cite[Thm.~15.1]{Kec1995}, the inclusion 
$L^2([0,T];H^2(\T)) \hookrightarrow L^2([0,T];H^1(\T))$ is Borel. 
We can then invoke the equality of laws to obtain
$$
\tilde{\Ex}\norm{\tilde{u}_n}_{L^2([0,T];H^2(\T))}^2
={\Ex}\norm{{u}_n}_{L^2([0,T];H^2(\T))}^2 < C,
$$
where $C$ is independent of $n$, 
by Theorem \ref{thm:galerkin_H1}.  
This implies that $\tilde{q}_n$ are uniformly bounded in 
$L^2(\Omega\times [0,T]\times \T)$. Therefore, by reflexivity, 
any weak limit is also in $L^2(\Omega\times [0,T]\times \T)$.

The inclusion $L^\infty([0,T];H^1(\T)) \hookrightarrow C([0,T];H^1_w(\T))$ 
is continuous because for any $\varphi \in H^1(\T)^*= H^{-1}(\T)$, 
$$
\sup_{t \in [0,T]} \abs{\LL u, \varphi\RR_{H^{-1},H^1,}} 
\le \norm{\varphi}_{H^{-1}(\T)} 
\sup_{t \in [0,T]} \norm{u}_{H^1(\T)}.
$$
Therefore, by the quasi-Polish version 
of the Lusin--Souslin theorem \cite[Cor.~A.2]{Ond2010},
we maintain as before the higher moment bound
$$
\tilde{\Ex}\norm{\tilde{u}_n}_{L^\infty([0,T];H^1(\T))}^p
	 ={\Ex}\norm{{u}_n}_{L^\infty([0,T];H^1(\T))}^p < C.
$$ 
\end{proof}

\section{Pathwise uniqueness and proof of 
Theorem \ref{thm:main1}}\label{sec:pw_uniqueness}

In this section, we will show pathwise uniqueness and, 
consequently, the existence of strong solutions in 
the energy space $L^2\left(\Omega;L^\infty([0,T];
H^1(\T))\right)$ (Theorem~\ref{thm:st_uniqueness}). This 
will involve estimates similar to the energy inequality in 
Proposition \ref{thm:galerkin_H1}. However, as we 
are dealing with solutions a.s.~in 
$L^\infty([0,T];H^1(\T))$, calculations using smooth  Galerkin 
approximations cannot be reproduced here. 
To keep using the standard (finite-dimensional) 
It\^o formula, we convolve the SPDE against a 
standard Friedrichs mollifier $J_\delta$, making it 
possible to interpret the SPDE pointwise in $x$.
Mollification introduces error terms to 
the equation, (see ~\eqref{eq:u_ch_ep_delta}). 
We will first state and prove convergence results 
for these error terms. 

\subsection{Regularisation errors}\label{sec:commutators}
We begin this subsection by proving first order 
commutator estimates in the stochastic setting. 
Notice that the fourth moment assumption is made. 
This assumption is the reason that a bounded $p > 4$ 
moment is needed on the initial condition (e.g., 
in Theorem \ref{thm:main1}). 
The assumption itself arises from \eqref{eq:L4}, 
where the $L^\infty_t L^2_x$ boudedness of 
$\pd_x u$ is exploited in applying Young's convolution inequality. 
It is true that $\pd_x u$ is in  $L^{3 -}_{\omega, t, x}$ 
uniformly in $\ep$, but because of the extra square in the exponent, 
this is difficult to exploit. Higher integrability bounds for fixed 
$\ep > 0$ exist but may only hold up to stopping time.

Throughout the paper we let $J_\delta$ 
be a standard Friedrichs (spatial) mollifier, and set 
$u_\delta := u *J_\delta$, and use $\delta$ as subscript to denote a mollified
function.

\begin{lem}[Commutator estimates]\label{thm:commutator1}
Let $u,v,w \in L^4(\Omega; L^\infty([0,T];H^1(\T)))$, and 
suppose $\sigma \in W^{1,\infty}(\T)$. 
Finally, let $K \in W^{1,\infty}(\T)$ be a 
given kernel function. Define the commutator functions:
\begin{equation}\label{eq:commutator_defin1}
	\begin{aligned}
		 E^1_\delta=E^1_\delta(u,v)
		&:= \Bigl(u \,\pd_x u 
		-v \,\pd_x v \Bigr)*J_\delta
		-\bk{u_\delta \,\pd_x u_\delta 
		-v_\delta \,\pd_x v_\delta} 
		\\ & \quad
		+\pd_x K*\bk{u^2-v^2
		+\frac12 \bk{\bk{\pd_xu}^2-\bk{\pd_x v}^2}}*J_\delta
		\\ & \quad
		-\pd_x K *\bk{u_\delta^ 2-v_\delta^2 
		+\frac12 \bk{\bk{\pd_x u_\delta}^2
		-\bk{\pd_x v_\delta}^2}},
		\\
		 E^2_\delta= E^2_\delta(w)
		&:=\bigl( \sigma \,\pd_x w\bigr)
		*J_\delta -\sigma \,\pd_x w_\delta,
		\\ 
		E^3_\delta=E^3_\delta(w) 
		&:=-\frac12 \Bigl(\sigma\,
		\pd_x\bk{\sigma\,\pd_x w} \Bigr)*J_\delta  
		+\frac12 \sigma\,\pd_x\bk{\sigma\,\pd_x w_\delta}.
	\end{aligned}
\end{equation}

The following convergences hold:
\begin{align}\label{eq:commutator_converge1}
	\Ex \norm{E^1_\delta}_{L^2([0,T]\times \T)}^2, 
	\, \Ex \norm{E^2_\delta}_{L^2([0,T];H^1 (\T))}^2, 
	\, \Ex \norm{E^3_\delta}_{L^2( [0,T] \times \T)}^2 
	\todelta 0.
\end{align}
\end{lem}

\begin{proof}

Whilst these commutator estimates are 
similar to the classical ones of \cite{DL1989}, 
we prove them here both because we are in 
the stochastic setting, with an extra 
integral in $\d \mathbb{P}$, and also 
because the extra temporal integrability 
on $\norm{u}_{H^1(\T)}$  permits for the slightly stronger results that 
we shall be using. In particular, we 
have bounds in $L^2_x$, and not 
only $L^1_x$, for $E^1_\delta$.

\medskip
{\it 1. Convergence of $E^1_\delta$.}
\medskip

For the transport terms we have
\begin{align*}
	&\norm{\bk{u\,\pd_x u}*J_\delta
	-u_\delta\,\pd_x u_\delta}_{L^2([0,T]\times \T)}^2
	\\ & \quad \lesssim  
	\norm{\bk{u\,\pd_x u}*J_\delta
	-u\,\pd_x u_\delta}_{L^2([0,T]\times \T)}^2
	+\norm{ u\,\pd_x u_\delta-u_\delta\,
	\pd_x u_\delta}_{L^2([0,T]\times \T)}^2
	\\ & \quad = \norm{\int_\T \frac{u(\dott)
	-u(y)}{\dott-y} \pd_y u(y) \,(\dott-y) 
	J_\delta(\dott-y) \,\d y}_{L^2([0,T]\times \T)}^2
	\\ &\quad\quad 
	+ \norm{\int_\T \frac{u(\dott)-u(y)}{\dott-y} 
	\pd_x u_\delta(\dott) \,(\dott-y) 
	\,J_\delta(\dott-y)\,\d y}_{L^2([0,T]\times \T)}^2
	\\ & \quad = : I_1^\delta + I_2^\delta.
\end{align*}

By Young's convolution inequality, and 
the fact that $\delta \norm{J_\delta}_{L^2(\T)} 
\lesssim \sqrt{\delta}$,
\begin{equation}\label{eq:L4}
	\begin{aligned}
		\Ex\abs{I_1^\delta} & 
		\lesssim \Ex \int_0^T \norm{\abs{\pd_xu}\, 
		\sup_{|h| \le \delta} 
		\abs{\frac{u(\dott+h)-u(\dott)}{h}}}_{L^1(\T)}^2 
		\delta^2 \norm{J_\delta}_{L^2(\T)}^2\,\d s 
		\\ & \lesssim \Ex \, \delta \int_0^T 
		\norm{\pd_x u}_{L^2(\T)}^4 \,\d s
		\lesssim_T \delta.
	\end{aligned}
\end{equation}
Similarly,
\begin{align*}
	\Ex \abs{I_2^\delta} & \lesssim 
	\delta\,\Ex \abs{\int_0^T
	\norm{\pd_x u}_{L^2}^2 
	\norm{\pd_x u_\delta}_{L^2(\T)}^2 \,\d s} 
	\\ & 
	\lesssim \delta\,\Ex \int_0^T 
	\norm{\pd_x u}_{L^2(\T)}^4 \,\d s
	\lesssim_T \delta.
\end{align*}
Therefore, 
$$
\Ex \norm{\bk{u\,\pd_x u}*J_\delta
-u_\delta\,\pd_x u_\delta}_{L^2([0,t]\times \T)}^2 
\le \Ex \bigl( I_1^\delta + I_2^\delta\bigr) 
\todelta 0.
$$

Consider the terms in $E^1_\delta$ involving 
the kernel $K$, for which $\norm{\pd_x K}_{L^2(\T)}\lesssim 1$. For any 
$\xi \in L^4\left(\Omega;L^\infty([0,T];L^2(\T))\right)$, we find 
\begin{align*}
	&\norm{\pd_x K * \xi^2 * J_\delta
	-\pd_x K *\xi_\delta^2}_{L^2(\T)}^2
	\\ & \qquad\quad
	\le \norm{\pd_x K}_{L^2(\T)}^2 
	\norm{\xi^2*J_\delta-\xi_\delta^2}_{L^1(\T)}^2
	\\ & \qquad\quad \lesssim 
	\norm{\xi^2*J_\delta-\xi^2}_{L^1(\T)}^2
	+\norm{\xi^2 - \xi \xi_\delta}_{L^1(\T)}^2
	+\norm{\xi \xi_\delta-\xi_\delta^2}_{L^1(\T)}^2.
\end{align*}
By standard properties of Friedrichs mollifiers, the 
terms on the right-hand side all tend to zero 
as $\delta \to 0$. We take $\xi = u, v$ or 
$\xi = \pd_x u, \pd_x v$ in the calculation above. 

Combining the foregoing calculations, we arrive at
$$
\Ex \norm{E^1_\delta}_{L^2([0,T]\times \T)}^2 
\todelta 0.
$$

\medskip
{\it 2. Convergence of $E^2_\delta$.}
\medskip

For any $\xi \in L^2(\Omega \times [0,T]\times \T)$ 
(such as $\xi = u$ or $\xi = \pd_x u$), the convergence 
$$
 \bigl(\sigma \xi \bigr)*J_\delta- \bk{\sigma \,\xi_\delta }
\todelta 0 \quad \text{in $L^2( \Omega \times [0,T] \times \T)$}.
$$
is a direct result of the dominated convergence theorem.

The convergence 
$$
\pd_x \bigl(\sigma \xi \bigr)*J_\delta-\pd_x \bk{\sigma \,\xi_\delta }
\todelta 0 \quad \text{in $L^2( [0,T] \times \T)$}
$$
follows directly from \cite[Lemma II.1]{DL1989}, 
where it was shown that
\begin{align*}
	\norm{\pd_x\bigl(\sigma \xi \bigr)*J_\delta
	-\pd_x\bk{\sigma \,\xi_\delta}}_{L^2([0,T]\times \T)} 
	\le C_{\delta,\sigma} 
	\norm{\xi}_{L^2([0,T] \times \T)},
\end{align*}
where $C_{\delta, \sigma} \todelta 0$ and 
is independent of $\xi$ (and therefore deterministic). 
This gives
$$
\Ex \norm{E^2_\delta}_{L^2([0,T] ;H^1( \T))}^2
\todelta 0.
$$

\medskip

{\it 3. Convergence of $E^3_\delta$ in $L^2(\Omega;L^\infty([0,T];L^2(\T)))$.}

\medskip

Since $w \in L^4\left(\Omega;L^\infty([0,T];
H^1(\T))\right)$, again setting $\xi = \pd_x w$, so 
that $\xi$ belongs to $L^4\left(\Omega;L^\infty([0,T];
L^2(\T))\right)$, the commutator can be written as
\begin{align*}
	-2 E^3_\delta & = \Bigl(\sigma\,\pd_x \bk{\sigma \xi}
	\Bigr) *J_\delta-\sigma \,\pd_x \bk{\sigma \xi_\delta}
	\\ & \quad 
	= \Bigl[\Bigl(\sigma\,\pd_x \sigma \xi \Bigr) *J_\delta 
	-\sigma \,\pd_x \sigma \xi_\delta \Bigr]
	+  \Bigl[\Bigl({\sigma^2 \,\pd_x \xi} 
	\Bigr) *J_\delta - \sigma^2\,\pd_x \xi_\delta\Bigr].
\end{align*}
We can then apply step 2 of the proof to conclude 
that \eqref{eq:commutator_converge1} holds.
\end{proof}

Next, we introduce an operator notation 
that will be indispensable in the next two results. 
For $f \in L^p(\T)$, $1 \le p \le \infty$, 
set 
\begin{equation}\label{eq:boldface}
\JJ_\delta f := f_\delta = f *J_\delta, \quad \KK f := \pd_x \bk{\sigma f}
\end{equation} 
where $J_\delta$ is the mollifier used in 
the definition of $E^3_\delta$. 
Finally, we define 
$$
\bigl[\KK,\JJ_\delta\bigr](f)
:=\KK\JJ_\delta f-\JJ_\delta \KK f
=\pd_x \bigl( \sigma f_\delta\bigr) 
- \pd_x (\sigma f) *  J_\delta.
$$

\begin{lem}[Double commutator estimate]\label{thm:commutator2}
Let $\xi \in L^2(\Omega \times [0,T] \times \T)$, 
and suppose $\sigma \in W^{2,\infty}(\T)$. 
Then 
\begin{align*}
	R_\delta :=J_\delta*\pd_x 
	\bigl(\sigma\,\pd_x\bk{\sigma \xi}\bigr) 
	-2 \pd_x\bigl(\sigma\,J_\delta*\pd_x\bk{\sigma \xi}\bigr)
	+\pd_x\bigl(\sigma \,\pd_x\bk{\sigma \xi_\delta}\bigr)
	\todelta 0
\end{align*}
in $L^2(\Omega \times [0,T]\times \R)$.
\end{lem}

\begin{rem}
An almost sure version of this lemma (instead 
of convergence in $L^2(\Omega)$) was stated with 
a similar proof in \cite[Lemma B.3]{HKP2021}. 
Furthermore, we allow for a more general $\sigma$ here 
than in \cite{PS2018}; the work \cite{PS2018} 
imposes a divergence-free condition on $\sigma$ 
(with $\T$ replaced by $\R^d$).
\end{rem}

\begin{proof}
Using the operator notation defined 
in \eqref{eq:boldface}, we can write 
$-R_\delta$ as a double commutator:
\begin{align}\label{eq:commutators}
	-R_\delta & =\Bigl[
	\bigl[\KK,\JJ_\delta\bigr],\KK\Bigr](\xi) 
	= \bigl[\KK,\JJ_\delta\bigr](\KK \xi)
	-\KK\bigl[\KK,\JJ_\delta\bigr](\xi)
	\notag \\ & 
	= 2 \KK \JJ_\delta \KK \xi-\JJ_\delta \KK \KK \xi
	-\KK \KK \JJ_\delta \xi.
\end{align}
Term-by-term we have
\begin{align}
	2\KK\JJ_\delta\KK \xi(x)
	& = 2\int_\R \pd_{xx}^2 J_\delta(x-y) 
	\sigma(x)\sigma(y) \xi(y) \;\d y 
	\label{eq:commute1}
	\\ & \quad 
	+ 2 \int_\R \pd_x J_\delta(x-y) 
	\pd_x \sigma(x) \sigma(y) \xi(y) \;\d y,
	\label{eq:commute2}
	\\ \JJ_\delta \KK \KK \xi (x)
	& = \int_\R \pd_{xx}^2 J_\delta(x-y) 
	\sigma^2(y) \xi(y) \;\d y
	\label{eq:commute3}
	\\ & \quad 
	- \int_\R \pd_x J_\delta(x-y)\sigma(y)
	\pd_y \sigma(y)\xi(y)\;\d y,
	\label{eq:commute4}
	\\ \intertext{and}
	\KK\KK \JJ_\delta \xi (x)
	& =\int_\R  J_\delta(x-y) \pd_x \big(\sigma(x) 
	\pd_x \sigma(x)\big) \xi(y) \;\d y
	\label{eq:commute5}
	\\ & \quad +3\int_\R \pd_x J_\delta(x-y) 
	\sigma(x)\pd_x \sigma(x) \xi(y)\;\d y 
	\label{eq:commute6}
	\\ & \quad
	+\int_\R \pd_{xx}^2 
	J_\delta(x-y) \sigma^2(x)\xi(y)\;\d y 
	\label{eq:commute7}.
\end{align}

We will estimate \eqref{eq:commute1} to 
\eqref{eq:commute7} by considering the sums
$$
\mathfrak{I}_1:= \eqref{eq:commute2} 
-\eqref{eq:commute4} -\eqref{eq:commute6},
\quad
\mathfrak{I}_2:= \eqref{eq:commute1} 
-\eqref{eq:commute3} - \eqref{eq:commute7},
$$
and the stand-alone integral \eqref{eq:commute5}, 
where, from \eqref{eq:commutators}, 
we see that
\begin{align}\label{eq:commutator_decomposition}
	-R_\delta=\Bigl[\bigl[\KK,\JJ_\delta\bigr],\KK\Bigr](\xi)
	=\mathfrak{I}_1+\mathfrak{I}_2
	-\eqref{eq:commute5}.
\end{align}
We will use \cite[Lemma II.1]{DL1989} 
to establish that \eqref{eq:commutator_decomposition} 
tends to zero in an appropriate sense. 
Estimating the terms in \eqref{eq:commutator_decomposition} 
separately, we have
{\small
\begin{align*}
\norm{\mathfrak{I}_1}_{L^2(\R)}
	 & 
	=\biggl\|\int_\R \pd_x J_\delta(\dott-y) 
	\\ &\quad \quad \times 
	\Bigl(2 \sigma(y) \pd_x \sigma(\dott)
	+\sigma(y) \pd_y \sigma(y)
	-3 \sigma(\dott) \pd_x \sigma(\dott)\Bigr) 
	\xi(y) \;\d y\biggr\|_{L^2(\R)}
	\\ & \quad 
	=\biggl\| \int_\R \pd_x J_\delta(\dott-y) 
	\\ & \quad \qquad\times 
	\Big(2 \bigl(\sigma(y)-\sigma(\dott)\bigr)
	\pd_x \sigma(\dott)+\big(\sigma(y) \pd_y \sigma(y) 
	-\sigma(\dott) \pd_x \sigma(\dott)\big)\Bigr) 
	\xi(y) \;\d y\bigg\|_{L^2(\R)}
	\\ & \quad \le 
	\biggl\| \int_\R \abs{\pd_x J_\delta(\dott-y)} 
	\\ & \quad \quad\times
 	\Big(2 \abs{\sigma(y)-\sigma(\dott)}
 	\, \abs{\pd_x \sigma(\dott)}
 	+\abs{\sigma(y) \pd_y \sigma(y) 
 	-\sigma(\dott)\pd_x\sigma(\dott)}\Big) 
 	\abs{\xi(y)} \;\d y\biggr\|_{L^2(\R)}
 	\\	& \quad 
 	\le C \biggl\| \int_\R
 	\abs{\dott-y}\abs{\pd_x J_\delta(\dott - y)} 
 	\\ & \quad \quad\times
 	\left(2\abs{\frac{\sigma(y)
 	-\sigma(\dott)}{y-\dott}\pd_x\sigma(\dott)} 
 	+\abs{\frac{\sigma(y)\pd_y \sigma(y) 
 	-\sigma(\dott) \pd_x \sigma(\dott)}{y-\dott}}
 	\right) \abs{\xi(y)} \;\d y\biggr\|_{L^2(\R)}
 	\\ & \quad \le C\left(\norm{\pd_x\sigma}_{L^\infty(\R)}^2 
 	+\norm{\pd_x\bigl(\sigma 
 	\pd_x \sigma\bigr)}_{L^\infty(\R)} \right)
 	\\ & \quad \quad \times 
 	\biggl\| \int_\R  \abs{\dott-y} 
 	\abs{\pd_x J_\delta(\dott-y)}\abs{\xi(y)} 
 	\;\d y\biggr\|_{L^2(\R)} 
 	\\	 & \quad \le C\Bigl(\norm{\pd_x\sigma}_{L^\infty(\R)}^2 
 	+ \norm{\pd_x\bigl(\sigma 
 	\pd_x \sigma\bigr)}_{L^\infty(\R)} \Bigr) 
 	\Bigl\|\abs{\dott} \pd_x J_\delta(\dott)\Bigr\|_{L^1(\R)}
 	\norm{\xi}_{L^2(\R)} 
 	\\ & \quad  
 	\le C \Big(\norm{\pd_x\sigma}_{L^\infty(\R)}^2
 	+\norm{\pd_x\bigl(\sigma \pd_x \sigma\bigr)}_{L^\infty(\R)}
 	\Bigr)\norm{\xi}_{L^2(\R)},
\end{align*}
}where we have used Young's convolution 
inequality and subsequently the basic estimate 
$\bigl\|\abs{\dott} \pd_x J_\delta(\dott)
\bigr\|_{L^1(\R)}\lesssim 1$. Similarly,
{\small
\begin{align*}
	\norm{\mathfrak{I}_2}_{L^2(\R)} 
	& =\norm{\int_\R \pd_{xx}^2 
	J_\delta(\dott-y)\big(2 \sigma(\dott)\sigma(y)
	-\sigma^2(\dott)-\sigma^2(y)\big)
	\xi(y) \;\d y}_{L^2(\R)}
	\\ & 
	= \norm{\int_\R \pd_{xx}^2 J_\delta(\dott-y) 
 	\bigl(\sigma(\dott)-\sigma(y)\bigr)^2 
 	\xi(y)\;\d y}_{L^2(\R)}
 	\\ & 
 	\le \norm{\int_\R 
 	\abs{\pd_{xx}^2 J_\delta(\dott-y)}
 	\abs{2 \sigma(\dott) \sigma(y) - \sigma^2(\dott)
 	-\sigma^2(y)}\, \abs{\xi(y)}\;\d y}_{L^2(\R)}
 	\\	 & 
 	\le C \norm{\int_\R \bigl(\dott-y\bigr)^2
 	\abs{\pd_{xx}^2J_\delta(\dott-y)}
 	\abs{\frac{\sigma(\dott) - \sigma(y)}{\dott-y}}^2 
 	\abs{\xi(y)} \;\d y}_{L^2(\R)}
 	\\ & 
 	\le C	\norm{\pd_x \sigma}_{L^\infty(\R)}^2
 	\norm{\int_\R
 	(\dott - y)^2\abs{\pd_{xx}^2 J_\delta(\dott-y)}
 	\abs{\xi(y)} \;\d y}_{L^2(\R)}
 	\\ & 
 	\le C\norm{\pd_x \sigma}_{L^\infty(\R)}^2
 	\Bigl\|(\dott)^2\pd_{xx}^2J_\delta(\dott)\Bigr\|_{L^1(\R)}
 	\norm{\xi}_{L^2(\R)} 
 	\\& \le 
 	C \norm{\pd_x \sigma}_{L^\infty(\R)}^2
 	\norm{\xi}_{L^2(\R)}.
\end{align*}}
We also have
$$
\norm{\eqref{eq:commute5}}_{L^2(\R)}
\le C \norm{J_\delta}_{L^1(\R)}
\norm{\pd_x(\sigma\pd_x \sigma)}_{L^\infty(\R)}
\norm{\xi(t)}_{L^2(\R)}.
$$

Given the last three ($\delta$-independent) bounds, 
it is sufficient to establish convergence of
\eqref{eq:commutators} under the 
assumption that $\sigma$, $\xi$ are smooth (in $x$). 
The general case follows by density using the 
established bounds. Under this assumption, we have
\begin{align*}
	\mathfrak{I}_2 
	& =\int_\R \pd_{xx}^2 J_\delta(x - y)
	\bigl( 2 \sigma(x)\sigma(y)-\sigma^2(x)
	-\sigma^2(y)\bigr) \xi(y) \;\d y
	\\ & 
	=- 2\int_\R\pd_{xx}^2 J_\delta(x-y) 
	\frac{(x - y)^2}{2}
	\left(\frac{\sigma(y)-\sigma(x)}{y-x} 
	\right)^2 \xi(y) \;\d y
	\\ & = -2 \bigl(\pd_x \sigma(x)\bigr)^2 
	\xi(x)\int_\R \frac{z^2}{2}\pd_{zz}^2 
	J_\delta(z) \; \d z+o_\delta(1),
\end{align*}
where $\int_\R \frac{z^2}{2}
\pd_{zz}^2 J_\delta(z) \; \d z=1$. A similar 
calculation can be done for $\mathfrak{I}_1$, 
in which case there is only one derivative on the mollifier 
and then the calculation can be found 
in the proof of \cite[Lemma II.1]{DL1989}. 
The limit of \eqref{eq:commute5} is standard. 
Reasoning as in the proof 
of \cite[Lemma II.1]{DL1989}, we arrive at
\begin{align*}
	\mathfrak{I}_1&\overset{\delta \downarrow 0}{\to}
	\pd_x \bigl(\sigma \pd_x \sigma\bigr)
	\xi+2\bigl(\pd_x \sigma\bigr)^2 \xi, 
	\quad \mathfrak{I}_2  \overset{\delta \downarrow 0}{\to} 
	- 2\bigl(\pd_x \sigma\bigr)^2 \xi, 
	\\ 
	-\eqref{eq:commute5}&\overset{\delta \downarrow 0}{\to} 
	-\pd_x \bigl(\sigma \pd_x \sigma\bigr) \xi
	\quad \text{in $L^2(\R)$, for 
	$\d\mathbb{P}\otimes \d t$-a.e.}
\end{align*}

Adding these terms together, with 
reference to \eqref{eq:commutator_decomposition}, and 
using the dominated convergence theorem, we conclude that   
$-R_\delta\todelta 0$ in $L^2(\Omega \times [0,T] \times\R)$.
\end{proof}

\begin{prop}[It\^{o}--Stratonovich conversion terms 
and regularisation errors]\label{thm:commutator3}
Let $S \in C^1(\R) \cap \dot{W}^{2,\infty}(\R)$ 
satisfy $S'(r)=O(r)$ and $\sup_r {\abs{S''(r)}} 
<\infty$. Let $\varphi \in C^\infty([0,T]\times \T)$.
 With $w$, $w_\delta$, $E^2_\delta$, 
and $E^3_\delta$ defined as in \eqref{eq:commutator_defin1} 
of Lemma \ref{thm:commutator1}, we have 
\begin{equation}\label{eq:commutator_converge2}
	\begin{aligned}
		\Ex \int_0^T\biggl| \int_\T &
		-\varphi S'(\pd_x w_\delta) \pd_x E^3_\delta 
		\\ & +\varphi  S''(\pd_x w_\delta) \bk{\frac12 
		\,\abs{\pd_x E^2_\delta}^2
		+\pd_x \bk{\sigma \,\pd_x w_\delta}
		\,\pd_x E^2_\delta}
		\,\d x\biggr| \,\d t\todelta 0.
	\end{aligned}
\end{equation}
\end{prop}

\begin{rem}
An almost sure version of this proposition, 
instead of convergence in $L^1(\Omega)$, 
and with $\varphi \equiv 1$, was 
stated with a similar proof in \cite[Lemma B.3]{HKP2021}. 
Moreover, the result in \cite{HKP2021} was 
stated with slightly more stringent conditions 
on $S$, requiring $\abs{S'(r)}=O(1)$ instead 
of $\abs{S'(r)} = O(r)$.	
For the remainder of this paper, 
we shall use  $\varphi \equiv 1$.
\end{rem}

\begin{proof}
In the following, we continue using the 
operator notation consisting of $\JJ_\delta$ 
and $\KK$ defined in \eqref{eq:boldface}. 
The estimate \eqref{eq:commutator_converge2} 
takes inspiration from the proof 
of \cite[Prop.~3.4]{PS2018}. However, whereas they 
considered the commutator between the operators 
$\tilde{\KK} f:= \sigma \pd_x f$ and 
${\JJ_\delta} f$, 
we have to consider the analogous question 
for $\KK f= \pd_x( \sigma f)$ and $\JJ_\delta$.  
Insofar as $\pd_x w$ can be any element $\xi \in 
	L^2(\Omega;L^\infty([0,T];L^2(\T)))$ (in fact, even 
just $\xi \in L^2(\Omega \times [0,T] \times \T)$!)  
for the purpose of the convergence, denote 
$\pd_x w$ by $\xi$, and $\pd_x w_\delta$ 
by $\xi_\delta$, since mollification 
commutes with (weak) differentiation.

We can express $\pd_x E^\delta_3$ in terms 
of commutator brackets as follows:
\begin{align}
	\pd_x E^3_\delta(\xi) &
	= \frac12\Bigl(\KK \KK \JJ_\delta \xi
	-\JJ_\delta \KK \KK \xi \Bigr)= \frac12 
	\Bigl( \KK \bigl[\KK, \JJ_\delta\bigr](\xi)
	+\bigl[\KK, \JJ_\delta\bigr] \KK (\xi) \Bigr). \label{eq:rho_ep_1}
\end{align}
Similarly, we can write the remaining part of the integrand 
of \eqref{eq:commutator_converge2} 
in the form $\frac12 S''(\xi_\delta) E^4_\delta$, where
\begin{align*}
	E^4_\delta(\xi)
	:= \bigl( \pd_x(\sigma \xi)*J_\delta\bigr)^2
	-\bigl(\pd_x(\sigma \xi_\delta)\bigr)^2 
	=  \bigl(\JJ_\delta \KK \xi\bigr)^2
	-\bigl(\KK \JJ_\delta \xi\bigr)^2 .
\end{align*}
Therefore, following the calculations 
in \cite[p.~655]{PS2018},
\begin{align*}
	- \frac{1}{2} S''(\xi_\delta) E^4_\delta
	& = \frac{1}{2} S''(\xi_\delta)
	\bigl( \KK \JJ_\delta \xi-\JJ_\delta \KK \xi\bigr)
	\bigl(\KK \JJ_\delta \xi+\JJ_\delta \KK \xi\bigr)
	\\ & 
	= -\frac{1}{2} S''(\xi_\delta) 
	\bigl(\bigl[\KK,\JJ_\delta\bigr](\xi)\bigr)^2
	+S''(\xi_\delta)\bigl(\KK \JJ_\delta \xi\bigr)
	\bigl[\KK,\JJ_\delta\bigr](\xi) 
	\\ & = -\frac{1}{2} S''(\xi_\delta) 
	\bigl(\bigl[\KK,\JJ_\delta\bigr](\xi)\bigr)^2 
	+S''(\xi_\delta) \pd_x \sigma \xi_\delta 
	\bigl[\KK,\JJ_\delta\bigr](\xi)
	\\ &\quad
	+\sigma \pd_x (S'(\xi_\delta)) 
	\bigl[\KK,\JJ_\delta\bigr](\xi)
	\\ &= -\frac{1}{2} S''(\xi_\delta) 
	\bigl(\bigl[\KK,\JJ_\delta\bigr](\xi)\bigr)^2 
	+S''(\xi_\delta) \pd_x \sigma \xi_\delta 
	\bigl[\KK,\JJ_\delta\bigr](\xi)
	\\ & \quad
	+\pd_x \bigl(\sigma S'(\xi_\delta) 
	\bigl[\KK,\JJ_\delta\bigr](\xi)\bigr) 
	- S'(\xi_\delta) \pd_x \bigl(\sigma
	\bigl[\KK,\JJ_\delta\bigr](\xi)\bigr)
	\\ & = -\frac{1}{2} S''(\xi_\delta) 
	\bigl(\bigl[\KK,\JJ_\delta\bigr](\xi)\bigr)^2 
	+S''(\xi_\delta) \pd_x \sigma \xi_\delta 
	\bigl[\KK,\JJ_\delta\bigr](\xi)
	\\ & \quad
	+ \pd_x \bigl(\sigma S'(\xi_\delta) 
	\bigl[\KK,\JJ_\delta\bigr](\xi)\bigr) 
	-S'(\xi_\delta)\KK\bigl[\KK,\JJ_\delta\bigr](\xi),
\end{align*}
by invoking the definition of $\KK$. 
Adding this to \eqref{eq:rho_ep_1}, we find that
\begin{align*}
	-\frac{1}{2} &S''(\xi_\delta)
	E^4_\delta + S'(\xi_\delta)\, \pd_x E^3_\delta
	\\ &
	= -\frac{1}{2} S''(\xi_\delta) 
	\bigl(\bigl[\KK, \JJ_\delta \bigr](\xi)\bigr)^2
	+S''(\xi_\delta) \pd_x \sigma 
	\xi_\delta \bigl[\KK,\JJ_\delta\bigr](\xi)  
	+\pd_x \bigl(\sigma S'(\xi_\delta)
	\bigl[\KK,\JJ_\delta\bigr](\xi)\bigr) 
	\\ & \quad 
	-S'(\xi_\delta) \KK \bigl[\KK,\JJ_\delta\bigr](\xi) 
	+\frac{1}{2}S'(\xi_\delta)
	\bigl(\bigl[\KK, \JJ_\delta\bigr](\KK \xi) 
	- \KK \bigl[\KK ,\JJ_\delta\bigr](\xi)\bigr)
	\\ & 
	= -\frac{1}{2} S''(\xi_\delta) 
	\bigl(\bigl[\KK, \JJ_\delta\bigr](\xi)\bigr)^2
	+S''(\xi_\delta) \pd_x \sigma \xi_\delta 
	\bigl[\KK,\JJ_\delta\bigr](\xi) 
	\\ & \quad 
	+ \pd_x \bigl(\sigma S'(\xi_\delta) 
	\bigl[\KK,\JJ_\delta\bigr](\xi)\bigr)
	+\frac{1}{2}S'(\xi_\delta)
	\bigl(\bigl[\KK, \JJ_\delta\bigr](\KK \xi) 
	- \KK \bigl[\KK ,\JJ_\delta\bigr](\xi)\bigr)
	\\ & 
	=-\frac{1}{2} S''(\xi_\delta) 
	\bigl(\bigl[\KK,\JJ_\delta\bigr](\xi)\bigr)^2
	+S''(\xi_\delta) \pd_x \sigma \xi_\delta 
	\bigl[\KK,\JJ_\delta\bigr](\xi)
	\\ & \quad 
	+\pd_x \big(\sigma S'(\xi_\delta)
	\bigl[\KK,\JJ_\delta\bigr](\xi)\bigr)
	+\frac{1}{2}S'(\xi_\delta)
	\Bigl[\bigl[\KK,\JJ_\delta\bigr],\KK\Bigr](\xi).
\end{align*}
For the term $\pd_x \big(\sigma S'(\xi_\delta)
	\bigl[\KK,\JJ_\delta\bigr](\xi)\bigr)$, we integrate-
by-parts in $x$ against $\varphi$. 
We know already that 
$\bigl[\KK,\JJ_\delta\bigr](\xi) =  \pd_x E^2_\delta \todelta 0$ 
in $L^2(\Omega\times [0,T]\times \T)$, cf.~Lemma 
\ref{thm:commutator1}. (So in fact cancellation only 
occurs between the $S'(\xi_\delta)\pd_x E^3_\delta$ 
and the $S''(\xi_\delta) \pd_x(\sigma\xi_\delta) \,\pd_x E^2_\delta$ terms.)
Convergence of the double commutator 
bracket is established in Lemma~\ref{thm:commutator2}. 
Now the entire claim \eqref{eq:commutator_converge2} 
follows, as $S \in C^1(\R) \cap \dot{W}^{2,\infty}(\R)$ 
and $S'(r) = O(r)$, so that $S'(\xi_\delta) \in L^2(\Omega 
\times [0,T]\times \T)$ and $S''(\xi_\delta) 
\in L^\infty(\Omega \times [0,T]\times \T)$.
\end{proof}

\subsection{Pathwise uniqueness}\label{sec:pathwise_unique}
To quickly establish pathwise uniqueness 
of solutions in the energy space 
$L^2([0,T];H^1(\T))$, we would need bounds 
on the solution in $L^2\left([0,T];
L^\infty(\Omega;W^{1,\infty}(\T))\right)$ 
to control exponential moments of cubic terms that 
appear in the exponent resulting from a Gronwall inequality. 
Unfortunately, such bounds are not available unless 
$T$ is replaced by a stopping time 
$\eta_R<T$ that converges a.s.~to $T$ as $R \to \infty$. 
However, integrating up to a stopping time, 
it is not possible to interchange the 
expectation (integral w.r.t.~$\d \mathbb{P}$) 
with the temporal integral and appeal 
to a standard Gronwall inequality. We will therefore 
rely on the stochastic Gronwall inequalities, see 
Lemmas \ref{thm:stochastic_gronwall} and \ref{thm:stochastic_gronwall_st}. 
Having shown uniqueness on $\bigl[0,\eta_R\bigr]$, we 
send $R \to \infty$ to conclude uniqueness on $[0,T]$.

\begin{thm}[Pathwise uniqueness in $H^1$]
\label{thm:st_uniqueness}
Let $u$, $v$ be strong $H^1$ solutions to 
the viscous stochastic Camassa--Holm equation \eqref{eq:u_ch_ep} 
with $\sigma \in W^{2,\infty}(\T)$ and intial condition 
$u_0 \in L^{p_0}(\Omega; H^1(\T))$ 
for some $p_0 > 4$. Then $\Ex 
\norm{u-v}_{L^\infty([0,T];H^1(\T))}=0$.
\end{thm}

\begin{proof}
Suppose $u$ and $v$ are strong solutions defined 
relative to the (same) stochastic basis $\bigl(\Omega,
\mathcal{F},\{\mathcal{F}_t\}_{t \ge 0},\mathbb{P}\bigr)$ 
and Brownian motion $W$. The difference $w=u-v$ 
obeys
\begin{align*}
	0 & = \d w-\ep \pd_x^2 w\,\d t
	+\bk{u \,\pd_x u-v \,\pd_x v} \,\d t
	\\ &\quad 
	+\pd_x K*\bk{u^ 2-v^2+
	\frac12\bk{\bk{\pd_x u}^2-\bk{\pd_x v}^2}}\,\d t
	\\ &\quad
	- \frac12 \sigma\,\pd_x\bk{\sigma\,\pd_x w} \,\d t
	+ \sigma \,\pd_x w\,\d W.
\end{align*}
The spatial derivative satisfies 
\begin{align*}
	0 & = \d \pd_xw-\ep \pd_x^3 w\,\d t
	+\pd_x \bk{u \,\pd_x u-v\,\pd_x v} \,\d t
	\\ & \quad
	+ \pd_x^2 K*\bk{u^2- v^2
	+\frac12 \bk{\bk{\pd_x u}^2-\bk{\pd_x v}^2} }\,\d t
	\\ & \quad
	-\frac12 \pd_x\bk{\sigma\,\pd_x\bk{\sigma\,\pd_x w}}\,\d t
	+\pd_x\bk{\sigma \,\pd_x w}\,\d W.
\end{align*}
Recall that $J_\delta$ is a standard Friedrichs mollifier
and $w_\delta=w*J_\delta$. 
We convolve both the foregoing equations 
against $J_\delta$ in order to obtain SPDEs that 
can be understood in the pointwise sense (in $x$):
\begin{equation}\label{eq:u_ch_ep_delta}
	\begin{aligned}
		0&= \d w_\delta-\ep \pd_x^2 w_\delta\,\d t
		+\bk{u_\delta \,\pd_x u_\delta
		-v_\delta \,\pd_x v_\delta} \,\d t
		\\ &\quad
		+\pd_x K*\bk{u_\delta^2-v_\delta^2
		+\frac12 \bk{\bk{\pd_x u_\delta}^2
		-\bk{\pd_x v_\delta}^2}}\,\d t
		\\ &\quad
		- \frac12 \sigma\,\pd_x\bk{\sigma\,\pd_x w_\delta} \,\d t
		+ \sigma \,\pd_x w_\delta\,\d W 
		+ E^1_\delta \,\d t + E^2_\delta\,\d W 
		+ E^3_\delta\,\d t,
		\\[3mm] 
		0 & = \d \pd_xw_\delta-\ep \pd_x^3 w_\delta \,\d t
		+\pd_x \bk{u_\delta \,\pd_x u_\delta
		-v_\delta \,\pd_x v_\delta } \,\d t
		\\ & \quad+ \pd_x^2 K*
		\bk{u_\delta^ 2 - v_\delta^2 
		+ \frac12 \bk{\bk{\pd_x u_\delta}^2
		-\bk{\pd_x v_\delta}^2} }\,\d t
		\\ & \quad
		-\frac12 \pd_x\bk{\sigma\,\pd_x
		\bk{\sigma\,\pd_x w_\delta}} \,\d t
		+\pd_x\bk{\sigma \,\pd_x w_\delta}\,\d W
		\\ & \quad
		+\pd_x E^1_\delta\,\d t+\pd_x E^2_\delta\,\d W 
		+\pd_x E^3_\delta\,\d t,
	\end{aligned}
\end{equation}
where $E^1_\delta$, $E^2_\delta$ and $E^3_\delta$ are as in \eqref{eq:commutator_defin1}.

Apart from the technical addition of the 
commutator terms $E^i_\delta$, $i=1,2,3$, uniqueness follows 
from a straightforward calculation. 
The quantities $u_\delta$, $v_\delta$, 
and $w_\delta$ are necessarily
$\tilde{\mathbb{P}}$-almost surely in 
$C([0,T];H^1(\T))$ by the inclusion of $u$, $v$, 
and consequently $w$ in $C([0,T];H^1_w(\T))$. 
As in deriving the energy inequality of 
Proposition \ref{thm:galerkin_H1} for 
the Galerkin approximations, 
repeated applications of the 
(finite-dimensional) It\^o formula gives
\begin{align}
	\frac12 &\norm{w_\delta(t)}_{H^1(\T)}^2
	+\ep \int_0^t\norm{\pd_x w_\delta(s)}_{H^1(\T)}^2
	\,\d s \label{eq:difference_1}
	\\ & 
	=-\int_0^t \int_\T \big[w_\delta 
	\bk{u_\delta \,\pd_x u_\delta-v_\delta \,\pd_x v_\delta}
	+\pd_x w_\delta\, \pd_x \bk{u_\delta 
	\,\pd_x u_\delta-v_\delta \,\pd_x v_\delta}\big] 
	\,\d x \,\d s\notag
	\\ & \quad 
	-\int_0^t \int_\T \Big[w_\delta \,\pd_x K*
	\bk{u_\delta^2-v_\delta^2
	+\frac12 \bk{\bk{\pd_x u_\delta}^2
	-\bk{\pd_x v_\delta}^2}} \notag
	\\ &\quad\qquad\qquad\qquad\quad
	+\pd_x w_\delta \, \pd_x^2 K*
	\bk{u_\delta^ 2-v_\delta^2
	+\frac12 \bk{\bk{\pd_x u_\delta}^2
	-\bk{\pd_x v_\delta}^2} }\Big]\,\d x \,\d s\notag
	\\ &\quad
	+\frac12 \int_0^t \int_\T\big[ \sigma w_\delta
	\,\pd_x\bk{\sigma\,\pd_x w_\delta}
	+\pd_x w_\delta\, \pd_x\bk{\sigma
	\,\pd_x\bk{\sigma\,\pd_x w_\delta}} \big]\,\d x\,\d s\notag
	\\ & \quad
	-\int_0^t \int_\T \big[w_\delta 
	\bigl(E^1_\delta + E^3_\delta\bigr) 
	+ \pd_x w_\delta \bk{\pd_x E^1_\delta+\pd_x E^3_\delta}\big]
	\,\d x\,\d s\notag
	\\ &\quad
	+\frac12 \int_0^t \int_\T\big[\bk{\sigma \,\pd_x w_\delta 
	+ E^2_\delta}^2 + \bk{\pd_x \bk{\sigma\,\pd_x w_\delta}
	+\pd_x E^2_\delta}^2\big]\,\d x\,\d s\notag
	\\ &\quad
	+ \int_0^t \int_\T\big[\sigma \, w_\delta\,\pd_x w_\delta 
	+\pd_x w_\delta\, \pd_x\bk{\sigma \,\pd_x w_\delta}\big]\,\d x\,\d W\notag
	\\ &\quad
	+\int_0^t \int_\T \big[w_\delta E^2_\delta
	+\pd_x w_\delta \,\pd_x E^2_\delta\big]\,\d x\,\d W\notag
	\\ & 
	=: I_1^\delta+I_2^\delta+I_3^\delta 
	+I_4^\delta+I_5^\delta+M_1^\delta+M_2^\delta,
	\notag
\end{align}
recalling that $\norm{w_\delta(0)}_{H^1(\T)}^2=0$. 
We split the remaining analysis into two parts --- one 
for $I_1^\delta$ and $I_2^\delta$, consisting of the mollified terms 
from the ``deterministic" part of the equation, 
and another for the remaining integrals consisting 
of the effects of the convective noise and all the 
mollification error terms.

\medskip
{\it 1. Estimating $I_1^\delta$ and $I_2^\delta$.} 

For $I_1^\delta$ we have
\begin{align*}
	\abs{I_1^\delta} & = 
	\frac12 \abs{\int_0^t\int_\T \big[w_\delta\,\pd_x 
	\bk{u_\delta^2-v_\delta^2}+\pd_x w_\delta\, 
	\pd_x^2\bk{u_\delta^2-v_\delta^2}\big] \,\d x \,\d s}
	\\ & = 
	\frac12 \abs{\int_0^t \int_\T 
	\big[\pd_x w_\delta \, w_\delta \bigl(u_\delta+v_\delta\bigr)
	+ \pd_x^2 w_\delta
	\, \pd_x\bk{w_\delta 
	\bigl(u_\delta+v_\delta\bigr)}\big] \,\d x \,\d s}
	\\ & 
	\le \frac12\int_0^t \Big[
	\norm{\pd_x w_\delta}_{L^2(\T)}
	\norm{w_\delta}_{L^2(\T)}
	\norm{u_\delta+v_\delta}_{L^\infty(\T)}
	\\ &\qquad\qquad\quad
	+\norm{\pd_x^2w_\delta}_{L^2(\T)}
	\norm{w_\delta}_{H^1(\T)}
	\norm{u_\delta+v_\delta}_{W^{1,\infty}(\T)}\Big]\,\d s
	\\ & 
	\le \frac12\int_0^t \Big[\norm{w_\delta}_{H^1(\T)}^2 
	\norm{u_\delta+v_\delta}_{L^\infty(\T)}
	\\ &\qquad\qquad\quad
	+\frac\ep2\norm{\pd_x^2w_\delta}_{L^2(\T)}^2
	+\frac1{2\ep}\norm{w_\delta}_{H^1(\T)}^2
	\norm{u_\delta+v_\delta}_{W^{1,\infty}(\T)}^2\Big]\,\d s.
\end{align*}
Using the identity $\bk{K-\pd_x ^2 K}*f= f$,
\begin{align*}
	\abs{I_2^\delta} & = 
	\abs{\int_0^t \int_\T \pd_x w_\delta
	\,\bk{u_\delta^ 2-v_\delta^2
	+\frac12 \bk{\bk{\pd_x u_\delta}^2
	-\bk{\pd_x v_\delta}^2}} \,\d x\,\d s}
	\\ & 
	=\abs{\int_0^t \int_\T \pd_x w_\delta 
	\,\bk{ w_\delta \bigl(u_\delta+v_\delta\bigr)
	+\frac12 \pd_x w_\delta \,\pd_x 
	\bigl(u_\delta+v_\delta\bigr)} \,\d x\,\d s}
	\\ & 
	= \abs{\int_0^t \int_\T \big[w_\delta
	\, \pd_x w_\delta\, \bigl(u_\delta + v_\delta\bigr)
	-\pd_x w_\delta\, \pd^2_x w_\delta 
	\bigl(u_\delta + v_\delta\bigr)\big] \,\d x\,\d s}
	\\ & 
	\le \int_0^t\Big[\norm{w_\delta}_{H^1(\T)}^2
	\norm{u_\delta+v_\delta}_{L^\infty(\T)}
	\\&\qquad\qquad\quad
	+\frac\ep4 \norm{\pd_x^2 w_\delta}_{L^2(\T)}^2 
	+\frac1{4\ep} \norm{\pd_x w_\delta}_{L^2(\T)}^2
	\norm{u_\delta+v_\delta}_{L^\infty(\T)}^2\Big] \,\d s.
\end{align*}
In the right-hand sides of $\abs{I_1^\delta}$ 
and $\abs{I_2^\delta}$, the terms involving 
$\ep\int_0^t \norm{\pd_x^2 w_\delta}_{L^2(\T)}^2\,\d s$ 
can be absorbed into \eqref{eq:difference_1}.

\medskip
{\it 2. Estimating $I_3^\delta+I_4^\delta+I_5^\delta$.}

Next, we have
\begin{align*}
	I_3^\delta & =
	-\frac12 \int_0^t \int_\T 
	\abs{\sigma\,\pd_x w_\delta} ^2 \,\d x\,\d s
	-\frac12\int_0^t  \int_\T \sigma\, \pd_x \sigma
	\, w_\delta\, \pd_x w_\delta\,\d x \,\d s
	\\ & \quad
	-\frac12 \int_0^t \int_\T 
	\abs{\sigma \pd_x^2 w_\delta}^2\,\d x\,\d s 
	-\frac12 \int_0^t \int_\T \sigma 
	\,\pd_x \sigma \,\pd_x w_\delta
	\,\pd_x^2 w_\delta \,\d x\,\d s
	\\ & = -\frac12 \int_0^t 
	\int_\T \big[\abs{\sigma\,\pd_x w_\delta} ^2
	+\abs{\sigma \pd_x^2 w_\delta}^2\big]\,\d x\,\d s 
	-\frac12\int_0^t \int_\T \sigma\, 
	\pd_x \sigma\, w_\delta\, \pd_x w_\delta\,\d x \,\d s
	\\ & \quad
	+\frac14 \int_0^t \int_\T 
	\pd_x \bk{\sigma \,\pd_x \sigma} 
	\,\abs{\pd_x w_\delta}^2\,\d x\,\d s.
\end{align*}

We add $I_4^\delta$ and $I_5^\delta$ together and use 
Lemma \ref{thm:commutator1} and Proposition 
\ref{thm:commutator3}, yielding
\begin{align*}
	I_4^\delta+I_5^\delta 
	& = \int_0^t \int_\T \big[
	-w_\delta \bk{E^1_\delta 
	+E^3_\delta}+\frac12\bk{\sigma 
	\pd_x w_\delta+E^2_\delta}^2\big]\,\d x \,\d s
	\\ &\quad
	+\int_0^t \int_\T\big[ -\pd_x w_\delta\, \pd_x E^1_\delta 
	+\frac12 \abs{\pd_x \bk{\sigma \,\pd_x w_\delta}}^2 \big]
	\,\d x \,\d s 
	\\ & \quad 
	+\int_0^t \int_\T\big[-\pd_x w_\delta \,\pd_x E^3_\delta
	+\frac12 \abs{\pd_x E^2_\delta}^2
	+\pd_x \bk{\sigma \,\pd_x w_\delta}
	\,\pd_x E^2_\delta\big] \,\d x \,\d s
	\\ & 
	=:I_{4+5,1}^\delta+I_{4+5,2}^\delta
	+\int_0^t I_{4+5,3}^\delta \, \d s.
\end{align*}
For $I_{4+5,1}^\delta$, we have
\begin{align*}
	\abs{I_{4+5,1}^\delta} 
	& \le \norm{w_\delta}_{L^2([0,t]\times \T)}^2
	+\norm{E^1_\delta + E^3_\delta}_{L^2([0,t]\times \T)}^2
	\\ &\quad
	+\norm{\sigma}_{L^\infty(\T)}^2
	\norm{\pd_x w_\delta}_{L^2([0,t]\times \T)}^2
	+\norm{E^2_\delta}_{L^2([0,t]\times \T)}^2
	\\ & \le 
	C_\sigma \norm{w_\delta}_{L^2([0,t];H^1(\T))}^2
	+2\norm{E_\delta^1}_{L^2([0,t]\times \T)}^2
	\\ & \quad+\norm{E_\delta^2}_{L^2([0,t]\times \T)}^2 
	+ 2\norm{E_\delta^3}_{L^2([0,t]\times \T)}^2.
\end{align*}
Furthermore,
\begin{equation*}
	\begin{aligned}
		I_{4+5,2}^\delta &= \int_0^t\int_\T 
		\big[-\pd_x w_\delta \,\pd_x E^1_\delta
		+\sigma \pd_x \sigma \,\pd_x w_\delta 
		\,\pd_x^2 w_\delta\big]\,\d x\,\d s
		\\ &\quad
		+\frac12 \int_0^t\int_\T \big[\abs{\sigma
		\, \pd_x^2 w_\delta}^2
		+\abs{\pd_x \sigma\,\pd_xw_\delta}^2\big]\,\d x\,\d s
		\\ & =  - \int_0^t\int_\T \big[\pd_x w_\delta 
		\,\pd_x E^1_\delta + \frac12 
		\,\pd_x \bk{\sigma \pd_x \sigma}
		\,\abs{\pd_x w_\delta}^2\big] \,\d x\,\d s
		\\ &\quad 
		+\frac12 \int_0^t\int_\T \big[\abs{\sigma
		\, \pd_x^2 w_\delta}^2
		+\abs{\pd_x \sigma\,\pd_xw_\delta}^2\big]\,\d x\,\d s.
	\end{aligned}
\end{equation*}

Adding $I_{4+5,2}^\delta$ to $I_3^\delta$, we obtain 
\begin{align*}
	I_3^\delta  + I_{4 + 5, 2}^\delta
	& = -\frac12 \int_0^t \int_\T 
	\abs{\sigma\,\pd_x w_\delta} ^2\,\d x\,\d s 
	-\frac12\int_0^t  \int_\T 
	\sigma\, \pd_x \sigma\, w_\delta
	\, \pd_x w_\delta\,\d x \,\d s
	\\ &\quad
	+ \int_0^t\int_\T \pd_x^2 w_\delta 
	\,E^1_\delta \,\d x\,\d s
	+\frac12 \int_0^t\int_\T 
	\abs{\pd_x \sigma\,\pd_x w_\delta}^2\,\d x\,\d s \\
	&\quad  -\frac14 \int_0^t\int_\T
	\pd_x(\sigma \pd_x\sigma)(\pd_x w_\delta)^2\,\d x\,\d s 
	\\ & \le \bk{ \norm{\pd_x^2 \sigma}_{L^\infty(\T)}^2 
	+\norm{\pd_x \sigma}_{L^\infty(\T)}^2
	+\norm{\sigma}_{L^2(\T)}^2+1}
	\norm{w_\delta}_{L^2([0,t];H^1(\T))}^2
	\\ &\quad
	+\frac\ep8 \norm{\pd_x^2 w_\delta}_{L^2([0,t]\times \T)}^2
	+\frac{16}\ep\norm{E^1_\delta}_{L^2([0,t]\times \T)}^2.
\end{align*}
The term involving 
$\ep\norm{\pd_x^2 w_\delta}_{L^2([0,t]\times \T)}^2$ 
can be absorbed into \eqref{eq:difference_1}.

Given Proposition \ref{thm:commutator3}, 
taking $S(r) = r^2/2$, we immediately get that
$$
\Ex\int_0^T \abs{I_{4+5,3}^\delta} \,\d t 
\todelta 0.
$$
Therefore, by Lemma \ref{thm:commutator1},
$$
I_3^\delta+I_4^\delta +I_5^\delta \le C_\sigma  
\norm{w_\delta}_{L^2([0,t];H^1(\T))}^2 
+\varrho_\delta(t),
$$
where $\varrho_\delta(t) \ge 0$ and 
$\sup\limits_{t \in [0,T]}
\Ex \, \varrho_\delta(t) \todelta 0$.

\medskip
{\it 3. Conclusion.}

Putting the estimates for $I_1^\delta$ 
through $I_5^\delta$ together we arrive at
\begin{equation}\label{eq:unique_pre_gronwall}
	\begin{aligned}
		 \frac12&\norm{w_\delta(t)}_{H^1(\T)}^2
		+\frac\ep4 \int_0^t 
		\norm{\pd_x w_\delta(s)}_{H^1(\T)}^2\,\d s
		\\ & \qquad\qquad\qquad\qquad
		\le C_{\ep, \sigma} \int_0^t 
		\norm{w_\delta}_{H^1(\T)}^2
		\bk{1+\norm{u_\delta
		+v_\delta}_{W^{1,\infty}(\T)}^2}\,\d s
		\\ & \qquad\qquad\qquad\qquad\quad 
		+M_\delta(t)+\varrho_\delta(t),
	\end{aligned}
\end{equation}
where $M_\delta(t):=M_1^\delta(t)+M_2^\delta(t)$.

The estimates on $I_1^\delta$ 
through $I_5^\delta$ also show, by the equality \eqref{eq:difference_1},
that the process $t \mapsto \norm{w_\delta(t)}_{H^1(\T)}^2$
is $\mathbb{P}$-almost surely continuous.

By Young's convolution inequality, 
\begin{align*}
	\int_0^t \norm{u_\delta(s)
	+v_\delta(s)}_{W^{1,\infty}(\T)}^2\,\d s 
	\le \int_0^t \norm{J_\delta}_{L^1(\T)}^2
	\norm{u(s)+v(s)}_{W^{1,\infty}(\T)}^2\,\d s,
\end{align*}
and $\norm{J_\delta}_{L^1(\T)}=1$ by construction. 
Let us therefore introduce the stopping time 
\begin{align*}
	\eta_R=\inf\seq{t \in \R_+: \int_0^{t \wedge T}
	\norm{u(s)+v(s)}_{W^{1,\infty}(\T)}^2\,\d s>R},
\end{align*}
with $\eta_R =\infty $ if the set on the right-hand side 
is empty.  We have that $\eta_R \toR T$ a.s.~(for  
fixed $\ep>0$). Indeed, from part (c) of Definition 
\ref{def:wk_sol} --- which says that 
$u,v\in L^2_\omega L^2_t H^2_x$ --- and the embedding 
$H^2(\T) \hookrightarrow W^{1,\infty}(\T)$, 
\begin{equation}\label{eq:stoptime_2_bd}
	\begin{aligned}
		\prob\bk{\seq{\eta_R<T}} 
		& \le \prob\bk{\seq{\int_0^T
		\norm{u(t)+v(t)}_{W^{1,\infty}(\T)}^2\,\d t>R}}
		\\ & 
		\le \frac1R \Ex \int_0^T 
		\norm{u(t)+v(t)}_{W^{1,\infty}(\T)}^2\,\d t
		\\ & 
		\le \frac{\tilde{C}}{R}\Ex \int_0^T
		\norm{u(t)}_{H^2(\T)}^2
		+\norm{v(t)}_{H^2(\T)}^2\,\d t
		\le \frac{C}{R}\toR 0,
	\end{aligned}
\end{equation}
where $C$ depends on $\ep$, cf.~\eqref{eq:galerkin_unifbd1}. 

With this stopping time, $M_\delta(t \wedge \eta_R)$ 
is a square-integrable martingale term.

Specifying $t=\eta_R$ in \eqref{eq:unique_pre_gronwall}, 
noting that
$$
\int_0^{\eta_R} 1+\norm{u_\delta(s)
+v_\delta(s)}_{W^{1,\infty}(\T)}^2\,\d s \le T+R,
$$
we can use the stochastic Gronwall inequality 
(Lemma \ref{thm:stochastic_gronwall_st} 
with $\nu = 1/2$ and any $1/2<r<1$) 
to conclude that
\begin{equation}\label{eq:H1-unique-tmp}
	\lim_{\delta \to 0}\bk{{\Ex} 
	\sup_{t \in \left[0,\eta_R\right]}
	\norm{w_\delta (t)}_{H^1(\T)}}^2=0.
\end{equation}
Recalling that the stopping times $\eta_R\toR T$ are 
independent of $\delta$ and, by the properties of 
mollification, $w_\delta(t)\to w(t)$ in $H^1(\T)$ 
for $\d \mathbb{P}\otimes \d t$-a.e.~$(\omega,t)\in 
\Omega\times [0,T]$, combining \eqref{eq:H1-unique-tmp} 
with the dominated convergence theorem implies that 
$$
\Ex
\norm{w}_{L^\infty([0,T];H^1(\T))}=0.
$$
\end{proof}

\subsection{Strong $H^1$-existence}\label{sec:H1_existence}

To establish the existence of strong $H^1$ solutions, 
and thereby concluding the proof of Theorem \ref{thm:main1}, 
we shall use an infinite dimensional version 
of the Yamada--Watanabe principle, following from 
Lemma \ref{thm:gyongy_krylov}.  As the 
path space $\mathcal{X}$ constructed immediately 
following the definitions \eqref{eq:pathspaces} is not 
a Polish space, we provide a slightly refined argument.

\begin{proof}[Concluding the proof of 
Theorem \ref{thm:main1}]
Recalling \eqref{eq:pathspaces}, we consider 
the extended path space
$\mathcal{Y}:= \bk{\mathcal{X}_{u,s} \times \mathcal{X}_{u,w}} \times 
\bk{\mathcal{X}_{u,s} \times \mathcal{X}_{u,w}}
\times \mathcal{X}_{W} \times \mathcal{X}_{0}$. 
Let $\seq{u_n}$ be the Galerkin solutions 
with initial conditions $\seq{\bm{\Pi}_n u_0}$, 
cf.~\eqref{eq:galerkin_n}. Set
\begin{align*}
\mu_{u,s}^n := \bk{u_n: \Omega \to \mathcal{X}_{u,s}}_* \mathbb{P},&
\quad 
\mu_{u,w}^n := \bk{u_n: \Omega \to \mathcal{X}_{u,w}}_* \mathbb{P},
\\
\mu_W^n := \bk{W}_* \mathbb{P},& 
\quad 
\mu_0^n := \bk{\bm{\Pi}_n u_0}_*\prob,
\end{align*}
as probability measures respectively 
on $\mathcal{X}_{u,s}$, $\mathcal{X}_{u,w}$, 
$\mathcal{X}_W$, and $\mathcal{X}_0$. 
Finally, define on $\mathcal{Y}$ 
the product measure
$$
{\mu}^{m,n} := \mu_{u,s}^{m} \otimes 
\mu_{u,w}^{m} \otimes \mu_{u,s}^{n} \otimes 
\mu_{u,w}^{n} \otimes \mu_W^{m} \otimes \mu_0^{m} .
$$

Consider an arbitrary subsequence 
$\seq{\mu^{m_k,n_k}}_{k \in \N}$ so that 
$\seq{m_k}_{k \in \N}$ and $\{n_k\}_{k \in \N}$ are 
increasing sequences. The tightness of 
$\seq{\mu_j^n}$ in $\mathcal{X}_j$,  taking 
$j$ equal to $(u,s)$, $(u,w)$, $W$, $0$, respectively, see 
Lemma \ref{thm:jointlaw_tightness}, implies 
the tightness of $\seq{\mu^{m_k,n_k}}_{k \in \N}$ on 
$\mathcal{Y}$. By Prohorov's theorem, this subsequence 
converges weakly to a probability measure $\tilde{\mu}$ 
on $\mathcal{Y}$.

By the Skorokhod--Jakubowski representation 
theorem (cf.~Theorem \ref{thm:skorohod}) (and 
the identification of Lemma \ref{thm:identification_lemma}), 
there exist a new probability space $\bigl(\tilde{\Omega},
\tilde{\mathcal{F}},\tilde{\prob}\bigr)$ and, passing 
to a further subsequence (not relabelled), new 
random variables
\begin{align}\label{eq:subsequence1}
	\bigl(\tilde{u}_{m_k}, \tilde{u}_{m_k},
	\tilde{u}_{n_k},\tilde{u}_{n_k},\tilde{W},
	\tilde{u}_{0,m_k}\bigr), \quad 
	\text{with joint laws $\mu^{m_k,n_k}$}, 
\end{align}
converging in $\mathcal{Y}$ to a limit 
$\bigl(\tilde{u}^\alpha, \tilde{u}^\alpha,
\tilde{u}^\beta, \tilde{u}^\beta, \tilde{W}, 
\tilde{u}_0\bigr)$, $\tilde{\mathbb{P}}$-a.s., 
whose joint law is $\tilde{\mu}$.

Construct now a (filtered) stochastic 
basis $\tilde{\mathcal{S}}$ as in the paragraph 
following Theorem \ref{thm:skorohod}. It then 
follows (as in Theorem \ref{thm:existence_H1}) 
that $\bigl(\tilde{u}^\alpha, \tilde{W}\bigr)$ 
and $\bigl(\tilde{u}^\beta, \tilde{W}\bigr)$ 
are weak (martingale) $H^1$ solutions with initial condition 
$\tilde{u}_0$ on $\tilde{\mathcal{S}}$. Therefore, 
by pathwise uniqueness (cf.~Theorem \ref{thm:st_uniqueness}), 
\begin{align*}
	\tilde{\mu}\Bigl((u,u,v,v) \in 
	\mathcal{X}_{u,s} \times &\mathcal{X}_{u,w}
	\times \mathcal{X}_{u,s} 
	\times \mathcal{X}_{u,w}:u=v\Bigr) 
	\\ &	
	=\tilde{\mathbb{P}}\Big(\tilde{u}^\alpha 
	=\tilde{u}^\beta \mbox{ in }
	\mathcal{X}_{u,s} \mbox{ and in }\mathcal{X}_{u,w} \Big)=1.
\end{align*}

We have constructed a pair of weak $H^1$ 
solutions that coincide almost surely. 
To use the Gy\"ongy--Krylov theorem to conclude 
convergence on the initial probability 
space $(\Omega, \mathcal{F}, \mathbb{P})$, 
a Polish space is needed,  whereas our path space 
$\mathcal{Y}$ is only quasi-Polish. We can, however, map 
$\mathcal{Y}$ into the Polish space $[-1,1]^\mathbb{N}$. 
This is the technique Jakubowski used  
to extend the Skorokhod representation theorem 
to non-metric (quasi-Polish) spaces \cite{Jak1998}.  
Indeed, because $\mathcal{Y}$ is a quasi-Polish space,
there is a countable family of continuous functions 
$\seq{f_\ell:\mathcal{Y}\to [-1,1]}_{\ell \in \mathbb{N}}$ 
that separate points, see \eqref{eq:countable}. 
Introduce the continuous map 
$f:\mathcal{Y} \to [-1,1]^\mathbb{N}$ 
(equipped with the product topology) by 
$f:u\mapsto \seq{f_\ell(u)}_{\ell \in \mathbb{N}}$. 
The map $f$ is a measurable bijective function 
when restricted to a $\sigma$-compact subspace of 
$\mathcal{Y}$ (i.e., a countable union of compact 
subspaces) of $\mathcal{Y}$, see \cite[Sec.~2]{Jak1998} and 
\cite[Cor.~3.1.14, p. 126]{Eng1989}.

Considering the first and third 
entries $(u_{m_k}, u_{n_k})$ 
of \eqref{eq:subsequence1}, also recalling 
their limits $\bigl(u^\alpha, u^\alpha\bigr)$,
 and using $f$ defined 
in the foregoing paragraph, we find by continuity of $f$ that 
$\bigl(f(u_{m_k}), f(u_{n_k})\bigr)$ converge 
in distribution (law) to $(f(u^\alpha),f(u^\alpha))$ 
as $k \to \infty$. By Lemma \ref{thm:gyongy_krylov}, 
there is a subsequence $\bigl\{f(u_{n_{k_j}})\bigr\}_{j \in\N}$ 
that converges in probability. Since $f$ separates 
points of $\mathcal{Y}$, we must necessarily have that 
also $\bigl\{u_{n_{k_j}}\bigr\}_{j \in\N}$ converges 
in probability on $\bigl(\Omega,\mathcal{F},\mathbb{P}\bigr)$, 
and hence $\mathbb{P}$-almost surely along 
a further subsequence.
\end{proof}

\subsection{Strong temporal continuity}\label{sec:temporalcont}

In this subsection, we finally establish the 
strong continuity of solutions as stipulated in 
(c) of Definition \ref{def:wk_sol}, completing 
the omission described in Remark \ref{rem:wk_time_cont}.

\begin{lem}[Energy equality]\label{thm:energy_eq}
Let ${u}$ be the solution found in 
Theorem \ref{thm:existence_H1}, and ${q}$ 
be the weak $x$-derivative of ${u}$. 
The following energy equality holds 
for every $s,t \in [0,T]$:
\begin{equation}\label{eq:energyeq_exp}
	\begin{aligned}
		&{\Ex} \norm{{u}(r)}_{H^1(\T)}^2 
		\bigg|_s^t+  \ep {\Ex} \int_s^t \norm{{q}(r)}_{H^1(\T)}^2\,\d r
		\\ &\qquad =-\Ex   \int_s^t \int_\T 
		\sigma \,\pd_x \sigma\,u \,q \,\d x \,d r
		+\Ex \int_s^t \int_\T \bk{ \frac14 \pd_x^2 \sigma^2
		-\abs{\pd_x \sigma}^2} q^2\,\d x \,\d r.
	\end{aligned}
\end{equation}
\end{lem}

\begin{proof}
We can derive the energy inequality 
by mollification as in \eqref{eq:u_ch_ep_delta} (e.g., 
by taking $v_0, v \equiv 0$, which is clearly a solution). 
That is, we have
\begin{align*}
	0 & = \d u_\delta-\ep \pd_x^2 u_\delta
	+\bk{u_\delta \,\pd_x u_\delta} \,\d t 
	+\pd_x K*\bk{u_\delta^2 +\frac12 \bk{\pd_x u_\delta}^2}  \,\d t\\ 
	&\quad 	- \frac12 \sigma\,\pd_x\bk{\sigma\,\pd_x u_\delta} \,\d t
	+ \bk{\sigma \,\pd_x u_\delta + E^2_\delta}\,\d W 
	+ E^1_\delta \,\d t + E^3_\delta\,\d t,
	\\[3mm] 
	0 & = \d q_\delta-\ep \pd_x^2 q_\delta 
	+\pd_x \bk{u_\delta \,q_\delta } \,\d t
	+ \pd_x^2 K*
	\bk{u_\delta^ 2 	+ \frac12 \bk{q_\delta}^2 }\,\d t
	\\ & \quad
	-\frac12 \pd_x\bk{\sigma\,\pd_x
	\bk{\sigma\,q_\delta}} \,\d t
	+\bk{ \pd_x\bk{\sigma \,q_\delta}+ \pd_x E^2_\delta }\,\d W
	\\ & \quad 
	+\pd_x E^1_\delta\,\d t +\pd_x E^3_\delta\,\d t,
\end{align*}
where, again,
\begin{align*}
	E^1_\delta & :=  (u \,\pd_x u) *J_\delta- u_\delta \,\pd_x u_\delta \\
	 & \quad +\pd_x K*\bk{u^2+\frac12 \bk{\pd_xu}^2}*J_\delta
	 -\pd_x K *\bk{u_\delta^ 2+ \frac12 \bk{\pd_x u_\delta}^2},
	 \\
	 E^2_\delta & :=  \bigl(\sigma \,\pd_x u\big)*J_\delta-\sigma \,\pd_x u_\delta,
	 \\
	 E^3_\delta & := 	-\frac12 \bigl(\sigma\,\pd_x\bk{\sigma\,\pd_x u}\bigr)
	*J_\delta +\frac12 \sigma 	\,\pd_x\bk{\sigma\,\pd_x u_\delta}.
\end{align*}

We multiply the equation for $\d u_\delta$ by $u_\delta$ 
and the equation for $\d q_\delta$ by $q_\delta$. 
Manipulations using the pointwise It\^o formula 
as in \eqref{eq:H1_intermediate1} of Proposition 
\ref{thm:galerkin_H1} then leads us upon integration to:
\begin{align*}
\frac12 \norm{u_\delta}_{H^1(\T)}^2 \bigg|_s^t
	+  \int_s^ t \ep \norm{q_\delta}_{H^1(\T)}^2\,\d r 
= \sum_{i = 1}^5 I^\delta_i + M^\delta,
\end{align*}
where
\begin{align*}
	I_1^\delta & :=  \frac12 \int_s^t \int_\T u_\delta
	\, \sigma\,\pd_x\bk{\sigma\,\pd_x u_\delta} 
	\,\d x \,\d r + \frac12\int_s^t \int_\T 
	q_\delta \pd_x\bk{\sigma\,\pd_x \bk{\sigma\,q_\delta}}\,\d x  \,\d r,\\
	I_2^\delta &:= -\int_s^t \int_\T u_\delta 
	\bk{E^1_\delta + E^3_\delta}\,\d x\,\d r 
	-\int_s^t \int_\T q_\delta \pd_x E^1_\delta\,\d x \,\d r,\\
	I_3^\delta &:=  -\int_s^t \int_\T q_\delta\, \pd_x E^3_\delta \,\d x \,\d r,
	\\
	I_4^\delta &:=   \int_s^t \int_\T \left[\frac12 \abs{\pd_x E^2_\delta}^2 
	+ \pd_x\bk{\sigma q_\delta} \,\pd_x E^2_\delta \right]\,\d x \,\d r,\\
	I_5^\delta &:=  \int_s^t \int_\T \left[\frac12  \abs{E^2_\delta}^2 
	+ \sigma \,q_\delta E^2_\delta \right]\,\d x \, \d r,\\
	I_6^\delta & := \frac12 \int_s^t \int_\T \left[\abs{\sigma\,q_\delta}^2  
	+ \abs{\pd_x\bk{\sigma q_\delta}}^2\right] \,\d x\,\d r,\\
	M^\delta & := -\int_s^t \int_\T \Big[\bk{\sigma \,\pd_x u_\delta 
	+ E^2_\delta} +  \bk{ \pd_x\bk{\sigma \,q_\delta}+ \pd_x E^2_\delta }
	\Big] \,\d x \,\d W.
\end{align*}
Terms associated with the deterministic 
CH equation (where $\sigma$ does not appear) 
cancel out due to the structure of the equation 
as in the proof of Proposition \ref{thm:galerkin_H1}. 
$I_1^\delta$ to $I_3^\delta$ arise from the 
standard chain rule, and $I_4^\delta$ to $I_6^\delta$ 
are It\^o correction terms. 
$M^\delta$ is a martingale term.

As in the proof of Theorem  \ref{thm:st_uniqueness}, 
by Lemma \ref{thm:commutator1} and 
Proposition \ref{thm:commutator3},  
as $\delta \downarrow 0$,
$$
\Ex I_2^\delta , \, \Ex I_5^\delta \to 0, 
\,\qquad\,  \Ex \left[ I_3^\delta + I_4^\delta\right] \to 0.
$$

Adding $I_1^\delta$ to $I_6^\delta$, 
and performing integration-by-parts multiple times, 
\begin{align*}
	-I^\delta_1 - I^\delta_6 
	&= \frac12 \int_s^t \int_\T \pd_x \bk{\sigma \,u_\delta} 
	\,\sigma \, q_\delta \,\d  x \,\d r 
	- \frac12  \int_s^t \int_\T \abs{\sigma q_\delta}^2\,\d x \,\d r \\
	& \quad+ \frac12 \int_s^t \int_\T  \sigma\, 
	\pd_xq_\delta\,\pd_x \bk{\sigma\,q_\delta}\,\d x\,\d r 
	- \frac12  \int_s^t \int_\T \abs{\pd_x\bk{\sigma q_\delta}}^2\,\d x \,\d r \\
	& = \int_s^t \int_\T  \sigma \,\pd_x \sigma\,u_\delta \,q_\delta \,\d x \,d r 
	+ \int_s^t \int_\T \bk{\abs{\pd_x \sigma}^2
	- \frac14 \pd_x^2 \sigma^2} q_\delta^2\,\d x \,\d r
\end{align*}

We need now to take $\delta \to 0$ in 
$\Ex [I^\delta_1 + I^\delta_6]$. Since 
$u \in L^2(\Omega\times [0,T];H^2(\T))$, 
it holds that $u, q \in L^2(\Omega \times [0,T]\times \T)$, 
so by Young's inequality and the 
dominated convergene theorem,
$$
-\Ex  \left[I^\delta_1
+I^\delta_6\right] 
\overset{\delta \downarrow 0}{\longrightarrow} 
\Ex \int_s^t \int_\T  \sigma \,\pd_x \sigma\,u \,q \,\d x \,d r 
+ \Ex \int_s^t \int_\T \bk{\abs{\pd_x \sigma}^2 
- \frac14 \pd_x^2 \sigma^2} q^2\,\d x \,\d r.
$$

Finally, $\Ex M^\delta = 0$, since its 
quadratic variation satisfies
$$
\Ex \int_0^T \abs{\int_\T \Big[\bk{\sigma 
\,\pd_x u_\delta + E^2_\delta} 
+\bk{ \pd_x\bk{\sigma \,q_\delta}
+\pd_x E^2_\delta }\Big] \,\d x}^2\,\d r < \infty.
$$

On the left-hand side, we can also 
pass $\delta \to 0$ at every $t \in [0,T]$ using
$$
\Ex \lim_{\delta \to 0}\norm{u_\delta(t)}_{H^1(\T)}^2 
= \lim_{\delta \to 0} \Ex \norm{u_\delta(t)}_{H^1(\T)}^2,
$$
(by the Lebesgue dominated convergence theorem) 
and the energy bound \eqref{eq:galerkin_unifbd1}. 
The passage in $\delta \to 0$ for the 
temporal integral $\ep \Ex \int_{s}^t
 \norm{q_\delta(r)}_{H^1(\T)}^2\,\d r$ is similar. 
We therefore arrive at \eqref{eq:energyeq_exp}.
\end{proof}

\begin{prop}\label{thm:st_time_cont}
Let $u$ be the solution of Theorem \ref{thm:existence_H1}. 
For any $t_0 \in (0,T)$, 
$$
\lim_{t\to t_0} {\Ex} \norm{{u}(t)-u(t_0)}_{H^1(\T)}^2 = 0,
$$
with corresponding one-sided limits $t \downarrow 0$ and $t \uparrow T$ 
at the end-points $t_0 = 0$ and $t_0 = T$, respectively.
Moreover, for a $p_0 > 4$, $u \in L^{p_0}(\Omega;C([0,T];H^1(\T)))$.
\end{prop}

\begin{proof}
We can upgrade the weak $H^1_x$-continuity into strong 
continuity using the Brezis--Lieb lemma (for $L^2$). 
Since we already know that $u\in C([0,T];L^2(\T))$ 
(see Lemma \ref{thm:u_Ctheta_L2} and argue as in 
Part 6 of the proof to Theorem \ref{thm:existence_H1}), 
we need only establish temporal 
continuity for $q=\pd_x {u}$ in $L^2(\T)$.

From ${u} \in C([0,T];H^1_w(\T))$, 
${q}(t) \rightharpoonup {q}(s)$ in $L^2(\T)$ as $t \to s$. 
Since ${\Ex}\norm{{q}}_{H^1(\T)}^2 \in L^1([0,T])$ 
(Theorem \ref{thm:existence_H1}), the map
$t \mapsto  \int_{0}^t{\Ex} \norm{{q}(r)}_{H^1(\T)}^2\,\d r$ 
is absolutely continuous on $[0,T]$, 
and from the energy equality of 
Lemma \ref{thm:energy_eq},  we obtain
\begin{align}\label{eq:limit_of_norms}
{\Ex} \norm{{u}(t)}_{H^1(\T)}^2 
	\overset{t \to t_0}{\longrightarrow} 
		{\Ex} \norm{{u}(t_0)}_{H^1(\T)}^2
		\qquad \mbox{for a.e. } t_0 \in [0,T].
\end{align}

Finally,
\begin{align*}
{\Ex} \norm{{u}(t) - {u}(t_0)}_{H^1(\T)}^2 
= {\Ex} \norm{{u}(t)}_{H^1(\T)}^2 
		- 2 {\Ex} \LL {u}(t), {u}(t_0)\RR_{H^1(\T)} 
		+ {\Ex} \norm{{u}(t_0)}_{H^1(\T)}^2.
\end{align*}
By weak continuity, the middle term tends 
to $- 2 {\Ex}\norm{{u}(t_0)}_{H^1(\T)}^2$; 
together with \eqref{eq:limit_of_norms}, 
we attain the lemma statement.
\bigskip

Since $ \lim_{t \to t_0}{\Ex} \norm{{u}(t) - {u}(t_0)}_{H^1(\T)}^2 = 0$, 
by Fatou's lemma, we also have
$$
{\Ex}  \lim_{t \to t_0}\norm{{u}(t) - {u}(t_0)}_{H^1(\T)}^2 = 0,
$$
and therefore $\lim_{t \to t_0}\norm{{u}(t) - {u}(t_0)}_{H^1(\T)}^2$, 
${\mathbb{P}}$-almost surely.

Since $u \in C([0,T];H^1(\T))$, $\mathbb{P}$-almost 
surely, and for initial conditions 
$u_0 \in L^{p_0}(\Omega; H^1(\T))$, 
$u \in L^{p_0}(\Omega;L^\infty([0,T];H^1(\T)))$, 
we can readily conclude that $u \in L^{p}(\Omega;C([0,T];H^1(\T)))$.
This follows from  the fact that the 
$C([0,T];H^1(\T))$ norm coincides with the 
$L^\infty([0,T];H^1(\T))$ for any element in $C([0,T];H^1(\T))$.
\end{proof}

\section{Higher regularity solutions 
(Theorem \ref{thm:main2})}\label{sec:higherreg}

In this section we fix $m \ge 2$ and 
consider the well-posedness of strong $H^m$ solutions.
We will emphasise the parts that differ 
from the well-posedness theory for strong $H^1$ solutions.
Throughout this section we will require that the noise function 
$\sigma$ belongs to $W^{m + 1,\infty}(\T)$.

\subsection{Weak existence}\label{sec:higherreg_existence} 

We begin by proving the existence of weak 
(martingale) $H^m$ solutions. Cubic nonlinearities in the 
SDE for $\d \norm{u_n(t)}_{H^1(\T)}^2$, which 
disappear due to the structure of the equation 
if $m=1$, are retained at the level of $H^m(\T)$. 
Therefore standard calculations, involving first taking 
the expectation and then applying a standard 
Gronwall inequality, oblige us to use 
Gagliardo--Nirenberg inequalities. 
These fail to give sufficiently controllable 
powers on certain norms due to the extra 
expectation integral under which these norms 
are bounded. In particular, there are no uniform bounds 
in $L^\infty(\Omega)$ for any stochastic quantity. 
As in Theorem \ref{thm:st_uniqueness}, we introduce 
stopping times $\eta^n_R<T$ to control exponential moments, 
so that the estimates derived below hold only 
on $\left[0,\eta^n_R\right]$, where
the stopping time $\eta^n_R$ depend on $n$ and
an ``auxiliary parameter" $R$.   
These stopping times converge a.s.~to the final time 
$T$ as $R\to \infty$. Finally, we use the obtained 
estimate to conclude uniform-in-$n$ stochastic 
boundedness in $L^2([0,T];H^m(\T))$, which is precisely 
the tightness condition required to apply the 
Skorokhod--Jakubowski procedure. 

We make here the technical observation that 
the projection operator $\bm{\Pi}_n$ acting 
on $f \in H^{m}(\T)$ also satisfies
\begin{equation}\label{eq:proj_Hm_conv}
	\norm{\bm{\Pi}_nf-f}_{H^{m}(\T)}\ton 0,
\end{equation}
because the basis in $H^1(\T)$ of trigonometric functions 
of integral frequencies forms a basis in $H^m(\T)$ as well.

\begin{prop}[$H^m$ and $H^{m+1}$ estimates up to a stopping time]
\label{thm:galerkinHm}
Let $u_n$ be a solution to \eqref{eq:galerkin_n} 
with $\sigma \in W^{m + 1, \infty}(\T)$ 
and initial condition ${u_0} \in L^{2p}(\Omega; H^{m}(\T))$, 
for some $p \in [1,\infty)$. For $R>1$, let 
$\eta_R^n$ be the stopping time
\begin{align}\label{eq:stoptime2}
	\eta_R^n := \inf \left\{t  \in [0,T] : 
	\int_0^t \norm{u_n(s)}_{W^{1,\infty}(\T)}^2
	\,\d s > R\right\},
\end{align}
setting $\eta_R^n = T$ if the set on the 
right-hand side is empty.  Then $\eta_R^n \toR T$, 
$\mathbb{P}$-a.s., uniformly in $n$. Moreover, there 
exists a constant
$$
C=C\left(p,T,R,\ep, \Ex \norm{u_0}_{H^m(\T)}^{2p}, 
\norm{\sigma}_{W^{m+1,\infty}(\T)}\right),
$$
independent of $n \in \mathbb{N}$, such that
\begin{equation}\label{eq:galerkin_Hm}
	\Ex \norm{u_n}_{L^\infty([0,\eta_R^n];H^m(\T))}^p 
	\le C.
\end{equation}
Finally, there is a constant 
$C=C\left(T,R,\ep,\Ex \norm{u_0}_{H^m(\T)}^2, 
\norm{\sigma}_{W^{m+1,\infty}(\T)}\right)$, which is
independent of $n$, such that 
\begin{equation}\label{eq:diss-term-new}
	\Ex\left(\int_0^{\eta_R^n} 
	\norm{u_n(t)}_{H^{m+1}(\T)}^2
	\,\d t\right)^{1/2}\le C.
\end{equation}

\end{prop}

\begin{rem}\label{rem:L8_omega}
Subsequently, we will take $2p = 2$ to show existence, 
but require $2p = 8$ for uniqueness, to 
use Lemmas \ref{thm:commutator1}, \ref{thm:commutator2}, 
and Proposition \ref{thm:commutator3}. 
From Lemma \ref{thm:commutator1}, to establish
uniqueness, we require $4$th moments on $\norm{u}_{L^\infty([0,T];H^1(\T))}$,
which in turn is bounded by the $4$th moment of the initial
condition. Application of the stochastic Gronwall inequality requires that
a strictly higher moment be boundeed. It is possible simply to take 
$2p  > 4$, but it is more convenient simply to take $2p = 8$.
\end{rem}

\begin{proof}  We divide the proof into several steps.

{\it 1. Pointwise limit of the stopping times.}

The fact that $\eta^n_R$ as defined in \eqref{eq:stoptime2} 
converges a.s.~ to $T$ as $R \to \infty$, uniformly in $n$, 
follows from a calculation similar to that 
in \eqref{eq:stoptime_2_bd}. More precisely, 
given the $n$-uniform (but $\ep$-dependent) 
$H^2$ estimate implied by \eqref{eq:galerkin_unifbd1} 
and the embedding $H^2(\T) \hookrightarrow 
W^{1,\infty}(\T)$, we have
\begin{align}
	\prob\bk{\seq{\eta_R^n<T}}
	&\le \prob\bk{\seq{\int_0^T
	\norm{u_n(t)}_{W^{1,\infty}(\T)}^2\,\d t>R}} 
	\label{eq:stoptime_2_bd-new} 
	\\ & 
	\le \frac1R \Ex \int_0^T 
	\norm{u_n(t)}_{W^{1,\infty}(\T)}^2\,\d t
	\le \frac{\tilde{C}}{R}
	\Ex \int_0^T\norm{u_n(t)}_{H^2(\T)}^2\,\d t \notag\\
	&\le \frac{C_\ep}{R}\toR 0.
	\notag
\end{align}

{\it 2. Bounds on higher regularity norms.}

Next we prove the bound \eqref{eq:galerkin_Hm}. 
Taking the $\ell$th derivative of \eqref{eq:galerkin_n}, 
we find that 
\begin{align*}
	0 & = \d \pd_x^\ell u_n - \ep \pd_x^{\ell+2} u_n \,\d t
	+\left[ \pd_x^\ell \bm{\Pi}_n \bk{u_n \,\pd_x u_n} 
	+ \pd_x^{\ell + 1} P[u_n] \right]\,\d t\\
	&\quad - \frac12 \pd_x^\ell\bm{\Pi}_n\bk{\sigma 
	\,\pd_x\bk{\sigma\,\pd_x u_n}} \,\d t
	+ \pd_x^\ell \bm{\Pi}_n \bk{\sigma\,\pd_x u_n}\,\d W.
\end{align*}
First we multiply through by $\pd_x^\ell u_n$, 
integrate in space, and use the commutativity 
between the projection and the derivative. 
Then we apply the It\^o formula 
for $r \mapsto r^p$. The result is
\begin{align*}
	\frac1{2p}&\norm{\pd_x^\ell u_n(s)}_{L^2(\T)}^{2p}
	\bigg|_0^t + \ep \int_0^t 
	\norm{\pd_x^\ell u_n(s)}_{L^2(\T)}^{2p-2}
	\norm{\pd_x^{\ell + 1} u_n(s)}_{L^2(\T)}^2\,\d s
	\\ & =\frac12  \int_0^t
	\norm{\pd_x^\ell u_n(s)}_{L^2(\T)}^{2p-2}
	\int_\T\pd_x u_n \bk{\pd_x^{\ell} u_n}^2 \,\d x\,\d s 
	\\ &\quad -\int_0^t
	\norm{\pd_x^\ell u_n(s)}_{L^2(\T)}^{2p - 2}
	\int_\T \pd_x^{\ell + 1} P [u_n] \,\pd_x^{\ell} 
	u_n \,\d x\,\d s
	\\ & \quad
	+ \int_0^t\norm{\pd_x^\ell u_n(s)}_{L^2(\T)}^{2p-2} 
	\int_\T \bk{u_n\pd_x^{\ell + 1} u_n -\pd_x^\ell
	\bk{u_n\,\pd_x u_n}}\pd_x^{\ell} u  \,\d x\,\d s
	\\ & \quad + \frac12 \int_0^t 
	\norm{\pd_x^\ell u_n(s)}_{L^2(\T)}^{2p - 2}
	\int_\T \pd_x^{\ell} u_n  \, 
	\pd_x^\ell \bk{\sigma \,\pd_x
	\bk{\sigma\,\pd_x u_n}} \,\d x \,\d s 
	\\ &\quad+ \frac12 \int_0^t 
	\norm{\pd_x^\ell u_n(s)}_{L^2(\T)}^{2p - 4} 
	\abs{\int_\T \pd_x^{\ell}u_n 
	\,\pd_x^\ell \bk{\sigma\,\pd_x u_n}
	\,\d x}^2\,\d s\\ &\quad
	+ \int_0^t \norm{\pd_x^\ell u_n(s)}_{L^2(\T)}^{2p-2}
	\int_\T \pd_x^{\ell}u_n \,\pd_x^\ell 
	\bk{\sigma\,\pd_x u_n}\,\d x\,\d W
	\\ & = : \sum_{i=1}^5\int_0^t I_i^n\,\d s
	+ \int_0^t I_6^n \,\d W.
\end{align*}

We again estimate  $I_1^n$ to $I_5^n$, leaving 
the martingale term $\int_0^t I_6^n\,\d W$ 
to be handled by the stochastic Gronwall inequality. 
We have readily that
\begin{align*}
	\abs{I_1^n} \le \frac12 
	\norm{\pd_x u_n}_{L^\infty(\T)} 
	\norm{\pd_x^{\ell} u_n}_{L^2(\T)}^{2p}.
\end{align*}
By the Cauchy--Schwarz and Young's inequalites,
\begin{align*}
	\abs{I_2^n}& \le C_\ep
	\norm{\pd_x^{\ell} u_n}_{L^2(\T)}^{2p-2}
	\norm{\pd_x^\ell P[u_n]}_{L^2(\T)}^2 
	+\frac\ep2 \norm{\pd_x^{\ell} u_n}_{L^2(\T)}^{2p -2}
	\norm{\pd_x^{\ell+1} u_n}_{L^2(\T)}^2.
\end{align*}

By the Leibniz rule and the 
Gagliardo--Nirenberg inequality, 
\begin{align}
	\norm{\pd_x^\ell u_n^2}_{L^2(\T)}
	& =\norm{u_n\pd_x^\ell u_n
	+\ell \pd_x u_n \,\pd_x^{\ell-1} u_n 
	+\ldots+\pd_x^\ell u_n\,u_n}_{L^2(\T)}\notag
	\\ & \lesssim 
	\norm{u_n}_{L^\infty(\T)} 
	\norm{\pd_x^\ell u_n}_{L^2(\T)},\label{eq:LGN}
	\\ \intertext{and likewise}
	\norm{\pd_x^{\ell-1}\bk{\pd_x u_n}^2}_{L^2(\T)} 
	& \lesssim \norm{\pd_x u_n}_{L^\infty(\T)}
	\norm{\pd_x^\ell u_n}_{L^2(\T)}. \notag
\end{align}
\begin{rem}
We add some details on the estimate \eqref{eq:LGN}. It suffices to show
\begin{equation*}
\norm{u^{(j)}u^{(\ell-j)}}_{L^2(\T)}\lesssim  \norm{u}_{L^\infty(\T)}\norm{u^{(\ell)}}_{L^2(\T)}, \quad j=0,\dots,\ell,
\end{equation*}
with $u^{(j)}=\pd_x^j u$. H\"older's inequality gives
\begin{equation*}
\norm{u^{(j)}u^{(\ell-j)}}_{2}\le  \norm{u^{(j)}}_{r}\norm{u^{(\ell-j)}}_{r'}, \quad \frac1r+\frac1{r'}=\frac12.
\end{equation*}
Apply now the 
Gagliardo--Nirenberg inequality \cite[Lemma 2.1]{Boritchev:2014} to find
\begin{equation*}
\norm{u^{(\beta)}}_{r}\lesssim  \norm{u^{(m)}}_{p}^\theta\norm{u}_{q}^{1-\theta}, 
\end{equation*}
assuming $\int_\T u\, \d x=0$. Here
\begin{equation*}
\frac1r=\beta-\theta(m-\frac1p)+(1-\theta)\frac1q, \quad m>\beta.
\end{equation*}
If $p$ equals $1$ or $\infty$, then $\theta=\beta/m$. Assume for the moment that $\int_\T u\,\d x=0$. We get
\begin{equation*}
\norm{u^{(j)}}_{r}\lesssim  \norm{u^{(\ell)}}_{2}^{j/\ell}\norm{u}_{\infty}^{1-j/\ell}, 
\end{equation*}
with
\begin{equation*}
\frac1r=j-\frac{j}{\ell}(\ell-\frac12).
\end{equation*}
Similarly,
\begin{equation*}
\norm{u^{(\ell-j)}}_{r'}\lesssim  \norm{u^{(\ell)}}_{2}^{\ell-j/\ell}\norm{u}_{\infty}^{j/\ell}, 
\end{equation*}
with
\begin{equation*}
\frac1{r'}=\ell-j-\frac{\ell-j}{\ell}(\ell-\frac12).
\end{equation*}
Note that $1/r+1/r'=1/2$.  Furthermore,
\begin{align*}
\norm{u^{(j)}}_{r}\norm{u^{(\ell-j)}}_{r'}&\lesssim \norm{u^{(\ell)}}_{2}^{j/\ell}\norm{u}_{\infty}^{1-j/\ell} \norm{u^{(\ell)}}_{2}^{\ell-j/\ell}\norm{u}_{\infty}^{j/\ell} \\
&=\norm{u^{(\ell)}}_{2}\norm{u}_{\infty}. 
\end{align*}
In the general case, let $v=u-\abs{\T}^{-1}\int_\T u\, \d x$. Then we find
\begin{align*}
\norm{u^{(j)}}_{r}=\norm{v^{(j)}}_{r} &\lesssim  \norm{v^{(\ell)}}_{2}^{j/\ell}\norm{v}_{\infty}^{1-j/\ell} \\
&\lesssim \norm{u^{(\ell)}}_{2}^{j/\ell}\norm{u}_{\infty}^{1-j/\ell}, 
\end{align*}
and similar for the other estimate.  This justifies \eqref{eq:LGN}.
\end{rem}
\smallskip


Writing the term $\pd_x^\ell P$ above as 
$$
\pd_x^\ell K *\bk{u_n^2+\frac12 \bk{\pd_x u_n}^2} 
= K* \pd_x^\ell u_n^2 
+ \frac12 \pd_x K * \pd_x^{\ell-1}\bk{\pd_x u_n}^2,
$$
the $L^2_x$ norm can be bounded by the 
Young convolution and Gagliardo--Nirenberg 
inequalities, see \cite[(2.6), (2.8)]{XZ2000}:
\begin{align*}
	\norm{\pd_x^\ell P[u_n]}_{L^2(\T)} 
	& \lesssim \bk{\norm{\pd_x^\ell u_n^2}_{L^2(\T)} 
	+ \norm{\pd_x^{\ell-1}
	\bk{\pd_x u_n}^2}_{L^2(\T)}}
	\\ & \lesssim 
	\bk{\norm{u_n}_{L^\infty(\T)} 
	+ \norm{\pd_x u_n}_{L^\infty}} 
	\norm{\pd_x^\ell u_n}_{L^2(\T)}.
\end{align*}
Because $u_n\pd_x^{\ell + 1} u_n
-\pd_x^\ell \bk{u_n\,\pd_x u_n}$ precisely removes 
all instances of the $(\ell+1)$st derivative, 
by the Gagliardo--Nirenberg inequality, 
as in \cite[(2.7)]{XZ2000},
\begin{align*}
	\abs{I_3^n} & \le  
	\norm{\pd_x^{\ell} u_n}_{L^2(\T)}^{2p -2}
	\norm{u_n\pd_x^{\ell + 1} u_n
	-\pd_x^\ell \bk{u_n\,\pd_x u_n}}_{L^2(\T)}
	\norm{\pd_x^{\ell} u_n}_{L^2(\T)}
	\\ & \le C \norm{\pd_x u_n}_{L^\infty(\T)}
	\norm{\pd_x^{\ell} u_n}_{L^2(\T)}^{2p}.
\end{align*}

In the parts involving $\sigma$, for $I_4^n$ 
we have
\begin{align*}
	& 2I_4^n \norm{\pd_x^{\ell} u_n}_{L^2(\T)}^{-2p + 2}
	\\ &\quad 
	= -\int_\T \pd_x^{\ell+1} u_n  
	\,\pd_x^{\ell-1} \bk{\sigma \,\pd_x 
	\bk{\sigma\, \pd_x u_n}}\,\d x
	\\ & \quad = -\int_\T \pd_x^{\ell + 1} u_n  
	\,\pd_x^{\ell-1} \bk{\sigma^2  \pd_x^2 u_n}
	+\pd_x^{\ell+1} u_n  \,\pd_x^{\ell-1} 
	\bk{\sigma \,\pd_x \sigma\,  \pd_x u_n}\,\d x
	\\ & \quad = -\int_\T \sigma^2 
	\abs{\pd_x^{\ell + 1}u_n}^2\,\d x
	-\int_\T \pd_x^{\ell+1} u_n  
	\sum_{k = 0}^{\ell-2} \binom{\ell -1}{k}
	\pd_x^{\ell-1-k}\sigma^2  
	\pd_x^{2+k} u_n\,\d x 
	\\ &\quad\quad
	-\int_\T \sigma\,\pd_x\sigma\,\pd_x^{\ell+1} 
	u_n \pd_x^\ell u_n + \pd_x^{\ell + 1} 
	u_n \sum_{k=0}^{\ell - 2} \binom{\ell - 1}{k}
	\pd_x^{\ell-1-k}\bk{\sigma \,\pd_x \sigma}
	\,\pd_x^{k+1} u_n\,\d x
	\\ & \quad = -\int_\T \sigma^2 
	\abs{\pd_x^{\ell+1}u_n}^2\,\d x  
	+\frac12 \int_\T \pd_x \bk{\sigma \pd_x \sigma}
	\,\abs{\pd_x^\ell u_n}^2\,\d x
	\\ & \quad \quad - \int_\T \pd_x^{\ell+1} 
	u_n  \sum_{k=0}^{\ell-2} 
	\binom{\ell-1}{k}\pd_x^{\ell-1-k}
	\sigma^2  \pd_x^{2 + k} u_n\,\d x 
	\\ & \quad \quad - \int_\T \pd_x^{\ell + 1} 
	u_n \sum_{k=0}^{\ell - 2} \binom{\ell - 1}{k} 
	\pd_x^{\ell-1-k}\bk{\sigma \,\pd_x \sigma}
	\,\pd_x^{k+1} u_n\,\d x,
\end{align*}
which implies
\begin{equation*}
2\abs{I_4^n} \norm{\pd_x^{\ell} u_n}_{L^2(\T)}^{-2p + 2}\le C_{\sigma, \ell}
	\norm{\pd_x^\ell u_n}_{L^2(\T)}^2,
\end{equation*}
because the summands do not have derivatives 
of order higher than $\ell$.

Similarly for $I_5^n$ we have
\begin{align*}
	&\int_\T \pd_x^{\ell+1} u_n 
	\,\pd_x^{\ell - 1} \bk{\sigma \pd_x u_n}\,\d x
	\\ & \qquad =  \int_\T \big[\sigma\, \pd_x^{\ell+1} 
	u_n\, \pd_x^{\ell } u_n+\pd_x^{\ell} u_n \pd_x 
	\sum_{k = 0}^{\ell -2} \binom{\ell-1}{k}
	\pd_x^{\ell-1-k}\sigma\,\pd_x^{k + 1} u_n\big]\,\d x
	\\ & \qquad = - \int_\T \big[\frac12\pd_x 
	\sigma \abs{\pd_x^{\ell} u_n}^2
	-\pd_x^{\ell} u_n \pd_x \sum_{k = 0}^{\ell-2}
	\binom{\ell-1}{k} \pd_x^{\ell-1-k}
	\sigma\,\pd_x^{k+1} u_n\big]\,\d x,
\end{align*}
and therefore
\begin{align*}
	\abs{I_5^n}\le C_{\sigma, \ell} 
	\norm{\pd_x^{\ell} u_n}_{L^2(\T)}^{2p-4} 
	\bk{\norm{\pd_x^{\ell} u_n}_{L^2(\T)}^2}^2 
	\le C_{\sigma, \ell}  
	\norm{\pd_x^{\ell} u_n}_{L^2(\T)}^{2p}.
\end{align*}

Gathering the estimates for $I_1^n$ to $I_5^n$, we find
\begin{align*}
	&\frac1{2p}\d
	\norm{\pd_x^\ell u_n}_{L^2(\T)}^{2p} 
	+\frac\ep2\norm{\pd_x^\ell u_n}_{L^2(\T)}^{2p- 2}
	\norm{\pd_x^{\ell+1} u_n}_{L^2(\T)}^2\,\d t
	\\ & \quad 
	\le  C_{\sigma, \ell, \ep} 
	\bk{1+\norm{u_n}_{L^\infty(\T)}
	+\norm{\pd_x u_n}_{L^\infty(\T)}}
	\norm{\pd_x^\ell  u_n}_{L^2(\T)}^{2p}
	\,\d t  + I_6^n \,\d W,
\end{align*}
where $\int_0^{t \wedge \eta_R^n} I_6^n \,\d W$ is a square-integrable martingale. 
We can overestimate the right-hand side by adding 
to ``$\norm{\pd_x^\ell u_n}_{L^2(\T)}^{2p}$" 
the term ``$\int_0^t \frac{\ep}{2}  \ldots \, \d s$". 
Setting 
\begin{align*}
	& \xi_n(t):= \frac1{2p}
	\norm{\pd_x^\ell u_n(t)}_{L^2(\T)}^{2p} 
	+\frac\ep2\int_0^t
	\norm{\pd_x^\ell u_n(s)}_{L^2(\T)}^{2p- 2}
	\norm{\pd_x^{\ell+1} u_n(s)}_{L^2(\T)}^2\, \d s,
	\\ & 
	A_n(t) := \int_0^t C_{\sigma, \ell, \ep}
	\bk{1+\norm{u_n(s)}_{L^\infty(\T)}
	+\norm{\pd_x u_n(s)}_{L^\infty(\T)}}\, \d s,
\end{align*}
and $M_n(t) := \int_0^t I_6^n \,\d W(s)$, we obtain
$\d\xi_n(t) \le \xi_n(t)\, \d A_n(t)+\d M_n(t)$.
Now an application of the stochastic Gronwall inequality 
(Lemma \ref{thm:stochastic_gronwall_st}) gives
\begin{equation}\label{eq:parabolic_2_m}
	\begin{aligned}
		&\Biggl(\Ex \sup_{s \in \left[0,\eta_R^n\right]}
		\, \biggl ( \, 
		\norm{\pd_x^\ell u_n(s)}_{L^2(\T)}^{2p} 
		\\ & \qquad \qquad
		+\frac{\ep}{2}\int_0^s 
		\norm{\pd_x^\ell u_n(s')}_{L^2(\T)}^{2p-2}
		\norm{\pd_x^{\ell+1}u_n(s')}_{L^2(\T)}^2
		\,\d s'\,\biggr )^{1/2} \Biggr)^2
		\\ & \qquad \qquad\qquad\qquad
		\le C_{p,\sigma,\ep,T,\ell,R}\,
		\Ex \norm{\pd_x^\ell u_n(0)}_{L^2(\T)}^{2p},
		\qquad \ell =0,\ldots,m,
	\end{aligned}
\end{equation}
from which \eqref{eq:galerkin_Hm} easily follows. 

With $p=1$ and $\ell=1,\ldots,m$, it also follows 
from \eqref{eq:parabolic_2_m} that
\begin{align*}
	\Ex \, \biggl ( \, 
	\frac{\ep}{2}\int_0^{\eta_R^n} 
	\norm{\pd_x^{\ell+1}u_n(s')}_{L^2(\T)}^2
	\,\d s'\,\biggr )^{1/2}
	\le C_{\sigma,\ep,T,\ell,R,u_0},
\end{align*}
which implies \eqref{eq:diss-term-new}.
\end{proof}

Next, we establish stochastic boundedness and 
tightness of laws for $\seq{u_n}$. The proof differs 
from the straightforward deduction leading to 
Lemma \ref{thm:stochastic_boundedness} ($m=1$), and 
$u_n\in_{\rm sb} L^2([0,T]; H^{2}(\T)) \cap 
W^{\theta', 2}([0, T]; L^2(\T))$ and the tightness 
of laws in $L^2([0,T];H^{1}(\T))$, where $u_n\in_{\rm sb}
L^2_tH^{2}_x$ follows trivially from 
$u_n\in_b L^2(\Omega;L^2_tH^2_x)$. For $m \ge 2$, we 
do not have $u_n\in_b L^2\left(\Omega;L^2_tH^{m+1}_x\right)$ 
for the entire interval $[0,T]$, but rather only 
up to a suitable stopping time, cf.~\eqref{eq:diss-term-new}. 
The next proof develops a refined stopping time argument 
to deal with this issue, leading to 
$u_n\in_{\rm sb} L^2_tH^{m+1}_x$. 

\begin{lem}\label{thm:stoch_bound_Hm_v2}
Let $u_{n}$ be a solution to \eqref{eq:galerkin_n} 
with $\Ex \norm{u_0}_{H^1(\T)}^p,\, 
\Ex \norm{u_0}_{H^{m}(\T)}^2 < \infty$, for some $p > 2$. 
For $\theta' < (2 - p)/4p$, the laws of $\{u_n\}_{n \in\N}$ are uniformly stochastically bounded in 
$L^2([0,T];H^{m+1}(\T))
	\cap W^{\theta',2}([0,T];L^2(\T))$, i.e.,
\begin{equation}\label{eq:stochastic-bound}
	\lim_{M\to\infty} 
	\prob\bk{\seq{\norm{u_n}_{L^2([0,T];H^{m+1}(\T))
	\cap W^{\theta',2}([0,T];L^2(\T))}>M}}=0
\end{equation}
holds, uniformly in $n$. 
The laws of $\{u_n\}_{n \in\N}$ are also uniformly stochastically bounded in 
$L^\infty([0,T];H^{m}(\T))$.

Moreover, the laws of $\seq{u_n}$ 
are tight on $L^2([0,T];H^{m}(\T))$.
\end{lem}

\begin{proof}
As in the proof of Lemma \ref{thm:stochastic_boundedness}, for any $\theta \in (\theta', (2 - p)/4p)$,
$$
\prob\bk{
\Bigl\{\norm{u_n}_{W^{\theta', r}([0, T]; L^2(\T))}>M
\Bigr\}} \le \frac1M 
\,\Ex \norm{u_n}_{C^\theta([0,T];L^2(\T))} 
\lesssim \frac1M,
$$
where, in passing, we mention that the requirement 
$u_0\in L^p_{\omega}H^1_x$ is linked to the application
of Lemma \ref{thm:u_Ctheta_L2}, which allows us to
 arrive at the final $1/M$ estimate. 

Following the proof of Lemma 
\ref{thm:stochastic_boundedness}, set
$$
X_n(t):=\bk{\int_0^t \norm{u_n(s)}_{H^{m+1}(\T)}^2
\,\d s}^{1/2}=\norm{u_n}_{L^2([0,t];H^{m+1}(\T))},
$$
and introduce the stopping time
$$
\xi_M^n =\inf\seq{t\in [0,T]: X_n(t)>M},
$$
setting $\xi_M^n=T$ if the set is empty. 
Let $\eta_R^n$ be the stopping time defined 
in \eqref{eq:stoptime2}. For any fixed $R$, we have
\begin{align*}
	\seq{X_n(t)> M} 
	=\seq{\xi_M^n < t} 
	& = \seq{\xi_M^n<t,\eta_{R}^n<t} 
	\cup \seq{\xi_M^n<t,\eta_{R}^n\ge t}
	\\ & \subseteq \seq{\eta_{R}^n<t}
	\cup \seq{\xi_M^n<t,\eta_{R}^n\ge t}.
\end{align*}
We can estimate the probability of the last 
event on the right-hand side separately 
as follows: 
\begin{align*}
	\seq{\xi_M^n<t,\eta_R^n\ge t} 
	&\subseteq 
	\seq{\xi_M^n<\bk{t\wedge \eta_R^n}} 
	\subseteq \seq{X_n(t \wedge \xi_M^n\wedge \eta_R^n) \ge M},
\end{align*}
which implies
\begin{align*}
	 \prob
	\bigl(\seq{\xi_M^n<t,\eta_R^n \ge t}\bigr) 
	&\le \prob
	\bigl(\seq{X_n(t\wedge\xi_M^n\wedge \eta_R^n) \ge M}\bigr)
	\\ & 
	\le \frac1M \Ex\, 
	X_n(t\wedge \xi_M^n \wedge \eta_R^n) 
	\le \frac1M \Ex\, X_n(t \wedge \eta_R^n) 
	\overset{\eqref{eq:diss-term-new}}{\le}
	\frac{C_{T,R,\ep}}{M}.
\end{align*}
Separately, via \eqref{eq:stoptime_2_bd-new},
$$
\mathbb{P}\bigl(\seq{\eta_{R}^n < t}\bigr)
\le\frac{C_{T,\ep}}R.
$$
Therefore,
\begin{align*}
	\prob\bigl(\seq{X_n(t)>M}\bigr) 
	\le \frac{C_{T,R,\ep}}{M}
	+\frac{C_{T,\ep}}{R}.
\end{align*}
Sending $M \to \infty$, recalling the 
definition of $X_n$, we arrive at
\begin{align*}
	\lim_{M\to\infty}
	\prob
	\bk{\seq{\left(\int_0^t \norm{u_n(s)}_{H^{m+1}(\T)}^2
	\,\d s \right)^{1/2}>M}}
	\le \frac{C_{T,\ep}}{R},
\end{align*}
which can be made arbitrarily small by taking $R$ 
large, uniformly in $t\in [0,T]$. 
This implies \eqref{eq:stochastic-bound}. 
Tightness on $L^2([0,T];H^m(\T))$ follows
from this, Lemma \ref{thm:temam_flandoli_gatarek}, 
and the $n$-uniformity of the limit $M\to \infty$, 
arguing as in  in the proof of 
Lemma \ref{thm:stochastic_boundedness}.

The same argument yields stochastic boundedness in 
$L^\infty([0,T];H^m(\T))$ for the laws of $\{u_n\}_{n \in\N}$.
\end{proof}

Introducing the path spaces:
\begin{align*}
	& \mathcal{X}_{u,s}^m := L^2([0,T];H^m(\T)), 
	\quad \mathcal{X}_{u,w}^m :=  C([0,T];H^1_w(\T)), 
	\\ 
	& \mathcal{X}_W := C([0,T]), 
	\quad
	\mathcal{X}_0 := H^{m}(\T),
\end{align*}
and setting $\mathcal{X}_m :=\mathcal{X}_{u,s}^m \times \mathcal{X}_{u,w}^m 
\times \mathcal{X}_W\times \mathcal{X}_0$, we 
repeat the procedure in Section \ref{sec:existence1}.

\begin{lem}\label{eq:JS_2}
The joint laws of $\bigl(u_n, u_n,W,\bm{\Pi}_n u_0\bigr)$ 
are tight on $\mathcal{X}_m$.
\end{lem}

\begin{proof}
By Proposition \ref{thm:stoch_bound_Hm_v2} and 
Lemma \ref{rem:Ctheta_tightness}, the laws of $u_n$ are tight 
on $\mathcal{X}_{u,s}^m$ and $\mathcal{X}_{u,w}^m$. 
Since $\bm{\Pi}_n u_0 \to u_0$ 
in $H^m(\T)$, cf.~\eqref{eq:proj_Hm_conv}, the laws of 
$\bm{\Pi}_n u_0$ are tight on $H^m(\T)$. As $n \to \infty$, 
the law of $W$ is stationary on $\mathcal{X}_W$ and therefore tight. 
\end{proof}

\begin{thm}[Weak $H^m$ solution]
Suppose $\sigma \in W^{m+1,\infty}(\T)$ and 
that $u_0$ belongs to $L^p(\Omega;H^{1}(\T))
\cap L^2(\Omega;H^{m}(\T))$, for $p\in [1,\infty)$. 
There exists a weak $H^m$ solution 
$\bigl((\tilde{\Omega},\tilde{\mathcal{F}},
\bigl\{\tilde{\mathcal{F}}_t\bigr\}_{t \ge 0},
\tilde{\mathbb{P}}),\tilde{u},\tilde{W}\bigr)$ 
to the viscous stochastic CH equation \eqref{eq:u_ch_ep} 
with initial condition $u|_{t=0}=u_0$.

\end{thm}

\begin{proof}
From the Skorokhod--Jakubowski theorem 
(Theorem \ref{thm:skorohod}), we can extract variables 
$(\tilde{u}_{n,s}, \tilde{u}_{n,w}, \tilde{W}_n, \tilde{u}_{0,n})$ 
and $(\tilde{u}_s, \tilde{u}_w,\tilde{W},\tilde{u}_{0})$ 
on a probability space $(\tilde{\Omega},
\tilde{\mathcal{F}},\tilde{\mathbb{P}})$ such that
\begin{align*}
	&(\tilde{u}_{n,s}, \tilde{u}_{n,w},\tilde{W}_n,\tilde{u}_{0,n}) 
	\sim (u_n, u_n, W,\bm{\Pi}_n u_0)
	\quad \text{in $\mathcal{X}_m$},
	\\
	& (\tilde{u}_{n,s}, \tilde{u}_{n,w}, \tilde{W}_n,\tilde{u}_{0,n})  
	\ton (\tilde{u}_s,\tilde{u}_w,\tilde{W},\tilde{u}_0) 
	\quad \text{in $\mathcal{X}_m$,
	\,\, $\mathbb{P}$-a.s.},
\end{align*} 
along a subsequence that is not relabelled. 
From Lemma \ref{thm:W_BM}, $\tilde{W}$ is 
a Brownian motion on $\bigl(\tilde{\Omega},
\tilde{\mathcal{F}},\bigl\{\tilde{\mathcal{F}}_t
\bigr\}_{t\ge 0},\tilde{\prob}\bigr)$, 
where $\bigl\{\tilde{\mathcal{F}}_t\bigr\}_{t\ge 0}$ 
is the canonical filtration defined by
$$
\tilde{\mathcal{F}}_t 
:={\it \Sigma}\bk{{\it \Sigma}
\bigl(u\Big|_{[0,t]}, W\Big|_{[0,t]}\bigr) 
\cup \seq{N \in \tilde{\mathcal{F}}:
\tilde{\prob}(N) = 0}}.
$$

As in Lemma \ref{thm:identification_lemma}, 
we can identify $\tilde{u}_n := \tilde{u}_{n,s} = \tilde{u}_{n,w}$, 
and $\tilde{u} := \tilde{u}_s = \tilde{u}_w$, 
$\tilde{\mathbb{P}} \otimes \d t \otimes \d x$-a.e.

Following the proof of Theorem \ref{thm:existence_H1}, 
the Galerkin equation \eqref{eq:galerkin_n} 
holds in the PDE weak sense using 
the equivalence of laws, for the variables 
$\bigl(\tilde{u}_n, \tilde{W}_n,\tilde{u}_{0,n}\bigr)$ 
in place of $(u_n, W,\bm{\Pi}_n u_0)$, 
$\tilde{\mathbb{P}}$-almost surely, 
up to any $t \in [0,T]$. Using the 
$\tilde{\mathbb{P}}$-almost everywhere 
convergence of  $\bigl(\tilde{u}_n, \tilde{W}_n, 
\tilde{u}_{0,n}\bigr)$ in the 
joint path space $\mathcal{X}^m$, as in the 
proof of Theorem \ref{thm:existence_H1}, 
we can extract the limiting equation for 
$\bigl(\tilde{u}, \tilde{W}, \tilde{u}_0\bigr)$, 
thereby establishing the existence of 
a weak (martingale) $H^m$ solution 
in the $n \to \infty$ limit.

Strong temporal continuity in $H^1(\T)$ can
be established exactly as in Section \ref{sec:temporalcont}, 
and stochastic boundedness in $L^2([0,T];H^{m + 1}(\T))$ 
follows from equality of laws and Lemma \ref{thm:stoch_bound_Hm_v2}
because, by the Lusin--Souslin theorem, $L^2([0,T];H^{m + 1}(\T))$ injects
continuously into $L^2([0,T];H^m(\T))$ and 
hence is Borel in the bigger space (see Part 6 of the proof of 
Theorem \ref{thm:existence_H1}).
\end{proof}

\subsection{Pathwise uniqueness and strong $H^m$ solutions}
\label{sec:st_sol_Hm}

In this section, we briefly conclude with 
pathwise uniqueness in $H^m$. 

\begin{thm}[Pathwise uniqueness in $H^m$]
\label{thm:st_uniqueness_Hm} 
Let $u$, $v$ be strong $H^m$ solutions to 
the viscous stochastic CH equation \eqref{eq:u_ch_ep}, 
with $\sigma \in W^{m + 1,\infty}(\T)$ 
and initial condition $u|_{t=0}=v|_{t=0}=u_0 
\in L^8(\Omega;H^{m}(\T))$. Then 
$$
\Ex \norm{u-v}_{L^\infty([0,T];H^m(\T))}=0.
$$
\end{thm}

\begin{proof}
Having established that 
$\Ex \norm{u-v}_{L^\infty([0,T];H^1(\T))}=0$ 
in Theorem \ref{thm:st_uniqueness}, 
we conclude that $u=v$, 
$\mathbb{P}\otimes \d t \otimes \d x$-a.e. 
Then necessarily, 
$\Ex \norm{u-v}_{L^\infty([0,T];H^m(\T))}=0$ 
also. This is uniqueness in 
$L^1(\Omega;L^\infty([0,T];H^m(\T)))$.
\end{proof}

With the same argument that was employed in 
Section \ref{sec:H1_existence}, we can now conclude 
that the second main theorem of the paper
(Theorem \ref{thm:main2}), holds.

\section*{Acknowledgement}

We thank the anonymous reviewers whose comments helped 
improve and clarify this manuscript.

\appendix 

\section{Stochastic toolbox}
\label{sec:toolbox}

In this section, we recall some notations 
and results from stochastic analysis that are 
used throughout the paper. We use \cite{Chow:2015,LR2015,Revuz:1999wi} 
as general references on stochastic analysis and SPDEs. 
Recall that $\bigl(\Omega,\mathcal{F},
\mathbb{P}\bigr)$ is a complete probability space 
with a countably generated $\sigma$-algebra 
$\mathcal{F}$ and probability measure $\mathbb{P}$. 
Let $\Bbb{B}$ be a separable Banach space, equipped 
with the Borel $\sigma$-algebra $\mathcal{B}(\Bbb{B})$. 
A $\Bbb{B}$-valued random variable $v$ is 
a measurable mapping from $(\Omega,\mathcal{F},\prob)$ to 
$\bigl(\Bbb{B},\mathcal{B}(\Bbb{B})\bigr)$, 
$\omega\mapsto v(\omega)$. The expectation of $v$ 
is $\Ex\, v:=\int_{\Omega} v\, d\prob$. 
We often use the abbreviation a.s.~or almost surely 
to mean for $\prob$-almost every $\omega\in \Omega$. 
The collection of $\Bbb{B}$-valued random variables $v$
for which $\Ex \abs{v}<\infty$ is denoted 
by $L^1(\Omega)=L^1(\Omega,\mathcal{F},\prob)$. 
This is a Banach space with norm 
$\norm{v}_{L^1(\Omega)}=\Ex\norm{v}_{\Bbb{B}}$. 
For $p>1$, $L^p(\Omega)$ is defined similarly, 
with $\norm{v}_{L^p(\Omega)}$ given by 
$\left(\Ex \norm{v}_{\Bbb{B}}^p\right)^{1/p}$ if $p<\infty$ 
and $\operatorname{ess\, sup}_{\omega\in \Omega} 
\norm{v(\omega)}_{\Bbb{B}}$ if $p=\infty$.

A stochastic process $v=\seq{v(t)}_{t\in [0,T]}$, for $T>0$, 
is a collection of $\Bbb{B}$-valued random variables $v(t)$.  
We say that $v$ is measurable if $v$ is jointly 
measurable from $\mathcal{F}\times \mathcal{B}([0,T])$ 
to $\mathcal{B}(\Bbb{B})$. Recall that we consider filtrations 
$\{\mathcal{F}_t\}_{t\in[0,T]}$ that satisfy 
the ``usual conditions" of being complete 
and right-continuous, and we refer 
to $\mathcal{S}:=\bigl(\Omega,\mathcal{F},
\{\mathcal{F}_t\}_{t\ge 0},\mathbb{P}\bigr)$, 
see \eqref{eq:stoch-basis}, as a stochastic basis. 
A stochastic process $v$ is \textit{adapted} if $v(t)$ 
is $\mathcal{F}_t$ measurable for all $t\in [0,T]$. 
When a filtration is involved there are additional notions of 
measurability (predictable, optional, progressive) 
that are more convenient to work with. Here we use the 
(stronger) notion of a predictable process. 
A \textit{predictable} process $v$ is a $\mathcal{P}_T\times 
\mathcal{B}([0,T])$-measurable map 
$\Omega\times [0,T]\to \Bbb{B}$, $(\omega,t)
\mapsto v(\omega,t)$, where $\mathcal{P}_T$ denotes the 
predictable $\sigma$-algebra on $\Omega\times [0,T]$ 
associated with $\seq{\mathcal{F}_t}_{t\in [0,T]}$ 
(the $\sigma$-algebra generated by all left-continuous 
adapted processes). A predictable process is adapted. 
Although the converse is not true, adaptive 
processes with regular (e.g.,~continuous) paths 
are predictable. To check for continuity, one uses 
the \textit{Kolmogorov test} \cite[p.~7]{Chow:2015}: suppose 
there are constants $\kappa>1$, $\delta>0$, and $K>0$ such that
$$
\Ex\norm{v(t)-v(s)}_{\Bbb{B}}^{\kappa}
\leq K \abs{t-s}^{1+\delta},
\quad \forall  s,t \in [0,T],
$$
then there exists a continuous modification of $v$, still 
denoted by $v$, such that $\Ex 
\norm{v}_{C^\gamma([0,T];\Bbb{B})}^\kappa\le K$, 
where the constant $K$ is independent of $v$ and 
$\gamma \in \left[0,\frac{\delta}{\kappa}\right)$.

\medskip

Throughout the work, we repeatedly end up with 
SDE inequalities of the form  
$\d \xi \le \eta \,\d t + L\xi\,\d t + \d M$, 
for some quantity of interest $\xi=\xi(\omega,t)$ and 
a zero-mean martingale $M$. For us $L\ge 0$ is often 
a stochastic process, so that the standard 
(deterministic) Gronwall inequality cannot be applied. 
The following stochastic Gronwall inequality 
is taken from \cite[Lemma~3.8]{XZ2017}, which 
is a version of a result proved first 
in \cite[Thm.~4]{Sch2013}. The term $L\xi\,\d t$ can 
be written as $\xi \d\int_0^t L(s)\,\d s=\xi\, \d A(t)$, which 
is the form used in the lemma. Besides, the 
inequality provides a bound on the $\nu$th moment 
of $\xi$ that does not depend on the martingale 
term $M$. It is this ``martingale 
uniformity" that forces the non-standard 
condition $\nu\in (0,1)$.

\begin{lem}[Stochastic Gronwall inequality]
\label{thm:stochastic_gronwall}
Relative to the stochastic basis $\mathcal{S}$, see 
\eqref{eq:stoch-basis}, let $\xi(t)$ and $\eta(t)$ be 
two non-negative adapted processes, $A(t)$ 
be an adapted non-decreasing process with $A(0)=0$, 
and $M$ a local martingale with $M(0) = 0$. 
Suppose $\xi$ is c{\`a}dl{\`a}g in time and satisfies 
the following SDE inequality on $[0,T]$:
$$
\d \xi \le \eta \,\d t + \xi\,\d A + \d M.
$$
For $0 < \nu < r < 1$ and $t \in [0,T]$, we have
\begin{align*}
	\Big(\Ex \sup_{s \in [0,t]} 
	\abs{\xi(s)}^\nu\Big)^{1/\nu} 
	\le \Big(\frac{r}{r-\nu}\Big)^{1/\nu} 
	\bk{\Ex \exp\left(\frac{rA(t)}{1-r}
	\right)}^{(1-r)/r}
	\Ex \Big(\xi(0) + \int_0^t \eta(s) \,\d s\Big).
\end{align*}
\end{lem}

This lemma can be formulated 
for stopping times $\tau$ in place of $t$.
For suppose $\xi$, $\eta$, $A$, and $M$ 
are as in Lemma \ref{thm:stochastic_gronwall},
then for any stopping time $\tau$,
$$
\d \xi(t\wedge \tau) 
\le \eta(t \wedge \tau)\,\d (t \wedge \tau) 
	+ \xi(t \wedge \tau)\,\d A(t \wedge \tau) 
	+ \d M(t \wedge \tau).
$$
Since $\tau$ is a stopping time, $M(t \wedge \tau)$ 
remains a local martingale 
(see \cite[Cor.~II.3.6, Def.~IV.1.5]{Revuz:1999wi}), 
and using the elementary equality
$$
\sup_{s \in [0, T]} \abs{\xi(s \wedge \tau)} 
= \sup_{s \in [0, T \wedge \tau]} \abs{\xi(s)},
$$
Lemma \ref{thm:stochastic_gronwall} is then seen to imply:
\begin{lem}\label{thm:stochastic_gronwall_st}
Let $\xi$, $\eta$, $A$ and $M$ be as in Lemma 
\ref{thm:stochastic_gronwall}. Let $\tau$ be a 
stopping time on the same filtration as $M$ is 
a martingale. For $0 < \nu <  r < 1$, we have
\begin{align*}
	\bigg(\Ex &\sup_{s \in [0, T \wedge \tau]} 
	\abs{\xi(s)}^\nu\bigg)^{1/\nu} \\
	&\le \bk{\frac{r}{r-\nu}}^{1/\nu} 
	\bigg(\Ex \exp\left(\frac{rA(T \wedge \tau)}{1-r} \right)\bigg)^{(1-r)/r}
	\Ex \bk{\xi(0) + \int_0^{T} \eta(s  \wedge \tau) \,\d s}.
\end{align*}
\end{lem}

\medskip

Next, we use on a few occasions 
the following convergence result for stochastic 
integrals, which is due to Debussche, Glatt-Holtz, and Temam, 
see \cite[Lemma 2.1]{DGT2011}. 

\begin{lem}[Convergence of stochastic integrals]
\label{lem:stoch-conv}
Fix a probability space $(\Omega,\mathcal{F},\prob)$. 
For each $n\in \N$, consider a stochastic basis 
$\mathcal{S}_n=\bigl(\Omega,\mathcal{F},
\{\mathcal{F}_t^n\}_{t\in [0,T]},\prob\bigr)$, 
a Wiener process $W^n$ on $\mathcal{S}_n$, and 
a predictable $L^2(\T)$-valued process $G^n$ 
on $\mathcal{S}_n$ satisfying $G^n\in L^2([0,T];L^2(\T))$, 
$\prob$-almost surely. Suppose there is a stochastic basis 
$\mathcal{S}=\bigl(\Omega,\mathcal{F},
\{\mathcal{F}_t\}_{t\in [0,T]},\prob\bigr)$, a 
Wiener process $W$ on $\mathcal{S}$, and a predictable 
$L^2(\T)$-valued process $G$ on $\mathcal{S}$ with 
$G\in L^2((0,T);L^2(\T))$ $\prob$-almost surely, such that
\begin{align*}
	& \text{$W^n \ton W$ in $C([0,T])$},
	\quad 
	\text{$G^n \ton G$ in $L^2([0,T];L^2(\T))$},
	\quad \text{in probability}.
\end{align*} 
Then
$$
\int_0^t G^n \, \d W^n \ton \int_0^t G \, \d W 
\quad \text{in $L^2([0,T];L^2(\T))$, in probability}.
$$
\end{lem}

\medskip

A sequence $\seq{v_n}$ of $\Bbb{B}$-valued random 
variables is \textit{stochastically bounded} (in $\Bbb{B})$ 
if 
\begin{align}\label{eq:stochbounded_defin}
\prob\bigl(\norm{v_n}_{\Bbb{B}}>M\bigr)\to 0, \mbox{ as } M\to \infty, \mbox{ uniformly in }n,
\end{align}
here written $v_n\in_{\rm sb} \Bbb{B}$. A simple 
approach for proving stochastic boundedness 
is---via Markov's (or Chebychev's) inequality---to bound 
$\norm{v_n}_{\Bbb{B}}$ in $L^p(\Omega)$, uniformly 
in $n$.  Denote by $\mu_n:=(v_n)_*\prob$ the probability 
law of $v_n$, i.e., for any 
$A\in \mathcal{B}(\Bbb{B})$, $\mu_n(A)=(v_n)_*\prob(A):=
\prob\bigl(X_n\in A\bigr)$. Stochastic boundedness 
is equivalent to the requirement that
$\mu_n\bigl(\seq{v\in \Bbb{B}:\norm{v}_{\Bbb{B}}>M}
\bigr)\to 0$ as $M\to\infty$, uniformly in $n$. If $\Bbb{B}$ 
is finite dimensional, this condition is that of 
tightness of the probability laws $\seq{\mu_n}$. 
If $\Bbb{B}$ is infinite dimensional, or more generally 
for a topological space $\bigl(\mathcal{X},
\mathcal{B}(\Bbb{\mathcal{X}})\bigr)$, 
by \textit{tightness} of a sequence of (Borel) 
probability measures $\seq{\mu_n}$ on $\mathcal{X}$, we 
mean that for any $\delta>0$, there is a 
compact set $K_\delta \subset \mathcal{X}$ 
such that $\mu_n\bigl(\mathcal{X} 
\backslash K_\delta\bigr)<\delta$, uniformly in $n$. 
The identification of a suitable compact set relies 
on Aubin--Lions--Simon type embedding 
theorems, see for example \cite{Simon:1987vn}. 
In a separable metric (or even a Hausdorff) 
space $\mathcal{X}$, by the well-known 
\textit{Prokhorov theorem}, tightness of the laws 
$\seq{\mu_n}$ implies weak compactness of $\seq{\mu_n}$, 
where we recall that $\seq{\mu_n}$ is 
\textit{weakly (or narrowly) convergent} to $\mu$ if 
$\int_{\mathcal{X}} f \, d\mu_n \ton \int_{\mathcal{X}} 
f \, d\mu$, for all $f\in C_b(\mathcal{X})$, 
the set of bounded continuous functions. 
If $\mathcal{X}$ is a Polish space, i.e., a 
separable completely metrisable topological space, then 
weak compactness implies tightness.

Finally, we will need the Gy{\"o}ngy-Krylov characterization 
of convergence in probability \cite{GK1996}. 
It will be used to upgrade weak (martingale) solutions 
to pathwise solutions.

\begin{lem}[Gy\"ongy--Krylov]\label{thm:gyongy_krylov}
Let $\mathcal{X}$ be a Polish space. For a sequence
$\seq{v_n}$ of $\mathcal{X}$-valued random variables 
define the joint probability laws $\seq{\mu^{m,n}}_{m,n}$ 
by setting, for all $A\in \mathcal{B}(\mathcal{X}
\times \mathcal{X})$, $\mu^{m,n}(A):=\prob\bigl(
\seq{(v_m,v_n)\in A}\bigr)$. Then the sequence $\seq{v_n}$ 
converges in probability if and only if for every subsequence 
$\seq{\mu^{m_k,n_k}}_{k}$, there exists a further 
subsequence that converges weakly to a 
probability measure $\mu$ supported on the diagonal:
$\mu\bigl( \seq{(v,w) \in \mathcal{X} 
\times \mathcal{X}: v = w}\bigr) = 1$.
\end{lem}

The fact that the support of the limit of the 
joint laws $\mu^{m,n}$ in Lemma \ref{thm:gyongy_krylov} 
lies on the diagonal follows from a 
pathwise uniqueness property, that is, 
for two solutions $v_a$ and $v_b$ of the same SPDE 
sharing the same initial condition, one has
$$
\mathbb{P}\bigl(\left\{\omega \in \Omega: 
\norm{v_a(\omega,t)-v_b(\omega,t)}_{\mathcal{X}}=0, \,\,
\forall t \in [0,T]\right\}\bigr) = 1.
$$
We point out that pathwise uniqueness also implies 
uniqueness in law \cite[Theorem IX.1.7]{Revuz:1999wi}, 
i.e., that for two weak solutions $\bigl(v_a, W_a,\mathcal{S}_a\bigr)$ 
and $\bigl(v_b, W_b, \mathcal{S}_b\bigr)$, with their 
respective Brownian motions $W_a,W_b$ and stochastic bases 
$\mathcal{S}_a, \mathcal{S}_b$, one has that 
the laws of $v_a$ and $v_b$ coincide, i.e., $v_a \sim v_b$. 

\vspace{0.25cm}

Generally, to ensure convergence of a sequence of
approximate solutions towards a solution for a nonlinear 
SPDE, it is essential that we secure strong compactness 
in the $\omega$ variable (a.s.~convergence). 
To that end, one often relies on the 
Skorokhod representation theorem for random variables 
taking values in a Polish space $\mathcal{X}$, delivering a 
new probability space and new random variables, 
with the same laws as the original ones, 
converging almost surely.  In this work, we use 
the spaces $L^2([0,T];H^1(\T))$ and $C([0,T];H^1_w(\T))$. 
The former is a Polish space, whereas the latter is not.  
Here $C([0,T];H^1_w(\T))$ refers to the continuous 
functions from $[0,T]$ to the Hilbert space $H^1(\T)$ 
equipped with the \textit{weak topology}. 
This is a locally convex space with the weak topology 
generated by the system of seminorms $\norm{v}_\phi
=\sup_{t\in[0,T]} \abs{\bigl \LL v(t),\phi\bigr\RR_X}$, 
for $\phi \in X:= H^1(\T)$.  Since $X$ is separable and 
reflexive, the unit ball $B_X \subset X$ is 
a metrisable compact set and one can equip $C([0,T];B_X)$ 
with a complete metric topology induced by the above system of 
seminorms. On $C([0,T];H^1_w(\T))$ we consider the 
$\sigma$-algebra $\mathcal{B}_T$ generated by the mappings 
$C([0,T];H^1_w(\T))\ni v\mapsto v(t)\in X$, $t\in [0,T]$. 

Weakly continuous functions taking values in 
a separable Banach space are not Polish but rather 
quasi-Polish. \textit{Quasi-Polish} refers to a 
topological space $\bigl(\mathcal{X},\tau\bigr)$ 
that asks for point-separability by countably 
many continuous functions, i.e., that there exists a 
countable family 
\begin{align}\label{eq:countable}
\seq{f_\ell:
\mathcal{X}\to [-1,1]}_{\ell\in \N}
\end{align}
of continuous functions that separate points of 
$\mathcal{X}$ \cite{Jak1998}.  
In other words, $\mathcal{X}$ is quasi-Polish if 
$\mathcal{X}$ is a Hausdorff space (but need not be regular) 
that admits a continuous injection 
$f(v)=\seq{f_\ell(v)}_{\ell\in \N}$ to the Polish 
space $[-1,1]^{\N}$. The idea behind the proof of 
the theorem below \cite{Jak1998} is to transfer the Skorokhod 
representation problem via homeomorphism 
methods to a compact subset of $[-1,1]^{\N}$, where the 
Skorokhod representation theorem is known to hold, and 
then map back to $\mathcal{X}$ via $f^{-1}$, noting that 
every compact set in $\mathcal{X}$ is 
$\sigma\bigl(\seq{f_\ell}\bigr)$-measurable and metriseable. 
Whenever the $\sigma$-algebra $\sigma\bigl(\seq{f_\ell}\bigr)$ 
is strictly smaller than the Borel $\sigma$-algebra 
$\mathcal{B}_\tau$, it turns out that every 
\textit{tight} Borel probability measure on 
$\bigl(\mathcal{X},\tau\bigr)$ is uniquely determined 
by its values on $\sigma\bigl(\seq{f_\ell}\bigr)$ and 
can be uniquely extended to $\mathcal{B}_\tau$. Besides, 
$f$ has a continuous inverse ($f$ is a homeomorphic 
embedding) when restricted to a $\tau$-compact subset 
of $\mathcal{X}$. As in 
\cite[Cor.~3.12]{Brzezniak:2013aa} (see 
also \cite{Brzezniak:2011aa,Brzezniak:2013ab,Ond2010}), 
one can easily prove that $C([0,T];H^1_w(\T))$ is quasi-Polish, 
and that the separating sequence 
$\seq{f_\ell}_{\ell\in \N}$ generates the $\sigma$-algebra 
$\mathcal{B}_T$. We refer to \cite[Sec.~3]{Brzezniak:2016wz} 
for a discussion that collects many relevant properties of 
quasi-Polish spaces, including $C([0,T];X_w)$ for an 
arbitrary separable Hilbert space $X$.

As the original Skorokhod theorem is not applicable in 
quasi-Polish spaces, we use the more recent version by Jakubowski 
\cite{Jak1998}. The following form of the theorem 
is taken from \cite{Brzezniak:2013aa,Brzezniak:2011aa,
Brzezniak:2013ab,Ond2010}, which are some of the 
first works to employ the theorem to construct 
martingale solutions of nonlinear SPDEs, including stochastic 
nonlinear wave equations and the stochastic 
incompressible Navier--Stokes equations, see also \cite{BH2016} 
for an application to the compressible 
Navier--Stokes equations. 

\begin{thm}[Skorokhod--Jakubowski a.s.~representations]
\label{thm:skorokhod}
Let $\bigl(\mathcal{X},\tau,\mathcal{B}_\tau\bigr)$ 
be a quasi-Polish space, and denote by 
$\Sigma_f \subset \mathcal{B}_\tau $ 
the $\sigma$-algebra generated by the sequence 
$\seq{f_\ell}$ of continuous functions that 
separate points. Then
\begin{enumerate}
	\item every $\tau$-compact subset of $\mathcal{X}$ 
	is metrisable;
	\item every Borel subset of a sigma compact set 
	in $\mathcal{X}$ belongs to $\Sigma_f$;
	\item every probability measure supported by 
	a sigma compact set in $\mathcal{X}$ has 
	a unique Radon extension to the Borel 
	$\sigma$-algebra $\mathcal{B}_\tau=
	\mathcal{B}(\mathcal{X})$.
\end{enumerate}
Moreover, if $\seq{\mu_n}$ is a tight sequence of probability 
measures on $\bigl(\mathcal{X},\Sigma_f\bigr)$, then 
there exist a subsequence $\seq{n_k}_k$, a 
probability space $\bigl(\tilde{\Omega},
\tilde{\mathcal{F}},\tilde{\prob}\bigr)$, and  Borel 
measurable $\mathcal{X}$-valued random variables 
$\tilde{v}_k$, $\tilde{v}$, such that $\mu_{n_k}$ 
is the law of $\tilde{v}_k$ and $\tilde{v}_k\to \tilde{v}$ 
$\tilde{\prob}$-a.s.~in $\mathcal{X}$. Besides, the 
law $\mu$ of $\tilde{v}$ is a Radon measure on 
$\mathcal{B}_\tau$.
\end{thm}

\begin{proof}
See \cite[pp.~169--173]{Jak1998}.
\end{proof}

A path space for a sequence of variables $\{v_n\}$ defines 
the topology in which we would like the
Skorokhod--Jakubowski representations $\{\tilde{v}_{k}\}$ 
to converge (a.s.). It is often important that 
$\tilde{v}_{k} \to \tilde{v}$ in multiple 
spaces/topologies (say, $\mathcal{X}_1$ and $\mathcal{X}_2$). 
When both spaces $\mathcal{X}_1$ and $\mathcal{X}_2$ 
are normed spaces, or when one space injects continuously 
into another, it is often possible to set up a topology on 
the intersection space $\mathcal{Y}:=\mathcal{X}_1 
\cap \mathcal{X}_2$ directly that meet two criteria:
\begin{itemize}
	\item[(i)] $\mathcal{Y} $ is quasi-Polish.
	\item[(ii)] Compact sets on $\mathcal{Y} $ 
	are sufficiently plentiful; in particular, tightness of laws 
	on $\mathcal{Y}$ can be readily deduced by the separate 
	tightness on $\mathcal{X}_1$ and on $\mathcal{X}_2$. 
\end{itemize}
These criteria are opposed in the sense that 
a topology on the intersection space $\mathcal{Y}$ 
stronger than (the subspace topology induced by) 
each of the topologies on $\mathcal{X}_1$ and $\mathcal{X}_2$ 
makes it easy to show that $\mathcal{Y}$ is quasi-Polish. 
One such example is the supremum topology.  
On the other hand, the strength 
of the topology placed on $\mathcal{Y} $ makes convergence 
there more difficult and compact sets harder to come by. 

Herein, the difficulty of characterising 
compact sets on any sufficiently strong topology 
on the intersection space $\mathcal{Y}$ is side-stepped 
by finding a.s.~representations and limits for $\{(v_n, v_n)\}$ 
on the product space $\mathcal{X}_1 \times \mathcal{X}_2$, 
and after that identifying their limits as the same 
process (Lemma \ref{thm:identification_lemma}).
 
(Countable) products of quasi-Polish 
spaces are quasi-Polish. We shall 
apply Theorem \ref{thm:skorokhod} in the product space  
$\mathcal{X} = \mathcal{X}_1 \times \mathcal{X}_2 
\times  \mathcal{X}_3 \times \mathcal{X}_4$,
where $\mathcal{X}_1$ and $\mathcal{X}_2$ are 
two path spaces for two copies of the same variable. 
A Cartesian product of topological spaces is always 
equipped with the product topology and, thus, the 
Borel $\sigma$-algebra generated by the product topology.

On a product space there are two natural 
$\sigma$-algebras: the product of the Borel $\sigma$-algebras 
and the already introduced Borel $\sigma$-algebra for 
the product topology. Although, in general, these two 
are not the same, they do coincide on a separable metric space. 
This implies that coordinatewise 
measurability and tightness is the same as joint 
measurability and tightness, which is convenient since 
we would want to use the product of the 
Borel $\sigma$-algebras in computations leading up 
to joint tightness and weak convergence 
in the product space. Whilst the setting of 
Theorem \ref{thm:skorokhod} goes far beyond separable 
metric spaces, in applications a priori estimates 
ensure that the involved random variables take values 
in a compact set, and then we can rely on (1) and (2) 
of Theorem \ref{thm:skorokhod}.

\bibliographystyle{plain}

\end{document}